\documentclass[10pt,reqno]{amsart}
\usepackage[numbers]{natbib}
\usepackage{amsaddr}
\usepackage{etoolbox}
\patchcmd{\section}{\scshape}{\bfseries\scshape}{}{}
\makeatletter
\renewcommand{\@secnumfont}{\bfseries}
\makeatother

\makeatletter
\renewcommand\subsection{\@startsection{subsection}{2}%
  \z@{-.5\linespacing\@plus-.0\linespacing}{0.3\linespacing}%
  {\bfseries}}
\makeatother

\pdfoutput=1 %
\usepackage[a4paper,margin=3cm]{geometry}

\usepackage[utf8]{inputenc} %
\usepackage[T1]{fontenc}    %

\usepackage[final]{hyperref}      %
\hypersetup{
    colorlinks=true,
    citecolor=green,
    filecolor=black,
    linkcolor=blue,
    urlcolor=blue
}

\usepackage{lipsum}
\usepackage{amsfonts}
\usepackage[dvipsnames]{xcolor}
\usepackage{graphicx}
\usepackage{epstopdf}
\usepackage{algorithmic}
\ifpdf
  \DeclareGraphicsExtensions{.eps,.pdf,.png,.jpg}
\else
  \DeclareGraphicsExtensions{.eps}
\fi

\usepackage{amsopn}

\usepackage{amssymb,amsmath}
\usepackage{dsfont}
\usepackage{url}
\usepackage{hyperref}

\usepackage{tikz}

\newcommand{\indic}{\mathds{1}}
\newcommand{\e}{\varepsilon}

\newcommand{\Ee}{\mathbb{E}}
\newcommand{\Rr}{\mathbb{R}}
\newcommand{\Pp}{\mathbb{P}}

\newcommand{\Zz}{\mathbb{Z}}

\newcommand{\cC}{\mathcal{C}}

\newcommand{\cH}{\mathcal{H}}

\newcommand{\cB}{\mathcal{B}}

\newcommand{\cX}{\mathcal{X}}
\newcommand{\cN}{\mathcal{N}}

\newcommand{\cM}{\mathcal{M}}
\newcommand{\cP}{\mathcal{P}}

\newcommand{\cU}{\mathcal{U}}
\newcommand{\cL}{\mathcal{L}}
\newcommand{\cE}{\mathcal{E}}

\newcommand{\ubar}[1]{\text{\b{$#1$}}}
\allowdisplaybreaks %

\newcommand{\BEAS}{\begin{eqnarray*}}
\newcommand{\EEAS}{\end{eqnarray*}}
\newcommand{\BEA}{\begin{eqnarray}}
\newcommand{\EEA}{\end{eqnarray}}
\newcommand{\BEQ}{\begin{equation}}
\newcommand{\EEQ}{\end{equation}}
\newcommand{\BIT}{\begin{itemize}}
\newcommand{\EIT}{\end{itemize}}
\newcommand{\BNUM}{\begin{enumerate}}
\newcommand{\ENUM}{\end{enumerate}}
\newcommand{\BA}{\begin{array}}
\newcommand{\EA}{\end{array}}

\newcommand{\argmin}{\mathop{\rm argmin}}

\newcommand{\tr}{\mathop{ \rm tr}}

\newcommand{\rb}{\mathbb{{R}}}

\newcommand{\ds}{\displaystyle }

 \def \ds { \displaystyle}

\def \X{{\mathcal X}}

\newtheorem{lemma}{Lemma}
\newtheorem{theorem}{Theorem}
\newtheorem{proposition}{Proposition}
\newtheorem{corollary}{Corollary}
\newtheorem{definition}{Definition}
\newtheorem{remark}{Remark}

\makeatletter%
\@mparswitchfalse%
\makeatother%
\normalmarginpar%

\usepackage[shortlabels]{enumitem}

\usepackage{tikz}
\usetikzlibrary{automata, positioning, arrows, calc}

\tikzset{
	>=stealth, %
	shorten >=2pt, shorten <=2pt, %
	node distance=4cm, %
	initial text=$ $, %
 }

\title[A stochastic optimization approach to control-affine optimal control problems]{A stochastic optimization approach to control-affine optimal control problems}

\date{\today}

\makeatletter
\g@addto@macro{\endabstract}{\@setabstract}
\newcommand{\authorfootnotes}{\renewcommand\thefootnote{\@fnsymbol\c@footnote}}%
\makeatother

\begin{document}

\begin{center}
	\LARGE 
	A Stochastic Optimization Approach to\\ Control-Affine Optimal Control Problems \par \bigskip
	
	\normalsize
	Eloïse Berthier\textsuperscript{1}, Ziad Kobeissi\textsuperscript{2} and
	Francis Bach\textsuperscript{3} \par \bigskip
	
	\textsuperscript{1}U2IS, ENSTA, Institut Polytechnique de Paris, Palaiseau, France
	\smallskip \par
	\textsuperscript{2}L2S, Inria, Université Paris-Saclay, CentraleSupelec, Gif-sur-Yvette, France\smallskip \par 
    \textsuperscript{3}Inria,
École Normale Supérieure, PSL Research University, Paris, France\smallskip \par 
\end{center}

\begin{abstract}
    We consider control-affine optimal control problems on the torus, where the dynamics and cost functions are only accessed through samples. Starting from a weak formulation of such problems, we derive a dual, a primal, and a primal-dual formulation,  compatible with stochastic optimization.  We show convergence of stochastic first-order methods to the optimal value  under generic conditions. In addition, we introduce a computable metric that upper-bounds the performance of suboptimal controllers produced during optimization, under additional regularity assumptions. Preliminary results show that the method can efficiently solve a simple control problem. Finally, we discuss conditions for the boundedness of the optimal occupation measure, a key assumption for the primal and primal-dual approaches. 
\end{abstract}

\section{Introduction}

Nonlinear optimal control problems (OCP)~\cite{vinter2010optimal} are generic modeling tools used in a variety of applications such as automobile, aerospace or robotics. In general, the numerical resolution of such problems, e.g., using direct or indirect methods \cite{rao2009survey}, is computationally costly, especially as the dimensions of the problem grow large. Furthermore, to employ such methods, a model---the dynamics of the system and the cost functions---must be known in advance. Approaches extending numerical methods for OCP to unknown models observed through samples are gathered under the name of data-driven or sample-based optimal control~\cite{prag2022toward}. Yet, theoretical analyses of such methods are often limited to the linear quadratic regulator (LQR) problem~\cite{dean2020sample}. 

Reinforcement learning (RL)~\cite{Sutton1998} is a subfield of machine learning dedicated to solving optimal control problems whose model is initially unknown, but progressively observed through state transitions and rewards (or negative costs), as the system interacts with its environment. Given this definition, the frontier with data-driven optimal control is often blurred~\cite{jiang2020learning, bensoussan2022machine}. RL is commonly used in applications involving simulated environments, in video games, but also increasingly in robotics \cite{kober2013reinforcement}. Such applications often involve human interactions, stochasticity, or other unmodeled phenomena. Stochastic approximation is a well-adapted approach to tackle this specificity of RL, where the model is only accessed through samples. This includes temporal-difference based algorithms, and stochastic optimization algorithms~\cite{meyn2022control}. Yet, beyond tabular settings (\textit{i.e.}, the state space is finite and not too large), and simple policy evaluation algorithms such as TD-learning, most RL algorithms only benefit from limited theoretical guarantees. For example, the SARSA and Q-learning algorithms with linear function approximation only converge under restrictive assumptions~\cite{melo2008analysis}, while policy gradient~\cite{sutton1999policy} and actor-critic algorithms~\cite{konda1999actor} only converge to stationary points~\cite{dai2018sbeed}. In both cases, little guarantees on the performance of the obtained policies (or controller) are provided.

Contrary to RL, supervised learning benefits from a richer collection of theoretical analyses, from a thorough understanding of least-squares regression~\cite{dieuleveut2017stochastic}, to more recent results on the convergence of stochastic algorithms in the optimization of wide neural networks~\cite{chizat2019lazy}, and non-convex optimization problems~\cite{danilova2022recent}. In particular, the stochastic gradient descent (SGD) algorithm~\cite{bottou2018optimization} is an essential tool for large-scale training, as it allows to iteratively optimize a function expressed as an expectation---typically a loss function---using only gradients computed at individual samples.

In this paper, we explore the following direction: can we use stochastic optimization to solve OCPs, like in supervised learning, \textit{i.e.}, with theoretical guarantees? As a first step towards this goal, we study a deterministic, infinite-horizon, control-affine  OCP on the torus. Using the linear-programming formulation of OCP~\cite{vinter1978equivalence,lasserre2008nonlinear}, we investigate under which conditions the primal and dual versions can be considered as convex and concave stochastic optimization problems, on which SGD can be used. 

While the dual formulation, which optimizes over value functions using a non-parametric representation, is naturally compatible with SGD, the convergence guarantee is on the residuals of the Hamilton-Jacobi-Bellman (HJB) equation. This does not provide any straightforward indication on the quality of the controller. The primal formulation, which involves optimizing over occupation measures~\cite{bhatt1996occupation}, is linked to the policy evaluation problem and hence naturally provides a control on the performance of the obtained controller. However, a stochastic optimization version of the primal cannot be obtained for any control-affine dynamics, but only for dynamics that are linear in the controller.

We introduce an additional primal-dual version of the OCP, compatible with stochastic optimization for any control-affine dynamics, with similar guarantees on the performance of the controller using a duality gap metric. This convex-concave primal-dual formulation manipulates both a value function, and state and control densities, which appear in similar optimization problems in optimal transport~\cite{benamou2000computational} or in mean-field games~\cite{ruthotto2020machine}.  Various other  primal-dual formulations, some of which convex-concave, others not,  have been studied for discrete or linearly approximated RL problems~\cite{hernandez2012discrete,wang2016online, wang2017primal, chen2018scalable, dai2018boosting, nachum2020reinforcement, jin2020efficiently,  lee2019stochastic,tiapkin2022primal,li2024accelerating}, with recent works making LP-based methods for RL an active field~\cite{malek2014linear,lakshminarayanan2017linearly, bas2022lagragian, gabbianelli2024offline, wolter2025two,lee2025analysis}. %
Moreover, actor-critic methods for continuous control problems, based on iteratively optimizing on the value function and the policy, significantly differ from our primal-dual approach in the time-scale of the policy updates (see Section~\ref{subsec:PDcomp}).

Finally, stochastic optimization can be employed on the primal, dual and primal-dual formulations under a boundedness condition on the optimal state occupation measure. In most of this paper, we assume that this condition is fulfilled, and we study this boundedness problem, of separate interest, in Section~\ref{sec:mustar}.

The rest of this work is structured as follows. In Section~\ref{sec:def}, we present generic results controlling the magnitude of the optimal value function and controller, and state the main assumptions. We then present in Section~\ref{sec:lp} the primal and dual linear programming (LP) formulations of the OCP, and show a duality result, along with an  equivalence with the original OCP, under mild assumptions. In Section~\ref{sec:dual}, we provide an unconstrained dual formulation of the OCP, compatible with SGD, along with a convergence result. We also illustrate with a simple numerical example the limitations of such convergence guarantees. In Section~\ref{sec:primal}, we study the unconstrained primal formulation, and show that it is compatible with SGD for dynamics that are linear in the control, along with a convergence result. We link this primal problem with policy evaluation, hence providing a way to measure the quality of the obtained controller, and illustrate its interest on a simple example.  In Section~\ref{sec:primaldual}, we introduce a primal-dual problem, compatible with stochastic optimization for generic control-affine dynamics, with performance bound over the controller. In Section~\ref{sec:mustar}, we study  conditions under which the optimal state-occupation measure can be bounded, a required condition to apply the above approaches. Numerical examples of the proposed approaches are provided throughout the different sections.

\section{Definition of the problem and preliminary results} \label{sec:def}

Let $d$ and $m$ be positive integers, and $\rho >0$ a discount factor. We consider the following discounted, infinite-horizon, control-affine optimal control problem on the $d$-dimensional torus $\cX = [0, 1]^d$: 
\begin{align} \tag{OCP}\label{eq:OCP}
   & V^*(x) = \inf_{u:\Rr_+ \rightarrow \Rr^m}~~ \rho \int_0^{+\infty} e^{-\rho t} \left( f(x_t) + R(u_t) \right) \mathrm dt  \\
 \nonumber  &  \text{such that } ~ \forall t \geq 0, ~ \dot x_t = a(x_t) + B(x_t) u_t \, ~\text{ and }  x_0=x,
\end{align}
where $x : \Rr_+ \rightarrow \cX $ is a state trajectory, $u:\Rr_+ \rightarrow \Rr^m$ is a controller. The affine dynamics $b(x, u) = a(x) + B(x) u$ is defined by $a : \cX \rightarrow \Rr^d$, $B : \cX \rightarrow \Rr^{d \times m}$, and the running cost is composed by the sum of a non-negative state cost $f : \cX \rightarrow \Rr_+$, %
and a control cost defined by $R(u)=\|u\|^q/q$, for some $q\in (1, 2]$, where $\| \cdot \|$ denotes the Euclidean norm. %
We make the following assumptions on the dynamics and cost:
\begin{enumerate}[label={\bf A1}]
    \item
    \label{hypo:regularity}
    $f, a, B $ are 
uniformly Lipschitz continuous. %
 In particular, they are uniformly bounded on their compact domain $\cX$, and 1-periodic along each dimension.
\end{enumerate}
This assumption directly allows us to derive 
a uniform bound on~$V^*$.
\begin{proposition}
    \label{prop:V*_bounded}
     $V^*$ is non-negative and uniformly bounded by $\|f\|_\infty$.
\end{proposition}

\begin{proof}
$V^*$ being non-negative follows from the non-negativity of $f$ and $R$. Let ${x \in \cX}$. Consider the suboptimal controller $u \equiv 0$ and its corresponding value function~$V_0$. We have:
\begin{align*}
    V^*(x) \leq V_0(x) = \rho \int_0^{+\infty} e^{-\rho t} (f(x_t)+0) \, \mathrm dt  \leq  \rho \int_0^{+\infty} e^{-\rho t} \|f\|_\infty \, \mathrm dt
     = \|f\|_\infty \, .
\end{align*}
\end{proof}

\begin{remark}
    The control-affine assumption is standard in the nonlinear control literature, and many examples of physical control systems can be modeled by equations of this form~\cite{isidori1985nonlinear, sastry2013nonlinear}. Control-affine systems are a significant extension of linear systems, since they allow for arbitrary nonlinearities in the state variable. Conversely, the affine dependence in the input is often a consequence of physical laws. %
     For example, Newton's second law of motion for a rigid-body system can be written in the following canonical form:
\begin{align*}
    \tau = H(q) \ddot q + C(q, \dot q) \, ,
\end{align*}
where $q$, $\dot q$ and $\ddot q$ are vectors of position, velocity and acceleration
variables, respectively, $\tau$ is a vector of applied forces (torque), and $H$ and $C$ represent inertia, and Coriolis and centrifugal forces. It can be reformulated as:
\begin{align*}
     \ddot q = H(q)^{-1} \left( \tau - C(q, \dot q) \right) \, .
\end{align*}
 While this relation is nonlinear in $x=(q, \dot q)$, it is affine in $u=\tau$. For rigid bodies with several joints and contacts like robots, $H$ and $C$ are usually not known in closed-form.  To compute them numerically, modern robotics~\cite{featherstone2008rigid} relies on the 
articulated body dynamics (ABA) algorithm or on the composite rigid body algorithm (CRBA) with the recursive Newton-Euler algorithm (RNEA), for which  efficient numerical implementations are available~\cite{todorov2012mujoco,carpentier2019pinocchio}. Such ways to obtain the components of the dynamics are compatible with our stochastic approximation approach, which only requires evaluations at stochastic configurations $(q, \dot q)$. %
\end{remark}

\subsection{Optimal controller}

The optimal value function $V^*$ is the unique viscosity solution~\cite{crandall1983viscosity,bardi1997optimal} of the 
HJB equation 
\begin{align*}
 - V^\ast(x) + f(x) + \inf_{u \in \Rr^m}  \left\{ R(u) + \frac{1}{\rho} \nabla V^\ast(x) ^\top (a(x) + B(x) u) \right\} = 0 \, .
\end{align*}
In particular, this formulation does not presuppose the pointwise existence of $\nabla V^\ast$. 
Define $R^*$ as the convex conjugate (or Legendre transform) of $R$, given by $R^\ast(v)=\|v\|^{q'}/q'$, where $q'\in[2,+\infty)$ is the conjugate exponent of $q$, \textit{i.e.}, $1/q+1/q'=1$.
At any point $x$ where $V^*$ is differentiable, the infimum in $u$ is uniquely attained, by strict convexity of $R$, and the optimal control is
\begin{align*}u^*(x) = \nabla R^* \left( - \frac{1}{\rho} B(x)^\top\nabla V^*(x) \right),
\end{align*}
which depends on $\nabla V^*(x)$ only through $B(x)^\top\nabla V^*(x)$.

For this characterization to be of any use, $u^*$ must be
defined at least almost everywhere, that is,
$B^{\top}\nabla V^*$ must be defined a.e.
As $V^*$ is so far only known to be continuous,
this is not guaranteed,
and we enforce it through the following assumption.
\begin{enumerate}[label={\bf A2}]
\item
\label{hypo:alpha}
There exists $\alpha\in L^{\infty}(\cX;\Rr^m)$
such that $a(x)=B(x)\alpha(x)$, for all $x\in\cX$.
\end{enumerate}

This assumption is made only to fix ideas: it may be replaced 
by any condition guaranteeing that $u^*$ is defined
and uniformly bounded a.e. 
(as established in Proposition~\ref{prop:bound} below).
Note moreover that, under Assumption~\ref{hypo:alpha},
the HJB equation involves $\nabla V^*$
only through $B^\top\nabla V^*$, 
since $\nabla V^{*\top}a = (B^\top\nabla V^*)^\top\alpha$;
we therefore never need $\nabla V^*$ itself to be defined a.e.,
which would otherwise be required if $a(x)\notin{\rm Im} B(x)$.

\begin{proposition} \label{prop:bound}
    Assume \ref{hypo:regularity} and \ref{hypo:alpha},
    the optimal control $u^*$ satisfies
\begin{align*}
        \|u^*(x)\|
        \leq
        C_{u}
        {:=
        \left(2 q'\|f\|_{\infty}+2^q(q'-1)\|\alpha\|_{\infty}\right)^{\frac1{q}}
        \qquad\text{ for a.e. }x\in\cX.
        } 
    \end{align*} 
\end{proposition}
    The proof, along with all the proofs of subsequent results, is deferred to Appendix~\ref{app:proofs}.  
Since we have derived a uniform bound on $u^*$, we can  formulate an equivalent OCP where the control is taken in a compact set, which will significantly simplify the  derivation of technical results afterwards.  
\begin{corollary} \label{cor_constr}  Let $C_u$ the bound computed in Proposition~\ref{prop:bound}, we define the compact set $\cU = \{ u \in \Rr^m \mid \|u\| \leq C_u \} \subset \Rr^m$. Then, for any $x\in \cX$ the value of the OCP is equal to the value of the following control-constrained OCP:
    \begin{align*} 
    \inf_{u:[0, +\infty) \rightarrow \cU} ~~\rho \int_0^{+\infty} e^{-\rho t} \left( f(x_t) + R(u_t) \right) \mathrm dt , ~ \dot x_t = a(x_t) + B(x_t) u_t , x_0=x .
\end{align*}
\end{corollary}

\subsection{Occupation measure}
Let $\cC(\cX\times\Rr^m)$ be the space of real-valued
continuous functions from $\cX\times\Rr^m$ equipped with the topology
of local uniform convergence, which makes it a Frechet space.
Its topological dual is $\cM_c(\cX)$, 
the Banach space of compactly-supported signed measures.
Define $\cM_+(\cX\times\Rr^m)$ and $\cP_c(\cX\times\Rr^m)$
as the subspaces of $\cM_c(\cX\times\Rr^m)$
of non-negative measures and probability measures, respectively.
Similarly, we equip $\cC^1(\cX\times\Rr^m)$
with local $\cC^1$-convergence
and denote by $\cC^1(\cX\times\Rr^m)^*$ its topological dual.
When considering functional spaces on $\cX$
we drop the index $c$ since $\cX$ is compact, 
and the resulting spaces are Banach. Let $\cP(\cX)$ the space of probability measures on $\cX$, and $\cB(\cX)$ the $\sigma$-algebra of Borel sets.

Consider $\mu_0\in\cP(\cX)$ admitting
a full support and a  bounded density
with respect to the Lebesgue measure,
denoted by $\mu_0$ as well,  using a common abuse of notation. 
For $g : \cX \rightarrow \Rr^p$ Lipschitz continuous, 
following~\cite{gaitsgory2009linear},
we define the $\rho$-discounted occupation measure of the dynamical system $\dot x = g(x)$ as follows.
\begin{definition}%
    Consider a dynamical system $\dot x(t) = g(x(t))$, with $x(0) \sim \mu_0 \in \cP(\cX)$.  Let $B \in \cB(\cX)$ and $x_0 \in \cX$. The $\rho$-discounted conditional occupation measure of this dynamical system is defined as:
    \begin{align*} \nu_{x_0}(B) = \rho \int_0^{+\infty}  e^{-\rho t} \indic_B(x(t \mid x(0)=x_0)) \mathrm dt \, ,\end{align*}
    where $x(t \mid x(0)=x_0)$ is the position of $x$ at time $t$, given that the initial point $x(0)$ is set to $x_0$.
    The $\rho$-discounted occupation measure is defined as:
  \begin{align*} \nu(B) = \int_\cX \nu_{x_0}(B) \mathrm d\mu_0(x_0) \, .\end{align*}
\end{definition}

\begin{remark}
    \label{rk:exponential_rv}
    The measure $\nu$ can be interpreted as the expected distribution of $x(T)$ at a random stopping time $T \sim \textnormal{Exp}(\rho)$, and with random starting point ${X_0 \sim \mu_0}$:
\begin{align*}
    \nu(B) &= \Ee_{T \sim \text{Exp}(\rho), ~X_0 \sim \mu_0} \left[ \indic_B(x(T \mid x(0)=X_0)) \right] = \Pp_{T, X_0} \left[ x(T) \in B  \right] \, .
\end{align*}
\end{remark}

Such occupation measures are directly related to the dynamics $g$ and initial measure~$\mu_0$ through Liouville's equation~\cite{henrion2013optimization}, also referred to as a \textit{conservation of mass} principle~\cite{villani2009optimal}.
\begin{lemma} \label{lemma1}
    The $\rho$-discounted occupation measure $\nu$ with dynamics $\dot x = g(x)$ and initial state distribution $\mu_0$ is a solution, in the sense of distributions, %
   of Liouville's discounted equation: 
    \begin{align*}
         \nu = \mu_0 - \frac{1}{\rho} \textnormal{div}(g \nu) \, .
    \end{align*}
\end{lemma}
Applying this result to the optimal trajectories, with $g(x(t)) = b(x(t), u^*(x(t))$, we obtain that the $\rho$-discounted occupation measure $\mu^*$ under the optimal controller~$u^*$ verifies:
\begin{align*}
    \mu_0 = \mu^* + \frac{1}{\rho} \text{div}(b(\cdot, u^*(\cdot)) \mu^*) \, .
\end{align*}
Let us finally summarize the previously stated equations relating $V^*$, $u^*$ and $\mu^*$ in the following proposition.
\begin{proposition} \label{prop:optimality}
    The optimal controller $u^*$, the optimal value function $V^*$  and the corresponding occupation measure $\mu^*$ are solutions  of the following equations:
    \begin{equation*}
    \left\{
        \begin{aligned}
            &u^*(x) = \nabla R^*( - \rho^{-1} B(x)^\top\nabla V^*(x) ) 
            \\
            &- V^*(x) + f(x) + \rho^{-1} \nabla V^*(x) ^\top a(x)
            - R^*\big(  - \rho^{-1} B(x)^\top \nabla V^*(x)  \big)  = 0   
            \\
            &\mu_0(x) = \mu^*(x) + \rho^{-1} \textnormal{div} \left(b(x, u^*(x) ) \mu^*(x) \right),
        \end{aligned}
        \right.
    \end{equation*}
    where the first equality holds for a.e. $x\in\cX$,
    the second has to be understood in the sense of viscosity
    and the last one in the sense of distribution.
    \end{proposition}

\section{Primal and dual linear programming formulations} \label{sec:lp} In this section, we define a linear programming formulation of our optimal control problem on the torus with infinite horizon. While this formulation has been first proposed for finite-horizon problems~\cite{vinter1993convex,hernandez1996linear,lasserre2008nonlinear}, a discounted infinite-horizon formulation has been proposed by~\cite{gaitsgory2009linear}. This rich line of work proposes a \textit{weak formulation}\footnote{Vinter describes in~\cite[Section~1]{vinter1993convex} in which sense this formulation is ``weak''. It relies on \textit{relaxed arcs}, weakly associated with linear functionals, called \textit{generalized flows}, through integration against continuously differentiable test functions.} of optimal control. One of the challenges around this infinite-dimensional linear programming (LP) formulation is to find conditions under which it is equivalent to the original optimal control problem, along with strong duality results. It turns out that under reasonable convexity and compactness conditions, this equivalence indeed holds~\cite{fleming1989convex,vinter1993convex,lasserre2008nonlinear,gaitsgory2009linear}.

After defining the primal and dual LP formulations in our setting in Section~\ref{subsec:def}, we prove in Section~\ref{subsec:eq} this equivalence, by  adapting a similar result from~\cite{lasserre2008nonlinear} to the present case. The main arguments rely on generic duality results in infinite-dimensional linear programming~\cite{anderson1983review}, and convexity arguments proposed by~\cite{vinter1993convex}. %

\subsection{Definition of the primal and dual} \label{subsec:def}

Let $\cL$ be the differential operator
\begin{equation*}
    \cL:
    \left\{
    \begin{aligned}
        \; \cC^1(\cX)&  \rightarrow \cC(\cX \times \Rr^m)
        \\
        \varphi ~~ &\mapsto  \Bigl\{ (x, u) 
        \mapsto
        \varphi(x) - \frac{1}{\rho} \nabla \varphi(x)^\top b(x, u) \Bigr\}.
\end{aligned}
\right.
\end{equation*}
We define $\cL^*:\cM_c(\cX\times\Rr^m)\to \cC^1(\cX)^*$
as the dual of $\cL$, i.e., such that
$\langle  \cL \varphi, \nu \rangle = \langle \varphi , \cL^* \nu \rangle  $
for all $\nu \in \cM_c(\cX \times \Rr^m)$
and $\varphi \in \cC^1(\cX)$.

For $\nu\in\cM_+(\cX\times\Rr^m)$
with first marginal $\mu\in\cM_+(\cX)$,
the disintegration theorem states the existence
of $(\nu_x)_{x\in\cX}\subset\cP_c(\Rr^m)$
a measurable family of compactly supported 
probability measures, 
such that $\nu(A,B)=\int_A\nu_x(B)\mu(\mathrm dx)$
for any Borel sets $A\subset\cX$ and $B\subset\Rr^m$.
For simplicity, we write $\nu(\mathrm dx, \mathrm du)=\nu_x(\mathrm du)\nu(\mathrm dx)$.
Using this formulation, $\cL^*$ writes
\begin{equation*}
    \cL^*\nu
    =
    \mu +\rho^{-1}{\rm div}_x
    \left(\mu\int_{\Rr^m}b(x,u)\,\nu_x(\mathrm du)\right)
    \in\cC^1(\cX)^*,
\end{equation*}
where the divergence operator takes as an argument a
function that is not differentiable in general
and should be understood in the sense of distribution.

We are now in position to state our primal and dual problems:
\begin{align} 
\label{P}\tag{P}
    &\inf_{\nu \in \cM_+(\cX \times \Rr^m)} 
    \iint_{\cX\times\Rr^m}
    (f(x)+R(u))\,\nu(\mathrm dx, \mathrm du)
    \hspace*{0.3cm}
    \text{such that}
    \hspace*{0.3cm}
    \cL^* \nu = \mu_0 \, .
    \\
    \label{D}\tag{D}
    &\sup_{V \in \cC^1(\cX)}
    \int_{\cX}V(x)\,\mu_0(x)\, \mathrm dx
    \hspace*{0.3cm}
    \text{such that}
    \hspace*{0.3cm}
    \cL V \leq f + R \, .
\end{align}

These two problems are dual to each other.
The primal problem aims at finding occupation measures minimizing the expected cost, following the dynamics through Liouville's equation. The dual problem looks for a maximal subsolution  of the HJB equation. Note that similar LP formulations exist for Markov decision processes~\cite{nazareth1986linear}.

\begin{lemma}
    \label{lem:primal_on_proba}
    Under Assumption \ref{hypo:regularity} and \ref{hypo:alpha},
    the infimum of \eqref{P} can be taken on the set of
    probability measure supported on the graph
    of an $L^{\infty}(\cX;\Rr^m)$-function. 
    May minimizers exist, they belong to the same set.
\end{lemma}

In the rest of this paper, our objective will be to provide formulations of such optimization problems which are compatible with stochastic optimization.%

\subsection{Equivalence between the LP formulations and the OCP} \label{subsec:eq} In this section, we state a triple  equivalence between the primal and dual formulations (proving \textit{strong duality}), and the original control problem (stating that the \textit{weak formulation} is exact). This result is an adaptation of a similar theorem proposed by~\cite{lasserre2008nonlinear} to our infinite-horizon, discounted problem on the torus, with a non-compact control set%
, and specialized to the control-affine setting. Note that a similar result is provided for the discounted setting by~\cite{gaitsgory2009linear}, yet with different state and control sets.

\begin{theorem} \label{thm_duality}
      Under Assumptions \ref{hypo:regularity} and \ref{hypo:alpha},
      the following statements hold:
    \begin{enumerate}
        \item \textnormal{(P)} and \textnormal{(D)}  are feasible, \textnormal{(D)}  is the Lagrange dual of \textnormal{(P)}, and \begin{align}
            \textnormal{val}\eqref{P}  
            \geq
            \textnormal{val}\eqref{D} 
            \geq 
            \int_\cX V^*(x) \mu_0(x) \, \mathrm dx \, .
        \end{align}
        \item We have strong duality: $\textnormal{val}\textnormal{(P)} = \textnormal{val}\textnormal{(D)}$.%
        \item Their common value is \begin{align}
            \textnormal{val}\textnormal{(P)}=\textnormal{val}(\textnormal{D})=\int_\cX V^*(x) \mu_0(x) \, \mathrm dx \, .
        \end{align} %
    \end{enumerate}
\end{theorem}

\begin{proof}[Proof sketch]
In order to apply the similar result of~\cite{lasserre2008nonlinear} to our setting, several adaptations must be made. Crucially, we need to consider the control-constrained problem, as compactness of the control-set is required to prove strong duality using an argument  from~\cite{anderson1983review}. %
The main argument supporting the equivalence between the weak and strong formulations of the OCP in the constrained case is a comparison principle provided by~\cite{bardi1997optimal}. The conclusion follows from Proposition~\ref{prop:optimality}. The sequence of proved claims is summarized in Figure~\ref{fig:sketchofproof}.
\end{proof}

\begin{remark} \label{kkt}
    The Karush-Kuhn-Tucker (KKT) optimality conditions can be derived for the linear program (P). Under the conditions of Theorem~\ref{thm_duality}, $V^*$ is the optimal dual variable associated to the equality constraint $\cL^* \nu^* = \mu_0$. Let $\chi^* : \cX \times \Rr^m \rightarrow \Rr_+$ the optimal dual variable associated to the inequality constraint $\nu^* \geq 0$. The stationarity condition is then:
    \begin{align} \label{eqn:stat}
      \forall (x, u) \in \cX \times \Rr^m, \quad  - V^*(x) + f(x) + R(u) +\frac{1}{\rho} \nabla {V^*(x)}^\top b(x,u) = \chi^*(x, u) \, .
    \end{align}
    Hence $\chi^*$  represents the slackness in the HJB inequality.
    In addition, the complementary slackness condition reads:
    \begin{align*}
        \int_\cX \int_{\Rr^m} \chi^*(x, u) \nu^*(x,u) ~ \mathrm du ~ \mathrm d x = 0 \, .
    \end{align*}
    Since $\chi^*$ and $\nu^*$ are both non-negative, this means that $\forall (x, u), \chi^*(x,u) \nu^*(x, u)=0$. Because of Lemma~\ref{lem:primal_on_proba}, $\nu^*$ only puts mass on a unique $u(x)$ for each $x$, so that we must have $\chi^*(x,u(x))=0$ at this point. Observing the stationary condition~\eqref{eqn:stat} and by strong convexity of $R$, this means that $u(x)$ is equal to the optimal controller $u^*(x)$. Therefore, by combining the KKT optimality conditions, we can recover all three optimality equations from Proposition~\ref{prop:optimality}. %
\end{remark}

\section{The dual problem} \label{sec:dual} Let us now consider the dual weak formulation~\eqref{D}. From Theorem~\ref{thm_duality}, we know that solving~\eqref{D} is equivalent to solving the OCP. Yet,  \eqref{D} is a linear program in infinite dimension, and cannot, in general, be solved exactly. In this section, we explore conditions under which stochastic approximation can be used to obtain an approximate solution with a convergence guarantee, and present the potential limitations of such an approach.

\subsection{Exact penalty formulation of the dual problem}

\eqref{D} has an infinite-dimensional constraint, which is an obstacle to a direct application of SGD. Several approaches are possible to get rid of such constraints. Entropy-regularized or KL-constrained policy-search methods, such as relative entropy policy search~\cite{peters2010relative}, use a penalization approach, but usually require a decreasing regularization parameter. Here, we consider an exact penalty approach, which is valid under a uniform boundedness assumption on the optimal occupation measure~$\mu^*$.

\begin{enumerate}[label={\bf A3}]
    \item
    \label{hypo:mustarbounded} The optimal occupation measure $\mu^*$ is absolutely continuous with respect to the Lebesgue measure and there exists $M >0$ such that for all $x \in \cX$, $\mu^*(x) \leq M$.
\end{enumerate}

Determining conditions on the dynamics, cost and discount factor under which Assumption~\ref{hypo:mustarbounded} holds is an interesting problem of separate interest, which we will discuss in Section~\ref{sec:mustar}. For now, let $M > 0$ such that \ref{hypo:mustarbounded} holds. We define the following unconstrained problem:
\begin{align*} \tag{$\text{D}(M)$} \label{DM} \sup_{V \in \cC^1(\cX)} \ \   \int_{\cX} V(x) \mu_0(x) \, \mathrm dx & \\  - \int_{\cX} M \left[V(x) - f(x) - \frac{1}{\rho}\nabla V(x)^\top a(x) \right. & \left. + R^* \left(-\frac{1}{\rho} B(x)^\top \nabla V(x) \right) \right]_+ \, \mathrm dx \, ,
\end{align*}
where $[y]_+:=\max(y, 0)$ denotes the positive part of $y$. To highlight the link between \eqref{D} and \eqref{DM}, we can rewrite \eqref{D} as:
\begin{align*}
\sup_{V \in \cC^1(\cX)} \ \  \inf_{\mu : \cX \to \Rr_+} \ \  \int_{\cX} V(x) \mu_0(x) \, \mathrm  dx + \int_{\cX} \mu(x) &\Big[  -  V(x) +  f(x) + \frac{1}{\rho} \nabla V(x) ^\top a(x) \Big. \\
& ~~~ \Big.  - R^*\left(-\frac{1}{\rho} B(x)^\top \nabla V(x) \right) \Big] \, \mathrm d x \, .
\end{align*}
Since strong duality holds, then, under Assumption~\ref{hypo:mustarbounded}, we can add the  extra constraint  $\forall x \in \cX,\,  \mu (x) \leq M$ in the problem above without changing  its value. The optimal solution of the inner problem in $\mu$ is to take $\mu(x)=M$ where the constraint is negative, and $\mu(x)=0$ elsewhere, \textit{i.e.},
\begin{align*}
    &\inf_{\mu:\cX \rightarrow [0, M]} \ \  \int_\cX \mu(x) \left( -  V(x) +  f(x)  + \frac{1}{\rho} \nabla V(x) ^\top a(x)  - R^*\left(-\rho^{-1} B(x)^\top \nabla V(x) \right) \right) \mathrm dx \\
    &\qquad\quad  = - \int_{\cX} M \left[   V(x) -  f(x)  - \frac{1}{\rho} \nabla V(x) ^\top a(x)  + R^*\left(-\rho^{-1} B(x)^\top \nabla V(x) \right)  \right]_+ \mathrm dx \, .
\end{align*}
The proposition below follows from taking the supremum on $V$ of this expression.
\begin{proposition} \label{prop:DDM}
    Under Assumptions \ref{hypo:regularity}-\ref{hypo:alpha}-\ref{hypo:mustarbounded}, \eqref{D} is equivalent to~\eqref{DM}.
\end{proposition}

\eqref{DM} is now an unconstrained problem appearing as  maximizing the expectation of a concave function of $V$, an infinite dimensional variable. We will introduce in Section~\ref{sec:sgddual} a non-parametric representation of $V$ in a reproducing kernel Hilbert space (RKHS)~$\cH$~\cite{shawe2004kernel} that will allow for a practical finite-dimensional implementation of a converging SGD algorithm to solve~\eqref{DM}. 
Therefore we start by proposing in Section~\ref{sec:auxsgd} a convergence result for SGD adapted to this setting, where the variable cannot be restricted to a bounded region. Since it can be of independent interest, we state it with generic notations.

\subsection{Auxiliary result on single pass SGD without projection} \label{sec:auxsgd}

We consider minimizing $F : \cH \rightarrow \Rr$ where $F$ is a convex function defined on a Hilbert space~$\cH$. We assume that, for $\theta \in \cH$:
\begin{align*}
    F(\theta) = \Ee_{z} \left[ \rho_1(g_1(\theta, z)) + \rho_2(g_2(\theta, z))\right] \, ,
\end{align*}
where for each $z$, $\theta \mapsto g_1(\theta, z)$ is  convex and $L_1$-smooth, $\theta \mapsto g_2(\theta, z)$ is convex and $L_2$-smooth, $\rho_1 : \Rr \rightarrow \Rr$ and $\rho_2 : \Rr \rightarrow \Rr$ are convex non-decreasing functions, with derivatives respectively in $[0, \alpha_1]$ and $[0, \alpha_2]$. Note that $F$ is then a convex function. Let $z_1, \dots , z_n$ i.i.d.~realizations of the random variable $z$, and $\gamma>0$. For $i \in \{1, 2 \}$, let $g_i' \in \cH$ denote the partial derivative of $g_i$ with respect to $\theta$. Starting from some $\theta_0 \in \cH$, we consider the single pass SGD recursion:
\begin{align} \label{sgd}
    \theta_n = \theta_{n-1} - \gamma \left[ \rho_1'(g_1(\theta_{n-1}, z_n)) g_1'(\theta_{n-1}, z_n) + \rho_2'(g_2(\theta_{n-1}, z_n)) g_2'(\theta_{n-1}, z_n) \right] \, .
\end{align} 

Various analyses of the convergence of SGD exist, yet in slightly different settings. Averaged SGD is known to provide an optimal convergence rate of $O(1/\sqrt{n})$ for optimizing Lipschitz continuous (non-smooth) functions (see \cite[Chapter~5]{nemirovskij1983problem} and \cite{agarwal2009information}), yet on bounded domains. In our case of interest, we do not assume that $\| \theta^*\|_\cH$ is finite, for $\theta^*$ minimizing $F$ on $\cH$. One possible option would be to apply this analysis to   projected SGD on a compact set of increasing radius. However, we can provide a convergence result \textit{without projection} by an approach similar to those proposed for kernel regression in~\cite{DieBac2016} and~\cite[Chapter~7]{bach2024learning}. In particular, we reuse tools developed in~\cite{schmidt2011convergence} for the distinct context of SGD with inexact gradients.  %

\begin{proposition} \label{prop:sgd} Consider the recursion \eqref{sgd} with $\gamma \leq \min \{ \frac{1}{4 \alpha_1 L_1}, \frac{1}{4 \alpha_2 L_2} \}$. Let, for $n \geq 1$, $\bar \theta_n = \frac{1}{n} \sum_{k=0}^{n-1} \theta_k$. Then we have:
\begin{align*}
     \Ee [F(\bar \theta_n)] - \inf F   \leq \inf_{\theta \in  \cH} \left\{ [ F(\theta) - \inf F]  + \left( \frac{1}{\gamma n} + 32 \gamma^2 n \alpha L + 32 \gamma   \right) \alpha^2  L^2\| \theta-\theta_0\|^2 \right\} 
 \\
  \qquad\qquad + \left( 32 \gamma^2 n \alpha^3 L + 32 \gamma \alpha^2 \right) \sigma^2(\theta_0) \, ,
\end{align*}
where $\alpha = \max \{ \alpha_1, \alpha_2 \}$, $L = \max \{L_1, L_2 \}$ and $\sigma^2(\theta_0)=\max_{i \in \{1,2\}}\Ee_z[ \| g'_i(\theta_0, z)\|^2]$. 
\end{proposition}

\subsection{Convergence of SGD on the unconstrained dual problem} \label{sec:sgddual} Let us consider an RKHS $\cH \subset L^2(\cX)$, with reproducing Mercer kernel~$K$, satisfying the \textit{universal approximation property}~\cite{micchelli2006universal}, \textit{i.e.}, ${\bar \cH = L^2(\cX)}$. Let us denote by $\langle \cdot , \cdot \rangle_\cH$ the inner product of $\cH$ such that $\forall h \in \cH, \langle h, K(x, \cdot) \rangle_\cH = h(x)$, and by $\| \cdot \|_\cH$ the corresponding norm. 
Examples of RKHS that are dense in $L^2(\cX)$ are based on  translation-invariant kernels. In dimension $d=1$, such kernels are of the form $K(x, y) = \kappa(x-y)$, where $\kappa : [0, 1] \rightarrow \Rr$ is a squared-integrable 1-periodic function with non-negative Fourier coefficients $(\hat \kappa_\omega)_{\omega \in \Zz}$~\cite{wahba1990spline}. In this case, an explicit feature embedding $\Phi(x) = K(x, \cdot)$ is given by $\Phi(x) = (\sqrt{\hat \kappa_\omega} e^{2i\pi \omega x})_{\omega \in \Zz}$, so that 
\begin{align*}
    K(x, y) = \langle \Phi(x), \Phi(y) \rangle_\cH = \sum_{\omega \in \Zz} \hat \kappa_\omega e^{2i\pi \omega (x-y)}  \,  .
\end{align*}
With different choices of $(\hat \kappa_\omega)_\omega$ sequences, one can construct different RKHS $\cH$, among which Sobolev spaces,  imposing different regularity conditions on the functions they contain~\cite{bach2017equivalence}. A simple extension to $\cX=[0, 1]^d$ is to define the kernel as a point-wise product of one-dimensional kernels $K(x, y) = \prod_{i=1}^d \kappa(x_i-y_i)$, although other options can be preferred to recover Sobolev spaces~\cite{bach2017equivalence,berlinet2011reproducing}.

In our specific setting, because we will consider operators that depend on the first derivatives of functions in $\cH$, we need to introduce a slight modification of the universal approximation property, and assume that \begin{align*}
    \overline{\cH}^{H^1(\cX)} = H^1(\cX) \, ,
\end{align*} where $H^1(\cX)$ is the Sobolev space $W^{1,2}(\cX)$ of functions whose weak-derivatives up to order 1 are square-integrable, and the density holds with respect to the $H^1$-norm.  %
Moreover, on the torus, it is easy to construct a 
large class of translation-invariant kernels that are 
dense in $H^1$, using the Fourier coefficients of the kernel~\cite{micchelli2006universal,adams2003sobolev,pillaud2023kernelized}.   %

 $V^*$ being uniformly bounded on~$\cX$, its $L^2$-norm is bounded by $\|f\|_\infty$. However, due to its lack of regularity, $\| V^* \|_\cH$ is not necessarily finite, meaning that generally, $V^*$ belongs to~$\bar \cH$ but not to~$\cH$. Therefore, we will propose a version of \eqref{DM} where $V$ is restricted to belong to~$\cH$, allowing for a practical finite-dimensional representation thanks to the so-called \textit{kernel trick}, and conclude by a density argument.
 
 In this Section, \textit{we restrict to the quadratic case $q=q'=2$}, where results from Section~\ref{sec:auxsgd} can be applied directly. If $V \in \cH$, the reproducing property ensures that for any $x \in \cX$, $V(x) = \langle V, \Phi(x) \rangle_\cH $. Similar reproducing properties for the first-order derivatives~\cite{zhou2008derivative} of $V$ exist as soon as $K \in \cC^2(\cX \times \cX)$, so that for any $i \in \{1,..., d\}$,%
 \begin{align*}
     \frac{\partial V}{\partial x_i}(x)&= \left\langle V, ~ \frac{\partial K}{\partial x_i}(x, \cdot) \right\rangle_\cH =   \left\langle V, ~ \frac{\partial \Phi}{\partial x_i}(x) \right\rangle_\cH  \, . %
 \end{align*}
Using this property, and the fact that $\cX=[0, 1]^d$ has Lebesgue measure 1, we can define the following counterpart of \eqref{DM} in $\cH$:
\begin{align*} & (\text{D}(M,\cH)) \quad\qquad\qquad
\sup_{V\in \cH} \ \ \Ee_{x \sim \cU(\cX)} ~ \Big\{ ~ \Lambda(V, x)  - M \left[ Q(V, x)  \right]_+ \Big\} \, \, , \hspace{8em}  %
\end{align*}
where for each $x\in\cX$, $\Lambda(\cdot, x)$ and $Q(\cdot, x)$ are respectively linear and quadratic functions of $V$:
 \begin{align*}
 \Lambda(V, x) &= \langle V, ~ \mu_0(x) \Phi(x) \rangle_\cH \\
 Q(V, x) &= - f(x) + \left\langle V, ~ \Phi(x) - \frac{1}{\rho}\sum_{i=1}^d a_i(x) \frac{\partial \Phi}{\partial x_i}(x) \right\rangle_\cH +  \langle V, ~ \Xi(x) V \rangle_\cH, \\
 \text{where ~~}  \Xi(x) V & = \frac{1}{2\rho^2} \sum_{k=1}^m \left\{ \left\langle V, ~\sum_{j=1}^d B_{j,k}(x) \frac{\partial \Phi}{\partial x_j}(x) \right\rangle_\cH \sum_{i=1}^d B_{i,k}(x)\frac{\partial \Phi}{\partial x_i}(x)  \right\} \, .
\end{align*}
Let $H_M(V, x) =  \Lambda(V, x) - M \left[ Q(V, x)  \right]_+ $ and $F_M : V \mapsto \Ee_x [H_M(V, x)] $, $V_0=0$  and $\gamma >0$. %
We define the following single-pass stochastic gradient ascent iterations on (D($M, \cH$)): at each time step $n \geq 1$, sample~$x^{(n)}$ uniformly on~$\cX$ and make the update:
\begin{align} \label{sgdH}
    V_{n} = V_{n-1} + \gamma \frac{\partial H_M}{\partial V}(V_{n-1}, x^{(n)}) \, .
\end{align}

\begin{theorem} \label{thm:dual}
Assume that $q=2$, and that Assumptions \ref{hypo:regularity}-\ref{hypo:alpha}-\ref{hypo:mustarbounded} hold. %
     Let $n \geq 1$, and set $\gamma = C n^{-2/3}$. Define $\bar V_n = \frac{1}{n} \sum_{k=0}^{n-1} V_k$. Then, there exist $C_1, C_2 \geq 0$ independent from $n$ such that:
     \begin{align*}
     \Ee [F_M(\bar V_n)]   \geq \sup_{V \in  \cH} \left\{  F_M(V)   - C_1 n^{-1/3}\| V\|_\cH
     ^2 \right\} - C_2 n^{-1/3} \, .
\end{align*}
If, in addition,  $\overline{\cH}^{H^1(\cX)} = H^1(\cX)$:
        \begin{align*}
            \lim_{n \to +\infty} \Ee [F_M(\bar V_n)] = \int_\cX V^*(x) \mu_0(x) \, \mathrm d x \, .
        \end{align*}
\end{theorem}

\begin{remark} The above theorem only provides an asymptotic result. One could also obtain \textit{non-asymptotic} convergence rates, but at the cost of additional regularity assumptions on $V^*$ with respect to $\cH$ (see, \textit{e.g.}, \cite[Chapter~7]{bach2024learning}), which do not generally hold in our present setting. 
\end{remark}

\subsection{Finite-dimensional implementation} Each iteration~\eqref{sgdH} takes place in an infinite-dimensional space. As often with kernel methods, it turns out that manipulating this abstract representation is not necessary, and the iterations can be derived by  using only finite-dimensional coefficients. This can be simply obtained by construction from~\eqref{sgdH}, and provides an algorithmic equivalent of classical representer theorems~\cite{scholkopf2001generalized, zhou2008derivative}, without any regularization needed.

The following proposition proves that at each step $n$, $V_n$ has a $n(d+1)$-dimensional representation. This representation can be computed exclusively from evaluations of the kernel $K$ and its first and second-order derivatives~\cite{zhou2008derivative} at observations $(x^{(k)})_{1\leq k\leq n}$ --- an observation generally referred to as a \textit{kernel trick}.%

\begin{proposition} \label{prop:implem}
  Let $N \geq 1$.  There exist $\alpha \in \Rr^N$ and $\beta \in \Rr^{N \times d}$ such that
\begin{align}
   \forall n \in \{1, ... , N \}, \quad V_n = \sum_{k=1}^{n}\alpha_{k} \Phi(x^{(k)}) + \sum_{k=1}^{n} \sum_{i=1}^d \beta_{k, i} \frac{\partial \Phi}{\partial x_i}(x^{(k)}) \, . \label{eqn:vn}
\end{align}
In particular, for any $y \in \cX$, $j \in \{1,...,d\}$ and $n \in \{1, ... , N \}$:%
\begin{align}
    V_n(y) & = \sum_{k=1}^{n}\alpha_{k} K(x^{(k)}, y) + \sum_{k=1}^{n} \sum_{i=1}^d \beta_{k, i} \frac{\partial K}{\partial x_i}(x^{(k)}, y) \label{eqn:vnx} \\
    \frac{\partial V_n}{\partial y_j}(y) &= \sum_{k=1}^{n}\alpha_{k} \frac{\partial K}{\partial y_j}(x^{(k)}, y)+ \sum_{k=1}^{n} \sum_{i=1}^d \beta_{k, i} \frac{\partial^2 K}{\partial x_i \partial y_j}(x^{(k)}, y) \, . \label{eqn:vndx}
\end{align}
\end{proposition}

\begin{remark} One should be careful when manipulating the derivatives of $K$: the partial derivatives with respect to the first $d$ variables are different from those from the last $d$ variables. However, since $K$ is translation-invariant, we have $\forall x, y \in \Rr^d$, $\forall j \in \{1, ..., d\}$,  $\frac{\partial K}{\partial y_j}(x, y) = - \partial_j \kappa(x-y) =  -\frac{\partial K}{\partial x_j}(x, y)$. Therefore, evaluations of~$V$ and its derivatives can be computed solely from the kernel matrices computed with~$K$ and its first and second derivatives with respect to the first $d$ entries.%
\end{remark}

\begin{remark} Adding a random perturbation --- usually a Brownian motion --- to the dynamics of the system is a standard way to enhance the regularity of the value function, making it twice differentiable~\cite{fleming2012deterministic}. This perturbation adds a Laplacian term~$\Delta V$ to the HJB equation. The dual approach proposed above can be readily adapted, at the mild cost of increasing the number of parameters of the representation of~$V_n$ from $n(d+1)$ to $n(d+2)$. The additional parameters account for  $n $ extra  terms of the form $\sum_{i=1}^d  \frac{\partial^2 \Phi}{\partial x_i^2}(x^{(k)})$  appearing in~\eqref{eqn:vn}, coming from the Laplacian.
\end{remark}

\subsection{Illustrative example} \label{subsec:ex}
We introduce the following example on the circle, with $\mathcal{X} = [0, 1]$, $\mathcal{U} = \mathbb{R}$, $\mu_0 = \cU([0, 1])$, $\dot x = u$,  $f(x) = x(x-1)+1/4$, and $R(u)=u^2/2$.
In this case, we can compute explicitly $u^*$, $V^*$ and $\mu^*$, and observe the influence of the discount factor~$\rho$. $V^*$ is computed from the HJB equation by identification of a degree-two polynomial, periodic on the torus, and we obtain $V^*(x) =  \alpha_\rho f(x)$, with \begin{align} \label{eqn:alpharho}
    \alpha_\rho = \left(\sqrt{1+ \frac{8}{\rho^2}} - 1 \right) \frac{\rho^2}{4} \, .
\end{align}
Then, we can compute the optimal controller and occupation measure as follows. $u^*(x) = - \frac{1}{\rho} \nabla V^*(x) = - \frac{2 \alpha_\rho}{\rho} (x-1/2)$. Let $y=x-1/2$. Then, under the optimal controller $u^*$, the trajectory at time $t$ is \begin{align*}
y(t) = y(0) e^{-\frac{2\alpha_\rho}{\rho} t} \, .\end{align*}
Note that $\alpha_\rho > 0$ for any $\rho >0$.  Hence $\mu^*(x)$ is the probability density of $Y(T)+1/2$, where $Y_0 \sim \cU([-1/2, 1/2])$ and $T \sim \text{Exp}(\rho)$, with $Y_0$ and $T$ independent. Let ${\beta_\rho = \frac{\rho^2}{2\alpha_\rho}}$. For $y > 0$, after some computations, we obtain, for $\rho \neq 1$:
\begin{align*}
    \mathbb{P}(Y(T) \leq y) = y+1/2 + \frac{y^{\beta_\rho}}{1-{\beta_\rho}} \left[ (1/2)^{1-{\beta_\rho}} - y^{1-{\beta_\rho}} \right] \, , 
\end{align*}
and for $\rho=1$, $\alpha_\rho=1/2$ and $\beta_\rho=1$, in which case:
\begin{align*}
    \mathbb{P}(Y(T) \leq y) = y+1/2 -y \log(2y) \, . 
\end{align*}
From that, we can deduce the expression of the density $\mu^*(x)$ as follows:%
\begin{align*}
    \mu^*(x) &=\left\{
    \begin{array}{ll}
        - \frac{{\beta_\rho}}{1-{\beta_\rho}} +(1/2)^{1-{\beta_\rho}} \frac{{\beta_\rho}}{1-{\beta_\rho}} {|x-1/2|}^{{\beta_\rho}-1}& \mbox{if }\rho \neq 1  \\
       - \log |2x-1| & \mbox{if }\rho=1 \, .
    \end{array}
\right.
\end{align*}
One can check that $\mu^*(x)$ grows unbounded as $x \rightarrow 1/2$ if and only if $\rho \leq 1$ (see Figure~\ref{fig:mustar}), therefore Assumption~\ref{hypo:mustarbounded} holds if and only if $\rho > 1$.

\begin{figure}
    \centering
    \includegraphics[width=0.5\linewidth]{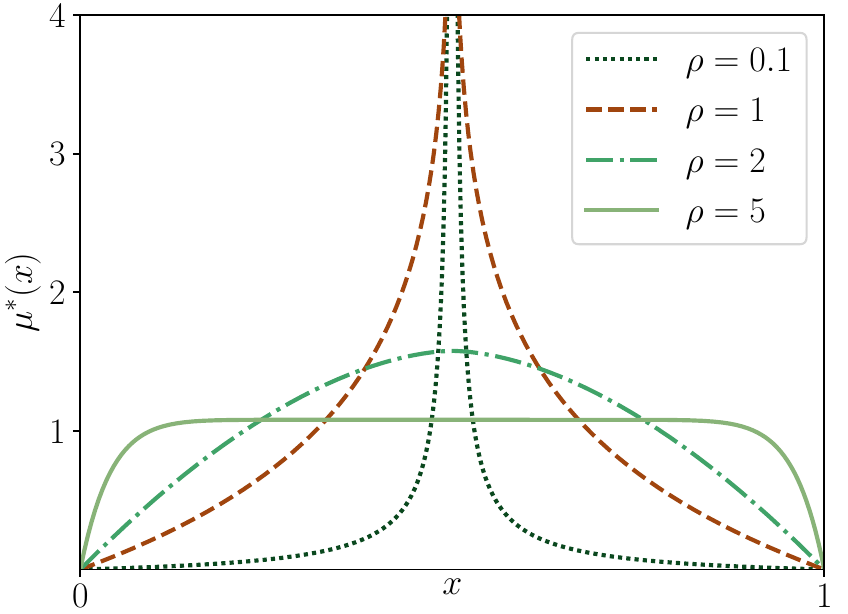} \vspace{-0.25cm}
    \caption{The optimal occupation measure~$\mu^*$ can be bounded or not, depending on the value of~$\rho$. In this example, this change occurs at $\rho=1$. This appears even on this simple example; this is not a pathological phenomenon, but a common one, as discussed in more details in Section~\ref{sec:mustar}.}
    \label{fig:mustar}
\end{figure}

\paragraph{Application of SGD on the dual} We implement the iterations~\eqref{sgdH} on this one-dimensional problem. We choose as RKHS~$\cH$ the Sobolev space of regularity $s=2$ denoted by $H_2^{\text{per}}$~\cite{wahba1990spline}. It is associated with the Sobolev kernel 
\begin{align*}
    K^{\text{Sob}}_{2}(x, y) = 1 - \frac{(2\pi)^4}{24} B_4(\{x-y \}) \, ,
\end{align*}
where $\{x\} := x - \lfloor x \rfloor$ and $B_4(x) = x^4 - 2 x^3 + x^2 - \frac{1}{30} $ is the fourth Bernoulli polynomial~\cite{olver2010nist}. One can check that we have $K^{\text{Sob}}_{2} \in \cC^2 (\cX \times \cX)$, which would not have been the case for the Sobolev kernel of regularity $s=1$. We compute the averaged iterates $\bar V_n$ with $\rho = 2$%
, $M=2$ --- so that Assumption~\ref{hypo:mustarbounded} holds --- and constant step size $\gamma= 5 \times 10^{-4}$.  The results are shown in Figures~\ref{fig:dualperf} and~\ref{fig:dual}.

Not only does $F_M(\bar V_n)$ converge from below  to the optimal value of the problem $\int_\cX V^* \mathrm d \mu_0$ (see Figure~\ref{fig:dualperf}), we also observe that $\bar V_n$ effectively approaches $V^*$ uniformly (Figure~\ref{fig:dual}). However, examining the candidate controller $\bar u_n = - \bar V'_n / \rho$ resulting from~$\bar V_n$, we see that it struggles to capture the right value of~$u^*$ as $x \rightarrow 0$ and $1$. This is because~$\cH$ only contains periodic functions of class~$\cC^2$ on the torus, hence cannot model functions with discontinuous derivatives. As we see in Figure~\ref{fig:dualperf}, the performance of this controller, defined as the expected cost of running it: \begin{align} \label{eqn:poleval}
    \text{Perf}(u) = \Ee_{x_0 \sim \mu_0} \left\{ \rho \int_0^{+\infty} e^{-\rho t}(f(x_t)+ R(u_t)) \,  \mathrm d t \right\}  \,  \text{ where } \dot x_t = g(x_t, u_t) \, ,
\end{align} remains slightly suboptimal. Achieving convergence of the controller in this example would require $\bar V_n$ to converge in the norm $\varphi \mapsto \| \varphi \|_{L^2} + \| \varphi' \|_{L^2}$ of the Sobolev space~$H^1$, ensuring that $\bar u_n = - \bar V'_n / \rho$ converges to $u^*$.

Drawing a parallel with generic methods that aim to solve partial differential equations (PDEs) by minimizing their residuals, \textit{e.g.}, the least-squares finite element method (LSFEM) based on minimizing their~$L^2$-norm, in general they only provide~$L^2$ convergence to a solution of the PDE~\cite[Chapter~4]{jiang2013least}. Only extra coerciveness properties, fulfilled by linear elliptic systems --- which is not the case of the HJB equation even on this simple example --- allow for $H^1$-norm convergence of the approximation.

In conclusion, %
while directly optimizing the value function through solving problem~\eqref{DM} is an effective way to find some $V \simeq V^*$, it provides little guarantee on the quality of the controller based on $\nabla V$. %
Yet, in the context of optimal control or reinforcement learning problems, the purpose of optimizing the value function is ultimately to find an approximately optimal controller $u$ such that $\text{Perf}(u) \simeq \text{Perf}(u^*)$. In the next sections, we will explore other approaches, not directly based on the value function, and explain why they could be more promising for this goal. %

\begin{figure}
    \centering
    \includegraphics[width=0.45\linewidth]{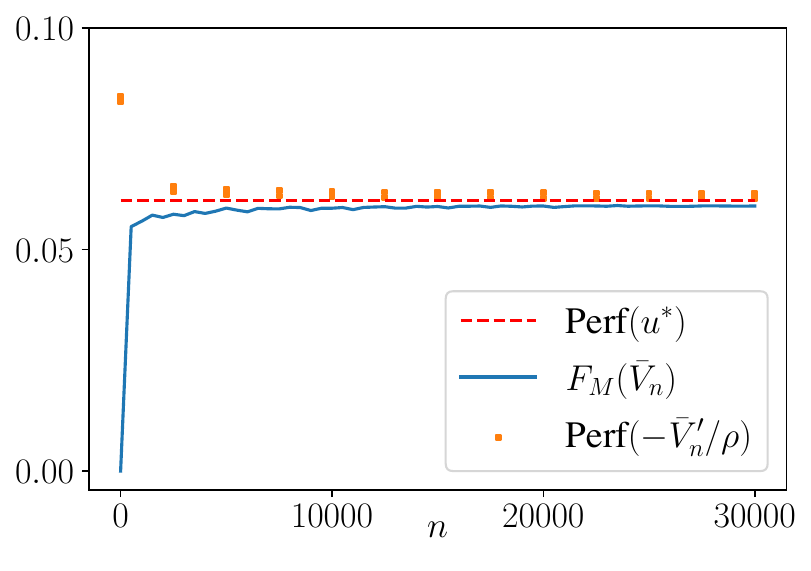}
    \vspace{-0.35cm}
    \caption{Evolution of the dual cost, along with the performance of the induced controller, over the iterations. Note that the performance is evaluated up to some finite precision, due to discrete-time approximations in the simulated trajectories.}
    \label{fig:dualperf}
\end{figure}

\begin{figure}
    \centering
    \includegraphics[width=0.98\linewidth]{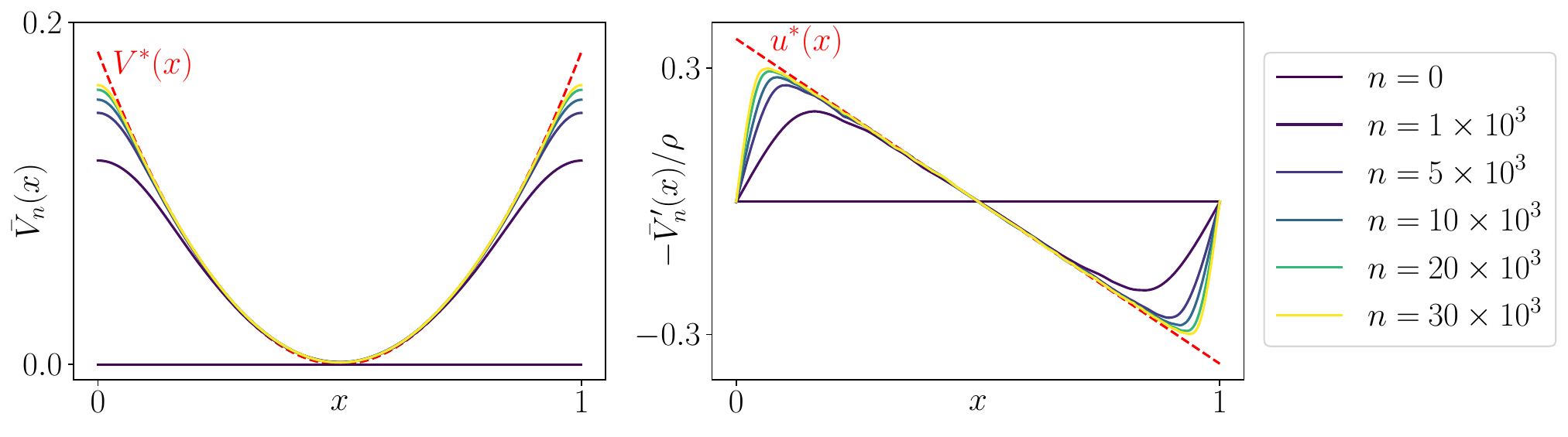} \vspace{-0.3cm}
    \caption{$\bar V_n$ obtained from iterations~\eqref{sgdH}, and the corresponding candidate controller obtained from  its gradient $\bar V'_n$ are plotted for growing values of~$n$, along with $V^*$ and $u^*$ in red dotted lines.}
    \label{fig:dual}
\end{figure}

\section{The primal problem} \label{sec:primal}

We now consider the primal linear program~\eqref{P}. Like~\eqref{D}, it has an infinite number of constraints. We start by reformulating it as a constrained problem, and then introduce a simple change of variables that will make the problem compatible with stochastic optimization when the dynamics is linear in the controller: $b(x, u) = B(x) u$.

\subsection{Exact penalty formulation of the primal}  \label{sec:exactpen} 
With the same manipulation as for the dual, we can transform Liouville's equality constraint into an exact penalty formulation. Indeed, by strong duality and noting that~$V^*$ takes values in $[0, \|f\|_\infty]$, \eqref{P} is equivalent to:%
\begin{align*}
     \inf_{\nu: \cX \times \Rr^m \to \Rr_+} \ & \sup_{V : \cX \rightarrow [0, \|f\|_\infty]} \  \int_{\X} \int_{\Rr^m} \big[ f(x) + R(u) \big] \nu(x,u) \, \mathrm du\,  \mathrm dx \\
    & + \int_\cX V(x) \left(\mu_0(x) - \mu(x) - \frac{1}{\rho} {\rm div} \Big(
\int_{\Rr^m} \nu(x,u) b(x, u) \, \mathrm du \Big) \right)  \mathrm dx \, ,
\end{align*}
where $\mu(x) = \int_{\rb^m} \nu(x,u) \, \mathrm du$. This is in turn equivalent to:%
\begin{align} 
    \inf_{\nu: \cX \times \Rr^m \to \Rr_+} \ &   \int_{\X} \int_{\Rr^m} \big[ f(x) + R(u) \big] \nu(x,u) \, \mathrm du\,  \mathrm dx \nonumber \\ 
    & + C_V \int_\cX  \left[\mu_0(x) - \mu(x) - \frac{1}{\rho} {\rm div} \left(
\int_{\Rr^m} \nu(x,u) b(x, u) \, \mathrm du \right) \right]_+  \mathrm dx \, , \label{eqn:Ppen0}
\end{align}
for any $C_V \geq \| f\|_\infty$. \begin{remark}
    While it is sufficient here to choose $C_V=\|f\|_\infty$, we will consider a larger $C_V$ in Section~\ref{sec:poleval}, to be defined later, for consistency with policy evaluation.
\end{remark}
\paragraph{Key reformulation with perspective functions} We now introduce a change of variable which unlocks a convenient representation of the occupation measure. Similarly to $\mu(x) = \int_{\Rr^m} \nu(x, u) \, \mathrm du$, let us define 
\begin{align} \label{eqn:wdef}
    w(x) = \int_{\Rr^m} \nu(x, u) u \, \mathrm du \in \Rr^m \, .
\end{align}
From Lemma \ref{lem:primal_on_proba}, the infimum in \eqref{P} can
be taken on probability measure supported by a $L^{\infty}$-function,
\textit{i.e.}, of the form:
\begin{align} \label{eqn:wstar}
    w(x) = \mu(x) u(x)
    \hspace*{0.3cm}
    \text{ for }
    \hspace*{0.3cm}
    u\in L^{\infty}(\cX).
\end{align}
Using~\eqref{eqn:wstar} and replacing $\nu:\cX \times \cU \rightarrow \Rr_+$ by $w$ and $\mu: \cX \rightarrow \Rr$ in \eqref{eqn:Ppen0}, we obtain the equivalent formulation:
\begin{align} 
        \inf_{\mu: \cX \to \Rr_+, ~w:\cX \to \Rr^m} \ &   \int_{\X} \left[ f(x) + R\left( \frac{w(x)}{\mu(x)} \right) \right] \mu(x) \,  \mathrm dx \nonumber \\ \label{eqn:Ppen} \tag{P-P}
    & + C_V \int_\cX  \left[\mu_0(x) - \mu(x) - \frac{1}{\rho} {\rm div} \left(
a(x) \mu(x) + B(x) w(x) \right) \right]_+  \mathrm dx \, .
\end{align}

\begin{remark}
    The objective of~\eqref{eqn:Ppen} (for penalized primal) is a convex function of $(\mu, w)$. In particular, $(\mu, w) \mapsto R(w/\mu) \mu$ is called a perspective function~\cite{combettes2018perspective}, and is convex. \eqref{eqn:Ppen} is an optimization problem on both the occupation measure $\mu : \cX \to \Rr_+$ and a new vector-valued function $w : \cX \to \Rr^d$, related at optimality to the controller by~\eqref{eqn:wstar} and playing the role of a ``control density''. Similar variables exist in other optimization problems used in related fields. For instance, in optimal transport, a related $w$ is called the ``momentum'' in~\cite{benamou2000computational}, the ``intensity of traffic'' in~\cite{brasco2010congested},  or the ``transport density'' in \cite{brasco2014continuous}. In mean field games~\cite{lasry2007mean}, a similar $w$ variable appears in variational formulations~\cite{cardaliaguet2015weak,briceno2019implementation}, without a specific name. 
\end{remark}

\subsection{Stochastic optimization for control-linear dynamics} \label{sec:primalstoch} Let us consider problem~\eqref{eqn:Ppen}. It appears as the minimization of an expectation, and we could do stochastic optimization provided we can parameterize $\mu$ and $w$. While $w$ can be parameterized in an RKHS like previously $V$ for the primal, parameterizing $\mu$ is more challenging because of the non-negativity constraint. Although some parameterizations could be considered, \textit{e.g.}, sum-of-squares based representations~\cite{lasserre2001global,rudi2025finding}, they typically require stronger regularity assumptions and introduce additional computational difficulties. Instead, let us focus on the case $a(x)\equiv 0$. The optimization problem in $\mu$ decouples and becomes separable for each $x \in \cX$, and reads
\begin{align} \label{eqn:decouple}
\inf_{w: \X \to \rb^d } \int_{\X} \left\{ \inf_{\mu(x)\geq 0} P_x(w, \mu(x)) \right\}  \mathrm dx \, ,
\end{align}
where
\begin{align*}
    P_x(w, \mu(x)) = \Big[ f(x) +  R \Big( \frac{ w(x)}{\mu(x)} \Big) \Big] \mu(x) 
+ C_V \left[  \mu_0(x) - \mu(x) - \frac{1}{\rho} {\rm div} \left(B(x) w(x) \right) \right]_+  \, .
\end{align*}
\begin{remark}
    Note that such a decoupling is not possible if $a(x) \neq 0$, because of the ``non-local'' term $\text{div}(a(x) \mu(x))$ that does not depend solely on the scalar $\mu(x)$. The primal-dual approach presented in Section~\ref{sec:primaldual} will overcome this difficulty.
\end{remark}
The inner minimization in~\eqref{eqn:decouple} is a scalar optimization problem that can be solved in closed-form using the following lemma.
\begin{lemma} \label{lemma2}
    Let $C \geq c \geq 0$. %
    Consider the following one-dimensional minimization problem, parameterized by $(\alpha, \beta) \in \Rr_+ \times \Rr$:
    \begin{align*}
        L_{c, C}(\alpha, \beta) := \min_{\mu \in \Rr_+} ~ c \mu + \frac{\alpha^q}{q} \mu^{1-q} + C [\beta - \mu]_+ \, .
    \end{align*}
    Let $c_0 = \left(\frac{q-1}{q}\right)^{1/q}  c^{-1/q}$ if $c>0$, and $c_0=+\infty$ if $c=0$. Then we have:%
    \begin{align*}L_{c, C}(\alpha, \beta) =\left\{
\begin{array}{cl}
   \left(\frac{q}{q-1}\right)^{\frac{q-1}{q}}  c^{\frac{q-1}{q}}  \alpha   & \textnormal{if }  \beta \leq c_0 \alpha \\
    c \beta + \frac{\alpha^q}{q} \beta^{1-q} & \textnormal{otherwise } , 
\end{array}
     \right.
\end{align*}
and the minimizer is $\mu^* = \max \{ c_0 \alpha, \beta \}$. 
Furthermore, $L$ is almost everywhere differentiable with:
\begin{align*} &\frac{\partial L_{c, C}}{\partial \alpha} = \left\{
\begin{array}{cl}
    \left(\frac{q}{q-1}\right)^{\frac{q-1}{q}}  c^{\frac{q-1}{q}}   & \textnormal{if ~} \beta < c_0 \alpha \\
    \alpha^{q-1} \beta^{1-q} & \textnormal{if ~} \beta > c_0 \alpha \, .
\end{array}
     \right.
      \\ \textnormal{ and } \quad &\frac{\partial L_{c, C}}{\partial \beta} =\left\{
\begin{array}{cl}
    0 & ~\,\textnormal{if ~}\beta < c_0 \alpha\\
    c - \frac{q-1}{q} \alpha^q \beta^{-q}  & ~\,\textnormal{if ~} \beta > c_0 \alpha \, .
\end{array}
     \right.
\end{align*}
\end{lemma}

\begin{proof}
     The problem is equivalent to:
\begin{align*}
    \min \left\{ \min_{0\leq \mu \leq \beta} (c-C) \mu  + \frac{\alpha^q}{q} \mu^{1-q}+C \beta, \min_{\mu \geq \beta, \mu \geq 0} c \mu  + \frac{\alpha^q}{q} \mu^{1-q} \right\}. 
\end{align*}
Since $c - C \leq 0$ and $1-q < 0$, the left-hand term is a non-increasing function of $\mu$, hence its minimum is attained at $\mu= \beta$. The right-hand side can be minimized by setting its derivative to zero.
\end{proof}

\subsection{Convergence of SGD on the unconstrained primal problem} \label{sec:sgdprimal} Similarly to Section~\ref{sec:sgddual}, let us represent $w : \cX \to \Rr^m$ using a universal RKHS~$\cH$. Since~$w$ is vector-valued, each of its component is represented in~$\cH$, and if $w \in \cH^m$, we have the following vector-valued reproducing property \begin{align*}w(x) = \begin{bmatrix}
    \langle w_1, \Phi(x) \rangle_\cH \\ \vdots \\ \langle w_m, \Phi(x) \rangle_\cH
\end{bmatrix} \in \Rr^m %
\, .
\end{align*} %
As previously for the dual, we choose $K \in \cC^2(\cX, \cX)$ so that the partial derivatives of~$w$ can be computed for any $i \in \{1,..., d\}$ using%
\begin{align*} 
     \frac{\partial w}{\partial x_i}(x)=  \begin{bmatrix}  \left\langle w_1, ~ \frac{\partial \Phi}{\partial x_i}(x) \right\rangle_{\cH} , \dots , \left\langle w_m, ~ \frac{\partial \Phi}{\partial x_i}(x) \right\rangle_{\cH}  \end{bmatrix}^\top \in \Rr^m \, . %
 \end{align*}
Finally, let us equip $\cH^m$  with a norm defined by $\|w\|_{\cH^m}^2 =  \sum_{i=1}^m \|w_i\|_\cH^2$.

Let $\Gamma(w, x) := L_{f(x), C_V} \left( \|w(x)\|_2, \mu_0(x) - \frac{1}{\rho}  {\rm div} \left(B(x) w(x) \right)  \right) $. %
A non-parametric version of problem~\eqref{eqn:decouple} where $w$ is represented in $\cH^m$ can be rewritten as a stochastic optimization problem of the form:
\begin{align*} (\text{P}(\cH)) \quad\qquad\qquad\qquad\qquad
\inf_{w\in \cH^m} \ \ \Ee_{x \sim \cU(\cX)} ~ \Big\{ \Gamma(w, x) \Big\} \, \, . \hspace{12em}  %
\end{align*}
Let  $w_0=0$  and $\gamma >0$. We define the following single-pass stochastic gradient descent iterations on (P($\cH$)): at each time step $n \geq 1$, sample~$x^{(n)}$ uniformly on~$\cX$ and make the update:
\begin{align} \label{eqn:primaliter}
    w_{n} = w_{n-1} - \gamma \frac{\partial \Gamma}{\partial w}(w_{n-1}, x^{(n)}) \, ,
\end{align}
where $\frac{\partial \Gamma}{\partial w}(w_{n-1}, x^{(n)}) \in \partial_w \Gamma(w_{n-1}, x^{(n)})$ is a  subgradient of $\Gamma(w_{n-1}, x^{(n)})$. The following proposition ensures that the stochastic subgradients remain uniformly bounded regardless of the magnitude of $\| w_n\|_\cH$, and enables the use of a standard stochastic optimization result to obtain a convergence result.

\begin{proposition} \label{prop:boundedgrad}
    For any $x$, $\Gamma(\cdot, x)$ is convex. Furthermore, there exists a constant $D \geq 0$ depending on the RKHS~$\cH$ and on the control problem such that for almost any $x$:%
    \begin{align*}
        \sup_{w \in \cH^m} ~\left\|  \frac{\partial \Gamma(w, x)}{\partial w} \right\|_{\cH^m} \leq D \, .
    \end{align*}
\end{proposition}

\begin{remark}
    In order to compute the gradient of $\Gamma(\cdot, x)$ for a given stochastic sample $x$,  we assume an access the evaluation of $f$, $a$, $B$ and $\nabla B$ at this $x$, recalling that by Rademacher's theorem, $B$ is differentiable almost everywhere. The access to $\nabla B$ is specific to the primal approach, and is not required in the dual (Section~\ref{sec:dual}) and  primal-dual (Section~\ref{sec:primaldual}) approaches, which only need accesses to evaluations of $f$, $a$ and $B$.
\end{remark}

\begin{theorem} \label{thm:primal}
Assume that $a(x) \equiv 0$, and that Assumptions \ref{hypo:regularity}-\ref{hypo:alpha} hold. %
     Let $n \geq 1$, define $\bar w_n = \frac{1}{n} \sum_{k=0}^{n-1} w_k$ and $G : w \mapsto \Ee_x [\Gamma(w, x)] $. Then we have:  %
     \begin{align} \label{eqn:convG}
     \Ee [G(\bar w_n)]   \leq \inf_{w \in \cH^m}  \left\{ G(w) + \frac{\|w\|_{\cH^m}^2}{2 \gamma n} \right\} + \frac{\gamma D^2}{2 } \, . 
\end{align} %
In particular, assume that  $\bar \cH = H^1(\cX)$, and  set $\gamma \propto \frac{1}{\sqrt{n}}$, then: %
        \begin{align*}
            \lim_{n \to +\infty} \Ee [G(\bar w_n)]  = \inf_{w \in (H^1(\cX))^m} G(w) \, .
        \end{align*} %
\end{theorem}

\subsection{Finite-dimensional implementation} Similarly to the implementation of the dual, the infinite-dimensional SGD iterations~\eqref{eqn:primaliter} of the primal problem can be represented in finite-dimension. The number of coefficients to represent $w \in \cH^m$ is~$m$ times larger compared to $V \in \cH$ because~$w$ is homogeneous to a controller whereas~$V$ was a scalar-valued function. 
\begin{proposition} \label{prop:implemprimal}
  Let $N \geq 1$.  There exist a matrix $\alpha \in \Rr^{N \times m}$ and a tensor $\beta \in \Rr^{N \times d \times m}$ such that %
\begin{align*}
   \forall n \in \{1, ... , N \}, \quad w_n = \sum_{k=1}^{n} (\alpha_{k, \cdot} \otimes \Phi(x^{(k)}) ) + \sum_{k=1}^{n} \sum_{i=1}^d \left( \beta_{k, i, \cdot} \otimes \frac{\partial \Phi}{\partial x_i}(x^{(k)}) \right) \, , %
\end{align*}
where $\otimes$ denotes the outer product defined, for $v \in \Rr^m$, $h \in \cH$ by: \begin{align*}
    v \otimes h = [v_1 h, \dots, v_m h]^\top \in \cH^m \, .
\end{align*}
In particular, for any $y \in \cX$, $j \in \{1,...,d\}$ and $n \in \{1, ... , N \}$:%
\begin{align*}
    w_n(y) & = \sum_{k=1}^{n} K(x^{(k)}, y) \alpha_{k, \cdot} + \sum_{k=1}^{n} \sum_{i=1}^d \frac{\partial K}{\partial x_i}(x^{(k)}, y) \beta_{k, i, \cdot}  %
    \\
    \frac{\partial w_n}{\partial y_j}(y) &= \sum_{k=1}^{n} \frac{\partial K}{\partial y_j}(x^{(k)}, y) \alpha_{k, \cdot}+ \sum_{k=1}^{n} \sum_{i=1}^d \frac{\partial^2 K}{\partial x_i \partial y_j}(x^{(k)}, y) \beta_{k, i, \cdot}  \, . %
\end{align*}
\end{proposition}

\begin{proof}
   This directly results from the computation of the partial derivative of~$\Gamma$ with respect to~$w$, as derived in the proof of Proposition~\ref{prop:boundedgrad}.
\end{proof}

\subsection{Link with policy evaluation} \label{sec:poleval} In this subsection, we get back to considering  generic control-affine dynamics $g(x, u) = a(x) + B(x)u$. We draw a link between the primal problem~\eqref{eqn:Ppen} and the policy evaluation problem --- to borrow the vocabulary of reinforcement learning --- aiming to evaluate the performance of a given feasible controller, as defined in~\eqref{eqn:poleval}. More specifically, for sufficiently regular controllers, their performance is upper-bounded by the value of the primal function, meaning that if the value of the primal converges to the optimum, so does the performance of the policy.

Let $u_c : \cX \to \Rr^m$ a candidate controller. We define its corresponding value function $V^{u_c}$, for any $x\in \cX$, by:
\begin{align*}
     V^{u_c}(x)  =  \rho \int_0^{+\infty} e^{-\rho t}(f(x_t)+ R(u_c(x_t)) \,  \mathrm d t   \,  \text{ where } \dot x_t = g(x_t, u_c(x_t)) \text{ and } x_0 = x \, , 
\end{align*}
so that $\int_\cX V^{u_c}(x)  \mu_0(x) \, \mathrm d x = \text{Perf}(u_c)$.
The value function $V^{u_c}$ is a viscosity solution of the following Bellman equation:
\begin{align*}
    \forall x \in \X,  \ \ - V^{u_c}(x) + f(x) +    R(u_c(x)) + \frac{1}{\rho} \nabla V^{u_c}(x) ^\top (a(x) + B(x) u_c(x)) = 0 \, .
\end{align*}
We can define primal and dual linear programs analogous to~\eqref{P} and~\eqref{D} for the policy evaluation problem:
\begin{align*} & (\text{P}(u_c)) \quad\qquad\qquad \inf_{\mu: \cX \to \Rr_+} \ \int_{\X} \left( f(x) + R(u_c(x)) \right) \mu(x) \,  \mathrm dx  \\
& \nonumber \mbox{ such that } \forall x \in \X, \ \mu_0(x) = \mu(x) + \frac{1}{\rho} {\rm div} \left(a(x) \mu(x) + B(x)  u_c(x) \mu(x) 
\right) . \\
& (\text{D}(u_c)) \quad\qquad\qquad \qquad \sup_{V: \cX \to \Rr} \ \    \int_{\cX} V(x) \mu_0(x) \, \mathrm dx \\
& \nonumber \text{ such that } \forall x \in \cX , \  -  V(x) +  f(x) + R(u_c(x)) + \frac{1}{\rho} \nabla V(x) ^\top (a(x) + B(x) u_c(x)) \geq 0 \, .
\end{align*}
Let $C_V := \|f\|_\infty + \frac{C_u^q}{q}$.
We further define the following penalized version of $\text{P}(u_c)$, and prove below that it is an exact penalty for appropriately bounded controllers. 
\begin{align*} & (\text{P-P}(u_c)) \quad \inf_{\mu: \cX \to \Rr_+} \ \int_{\X} \left[ f(x) + R(u_c(x)) \right] \mu(x) \,  \mathrm dx  \\
& \qquad\qquad\qquad~~~\quad + {C_V} \int_\cX  \left[\mu_0(x) - \mu(x) - \frac{1}{\rho} {\rm div} \Big( \left(
a(x)  + B(x) u_c(x) \right) \mu(x)  \Big) \right]_+  \mathrm dx \, . 
\end{align*}

\begin{lemma} \label{lemmapolicyeval}
Assume that $u_c : \cX \to \Rr^m$ is uniformly Lipschitz continuous. %
Then \begin{align*}
    \textnormal{val}(\textnormal{P}(u_c)) = \textnormal{val}(\textnormal{D}(u_c)) = \textnormal{Perf}(u_c) \, .
\end{align*}
    If, in addition, $u_c$ is such that $\sup_{x \in \cX} \| u_c(x) \|_2 \leq C_u$, %
    then for any $x\in \cX$, we have  $0 \leq V^{u_c}(x) \leq C_V $, and
    \begin{align*}
         \textnormal{val}(\textnormal{P-P}(u_c)) = \textnormal{val}(\textnormal{P}(u_c)) = \textnormal{Perf}(u_c) \, .
    \end{align*}
\end{lemma}

\begin{remark}
    We recall that $C_u$ is an upper-bound on $\sup_{x\in \cX} \|u^*(x)\|_2$.
    $C_V$ is chosen to ensure that the boundedness assumption holds for $u^*$. However, potential discontinuities of controllers are a common phenomenon in optimal control~\cite{kamien2012dynamic}, as in the simple example presented in Section~\ref{subsec:ex}.%
\end{remark}

Let us denote by $J$ the objective function of the optimization problem \eqref{eqn:Ppen}: \begin{align} \label{eqn:J}
        J(w, \mu) &:=  \int_{\X} \left[ f(x) + R\left( \frac{w(x)}{\mu(x)} \right) \right] \mu(x) \,  \mathrm dx \\
     \nonumber   &\qquad+ C_V \int_\cX  \left[\mu_0(x) - \mu(x) - \frac{1}{\rho} {\rm div} \left(
a(x) \mu(x) + B(x) w(x) \right) \right]_+  \mathrm dx \, .
\end{align}
The following proposition shows that the performance suboptimality of the controller obtained from $w/\mu$ can be controlled by the suboptimality of the primal objective~$J(w, \mu)$, under the regularity conditions of Lemma~\ref{lemmapolicyeval}.

\begin{proposition} \label{prop:subopt}
Consider a candidate pair $(w, \mu)$, where $\mu: \cX \to \Rr_+$ and ${w:\cX \to \Rr^m}$. If $(w, \mu)$ is such that $u_c := w/\mu$ is continuously differentiable and $\sup_{x\in\cX} \|u_c(x)\|_2 \leq C_u$,  then 
\begin{align*}
    \textnormal{Perf}(u_c) - \textnormal{Perf}(u^*) \leq J(w, \mu) - J(w^*, \mu^*) \, ,
\end{align*}
where $\mu^*$ is the optimal occupation measure and $w^*$ is defined as in \eqref{eqn:wstar}. 
\end{proposition}

Contrary to the dual approach, which optimizes $V$ but gives little guarantee on the resulting controller derived from $\nabla V$, the primal objective optimizing $(w, \mu)$ is directly linked to the performance of the derived policy $w/\mu$. Throughout  primal  (and primal-dual, as defined in Section~\ref{sec:primaldual}) optimization, evaluating $J$ can be used as a performance metric, providing a stopping criterion without having to directly simulate the performance of the current policy. 

A limitation of Proposition~\ref{prop:subopt} is the double requirement of  boundedness and smoothness for the candidate controller. When solving~(P($\cH$)), each iterate $w \in \cH^m$ is continuously differentiable, hence uniformly bounded; but $\mu$ is obtained according to Lemma~\ref{lemma2} by 
\begin{align*}
    \mu(x) = \max \left\{ \left(\frac{q-1}{q}\right)^{1/q} f(x)^{-1/q} \|w(x)\|_2, ~ \mu_0(x) - \frac{1}{\rho} \text{div}(B(x)w(x)) \right\} \, ,
\end{align*}
which  might touch zero if $w(x)=0$,
and
cannot be guaranteed to be smooth with only Assumption~\ref{hypo:regularity}. %
This issue could be fixed by adding a strongly convex regularization, and bounding $\mu$ away from zero in Lemma~\ref{lemma2}, at the cost of an additional bias. Another option is  to evaluate a smoothed\footnote{Smoothing can be achieved by computing $\tilde \mu_\epsilon (x) = \Ee[ \mu(x + \epsilon y)] $, with $y \sim \cN(0, I_d)$ and $\epsilon>0$ a small parameter. In our practical implementations, we did not find this smoothing to be necessary.
} and clipped version $\tilde \mu_\epsilon$ instead of $\mu$.%

\subsection{Numerical example} We use the same example as in Section~\ref{subsec:ex}, whose simple dynamics $\dot x = u$ is such that $a(x) \equiv 0$. We implement the iterations~\eqref{eqn:primaliter}, with the same RKHS as above, and compute the averaged iterates $\bar w_n$ with $\rho=2$, $C_V = 2$  --- so that $C_V \geq \|f\|_\infty+\sup_{x \in \cX}\|u^*(x)\|_2^2/2$ and Proposition~\ref{prop:subopt} holds --- and constant step size $\gamma=2 \times 10^{-3}$.  The results are shown in Figures~\ref{fig:primalperf} and~\ref{fig:primal}. 

In Figure~\ref{fig:primalperf}, we see that $G(\bar w_n)$ converges from above to the optimal value $\int_\cX V^* \mathrm d\mu_0$, and in Figure~\ref{fig:primal} that $\bar w_n$ precisely approximates~$w^*$, which is smoother than $V^*$ in this example. At each time step, a candidate controller $u_n$ can be derived from the current $\bar w_n$ by first computing $\mu_n$ using Lemma~\ref{lemma2}, and taking $u_n = \bar w_n / \mu_n$.  As we see in Figure~\ref{fig:primalperf}, this controller is near-optimal after only 5k iterations (vs around 10-15k iterations in the dual formulation). But more importantly, computing the primal function $G$ at a given $\bar w_n$ provides an explicit control over the performance of the corresponding controller, without having to evaluate it with trajectory simulations.  Note that projecting and smoothing the iterates~$u_n$ was not necessary in this experiment, as their magnitude did not grow. In fact, as can be seen in Figure~\ref{fig:primalperf}, in this primal example where $\mu_n$ is optimal for a given $\bar w_n$, $G$ is not only an upper-bound but also a good estimator of $\text{Perf}(u_n)$.   For this reason, the primal might be considered better suited than the dual for optimizing the policy.

\begin{figure}
    \centering
    \includegraphics[width=0.45\linewidth]{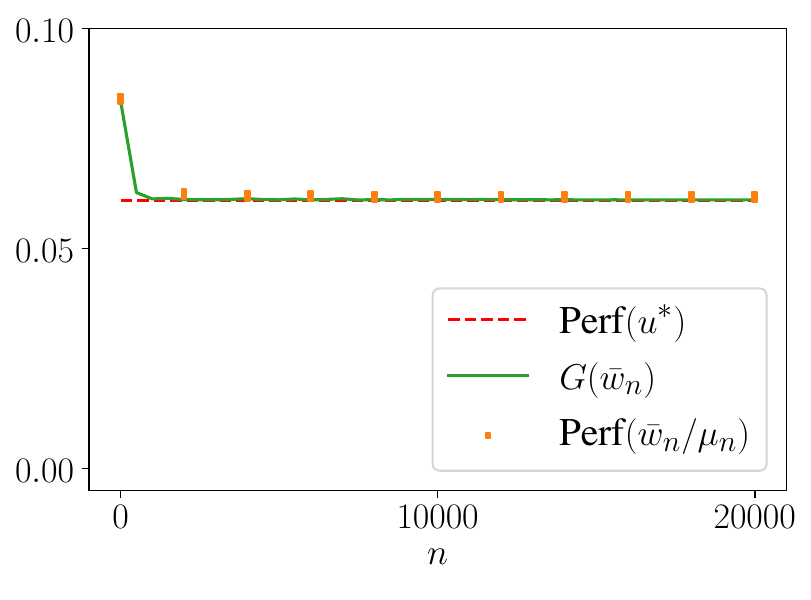} \vspace{-0.4cm}
    \caption{Evolution of the primal cost, along with the performance of the induced controller, over the iterations. Note that the performance is evaluated up to some finite precision, due to discrete-time approximations in the simulated trajectories.}
    \label{fig:primalperf}
\end{figure}

\begin{figure}
    \centering
    \includegraphics[width=0.98\linewidth]{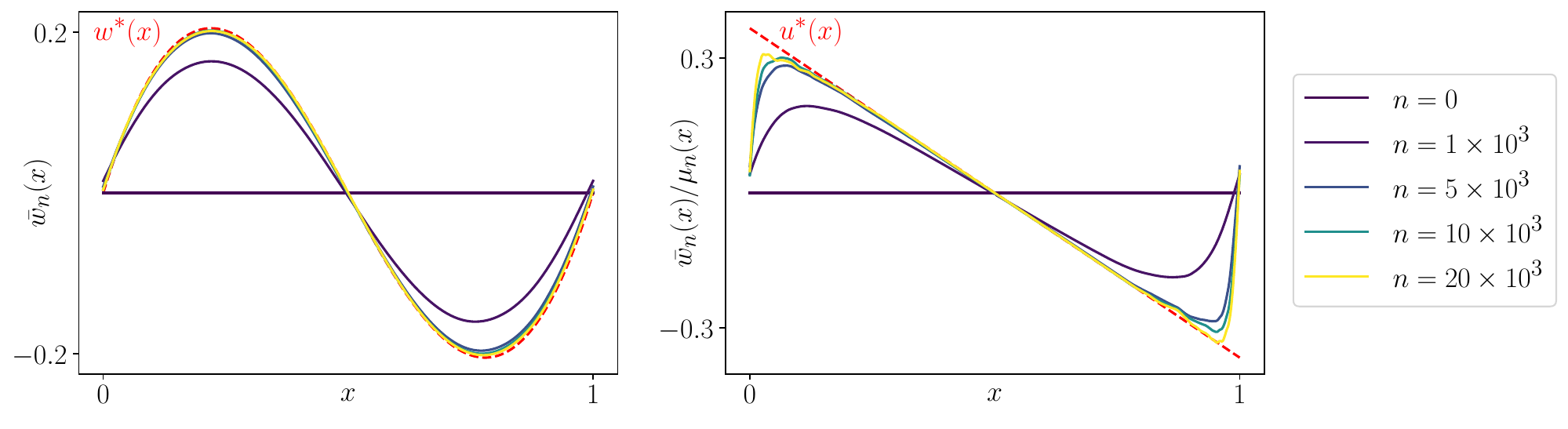} \vspace{-0.3cm}
    \caption{$\bar w_n$ obtained from iterations~\eqref{eqn:primaliter}, and the corresponding candidate controller obtained from $\bar w_n/\mu_n$, where $\mu_n$ is computed using Lemma~\ref{lemma2}, are plotted for growing values of~$n$, along with $w^*$ and $u^*$ in red dotted lines.}
    \label{fig:primal}
\end{figure}

\section{The primal-dual problem} \label{sec:primaldual} The primal formulation has some advantages over the dual one, yet, it is restricted to dynamics such that $a(x) \equiv 0$. In this section, we develop a primal-dual approach which generalizes the primal approach to any control-affine dynamics.

\subsection{Towards a tractable primal-dual formulation}

Starting from the primal problem~\eqref{P}, using strong duality and  the variable $w(x) = \int_{\Rr^m} \mu(x, u)u\, \mathrm du $ as defined in~\eqref{eqn:wdef}, such that at optimality $w^*=u^* \mu^*$, we can define the equivalent primal-dual problem:
\begin{align}
\label{eqn:pd0} \sup_{V \in \cC^1(\cX)} \ \ & \inf_{\mu: \X  \to \rb_+, \ w: \X \to \rb^m }  \ \int_{\X}  \Big[ f(x) +  R \Big( \frac{ w(x)}{\mu(x)} \Big) \Big] \mu(x) \, \mathrm dx
\\
&  ~~+   \int_\X V(x) \left( \mu_0(x) - \mu(x) - \frac{1}{\rho} {\rm div} ( \mu(x) a(x) + B(x)w(x)) \right) \mathrm  dx \, . \nonumber
\end{align}
As previously noted in Section~\ref{sec:primalstoch} for problem~\eqref{eqn:Ppen}, this formulation is not separable in $\mu(x)$, hence not directly compatible with SGD, except if $a(x) \equiv 0$. A simple manipulation, enabled by the introduction of the dual variable~$V$, will solve this issue. Indeed, for periodic   differentiable functions $v:\cX \to \Rr$ and $z:\cX \to \Rr^d$ on the torus, Green's first identity gives
\begin{align} \label{eqn:green}
     \int_\cX v(x) \ \textnormal{div}(z(x)) \, \mathrm dx = - \int_\cX \nabla v(x)^\top z(x) \, \mathrm dx \, .
\end{align}
Applying~\eqref{eqn:green} in~\eqref{eqn:pd0}, we then obtain the more tractable primal-dual formulation:
\begin{align} \label{eqn:PD}
\sup_{V \in \cC^1(\cX)}  \inf_{\mu: \X  \to \rb_+, \ w: \X \to \rb^m } \ \int_{\X} & \Big[ f(x) +  R \Big( \frac{ w(x)}{\mu(x)} \Big) \Big] \mu(x) \, \mathrm  dx
+   \int_\X V(x) \left( \mu_0(x) - \mu(x) \right) \mathrm dx \notag \\ &\qquad\qquad\quad+ \frac{1}{\rho} \int_\cX \nabla V(x)^\top \left(  \mu(x) a(x) + B(x)w(x) \right) \mathrm dx \, . \tag{PD}
\end{align}
As a sanity check, the minimum with respect to $\mu$ and $w$ is separable in $x$, and can then be obtained in closed form through the following lemma. 
\begin{lemma}[Fenchel conjugate of perspective functions] Let $\varphi: \rb^m \to \Rr$ a convex function and let $(\alpha, \beta) \in \Rr^m \times \Rr$. Then:
\begin{align*}
\inf_{t \geqslant 0, \ v \in \rb^m} \ t \varphi \big( \frac{v}{t} \big) + \alpha^\top v + \beta t = 0  \mbox{ if } \varphi^\ast(-\alpha) \leqslant \beta \mbox{ and } -\infty \mbox{ otherwise}.
\end{align*}
\end{lemma}
We thus obtain the constrained problem
\begin{align*}
&\sup_{V \in \cC^1(\cX)}  \     \int_\X V(x) \mu_0(x) \, \mathrm dx \mbox{~~~ such that } \\
& \qquad \forall x \in \X, \ R^\ast
\left( - \frac{1}{\rho} B(x)^\top \nabla V(x) \right)  \leqslant f(x) -  V(x) + \frac{1}{\rho} \nabla V(x)^\top a(x) \, ,
\end{align*}
which is exactly the dual problem~\eqref{D}.  Alternatively, starting from~\eqref{eqn:PD}, if we minimize with respect to $\mu$ only, the problem is also separable in $x$, and reduces to
\begin{align}
  \label{eqn:pdk}  & & \sup_{V \in \cC^1(\cX)} \ \inf_{w: \X \to \rb^d } ~& 
   \int_\X V(x)  \mu_0(x) \, \mathrm dx +  \frac{1}{\rho}\int_\X \nabla V(x)^\top B(x) w(x) \, \mathrm dx  \\
   & & & + \int_{\X}  S\left(w(x), f(x) -   V(x) + \frac{1}{\rho} \nabla V(x)^\top a(x) \right)  \mathrm dx \, , \nonumber
\end{align}
where 
\begin{align*}
    \forall  (\alpha, \beta) \in \Rr^m \times \Rr, \quad  \ds S(\alpha, \beta) :=
\inf_{\mu \geqslant 0 } \Big\{\beta  \mu  +  R \Big( \frac{\alpha}{\mu} \Big) \mu \Big\} \, .
\end{align*}
One can easily check that $S$ is finite if and only if $\beta \geq 0$. In~\eqref{eqn:pdk}, this imposes the constraint $ f(x) -   V(x) + \rho^{-1} \nabla V(x)^\top a(x)\geqslant 0$ on $V$, for all $x\in \cX$. While the constraint is feasible (consider, \textit{e.g.}, the value function corresponding to the controller $u \equiv 0$, for which the constraint is equal to zero everywhere), %
enforcing this infinite-dimensional constraint for each iterate of a stochastic optimization algorithm is not straightforward. %

If Assumption~\ref{hypo:mustarbounded} holds, which we assume from now on, $\forall x\in\cX$, $\mu^*(x) \leq M$, and we can replace $S$ in~\eqref{eqn:pdk} by $S_M$ defined by:
\begin{align}
    \forall  (\alpha, \beta) \in \Rr^m \times \Rr, \quad  \ds  S_M(\alpha, \beta) =
\inf_{\mu \in [0, M]} \Big\{\beta  \mu  +  R \Big( \frac{\alpha}{\mu} \Big) \mu \Big\} \, .
\end{align}
Assumption~\ref{hypo:mustarbounded} has a regularization effect: $S_M$ is a smoothed version of $S$, and removes the strong constraint on $V$. $S_M$ and its derivatives can be computed in closed form as follows. In particular, $S_M$ remains finite for any $\beta \in \Rr$.

\begin{lemma} \label{lemma3} Let $(\alpha, \beta) \in \Rr^m \times \Rr$. Then
\begin{align*}
    S_M(\alpha, \beta) = 
\left\{\begin{array}{ll}
 \beta M + \frac{\|\alpha\|^q_2}{q} M^{1-q} & \mbox{ with }~ \mu^\ast =  M ~~~ \mbox{ if } ~ \beta \leq \frac{q-1}{q} \frac{\| \alpha\|^q_2}{M^q}\\
\left( \frac{q}{q-1}\right)^{\frac{q-1}{q}} \beta^{\frac{q-1}{q}} \| \alpha\|_2 & \mbox{ with }~  \mu^\ast =  \left(\frac{q-1}{q}\right)^{1/q} \beta^{-1/q} \|\alpha\|_2 ~~ \mbox{ else} \, .\end{array}\right.
\end{align*} 
Moreover if $\alpha \neq 0$, we have
\begin{align*}
    \nabla_\alpha  S_M (\alpha, \beta) = \|\alpha\|_2^{q-2} \alpha {(\mu^*(\alpha, \beta))}^{1-q} \quad\text{~ and ~}\quad
    \frac{\partial  S_M}{\partial \beta}(\alpha, \beta) = \mu^*(\alpha, \beta)  \, ,
\end{align*}
where $\mu^*(\alpha, \beta) := \mathbf{1}_{\beta \leq 0} M + \mathbf{1}_{\beta > 0}\min \left\{\left(\frac{q-1}{q}\right)^{1/q} \beta^{-1/q} \|\alpha\|_2 \, , \,  M \right\} \in [0, M]$. If $\alpha = 0$, one can select the subgradient $ 0 \in \partial_\alpha S_M(\alpha, \beta)$ and supergradient $M \in \partial_\beta^+ S_M(\alpha, \beta)$.

\end{lemma}

\begin{remark}
    Contrary to the inner minimization function $L$ appearing in the primal, $S_M$ does not have uniformly bounded derivatives. This additional difficulty, along with the bilinear terms appearing in the objective of~\eqref{eqn:PD}, will require the use of projections of the primal and dual iterates in the next section.
\end{remark}

\subsection{Stochastic optimization on the primal-dual} Similarly to the dual and primal approaches, we represent both the primal $w:\cX \to \Rr^m$ and dual variable $V:\cX \to \Rr$ as elements of the RKHS~$\cH$.  Starting from \eqref{eqn:pdk}, we can build the following stochastic optimization problem, which is convex-concave in $(w, V)$:
\begin{align*}  (\text{PD}(\cH)) \quad~~\,&\qquad\qquad\quad
\sup_{V \in \cH} \ \inf_{w\in \cH^m} \ \ \Ee_{x \sim \cU(\cX)} ~ \Big\{ \Psi(V, w, x) \Big\} \, \, , \hspace{10em}   \\
&\text{where ~~~\quad}   \Psi(V, w, x) :=  V(x) \mu_0(x) + \frac{1}{\rho} \nabla V(x)^\top B(x) w(x) \\
& \qquad\qquad\qquad\qquad\qquad~~ +  S_M\left(w(x), f(x) -   V(x) + \frac{1}{\rho} \nabla V(x)^\top a(x) \right) \, .
\end{align*}
Let us define\footnote{With a slight abuse of notations, we denote as gradients objects which are actually sub- and supergradients if $w(x)=0$, in which case we use the measurable selection defined in Lemma~\ref{lemma3}.}, for any $V \in \cH$, $w \in \cH^m$ and $x \in \cX$:
\begin{align*}
    \Pi(V, w, x) := \frac{\partial \Psi(V, w, x)}{\partial w} \in \cH^m ~~; ~~\quad \Delta(V, w, x) := - \frac{\partial \Psi(V, w, x)}{\partial V}  \in \cH \, ,
\end{align*}
$\Omega(V, w, x) := [\Pi(V, w, x), \Delta(V, w, x)]^\top$, and the variational inequality functional \begin{align*}
     \Omega(V, w) := \begin{pmatrix}
    \Ee_{x \sim \cU(\cX)}[\Pi(V, w, x)] \\ \Ee_{x \sim \cU(\cX)}[\Delta(V, w, x)]
\end{pmatrix} \in \cH^m \times \cH \, .
\end{align*} Since $\Psi$ is convex in $w$, concave in $V$, $\Omega$ is monotone mapping on $Z := \cH^m \times \cH$, \textit{i.e.},
\begin{align*}
     \forall z, z' \in Z, \quad \langle \Omega(z) - \Omega(z') , z - z' \rangle \geq 0 , 
\end{align*}
where $\langle \cdot, \cdot\rangle$ is the natural inner product on the Cartesian product space $Z$. Solving~($\text{PD}(\cH)$) is equivalent to finding a solution to the variational inequality \begin{align} \label{eqn:VI}
    \text{find } z^* \in Z ~~\text{ such that } ~~ \forall z \in Z, ~ \langle F(z), z^* - z \rangle \leq 0 \, .
\end{align}
To solve~\eqref{eqn:VI}, we have access to stochastic oracles $z \mapsto [\Pi(z, x_n), \Delta(z, x_n)]^\top$, where the $x_n$ are sampled i.i.d., uniformly on $\cX$. Each oracle satisfies the following property.
\begin{proposition}\label{prop:projection} There exists a constant $c_\Omega > 0$ depending only on the control problem and the RKHS, such that for any $V \in \cH$, $w \in \cH^m$, $x \in \cX$ and $r_V, r_w  \geq 0$ such that $\|V\|_\cH \leq r_V$ and $\|w\|_{\cH^m} \leq r_w$, then %
\begin{align*}
    \|\Pi(V, w, x)\|_{\cH^m} \leq c_\Omega(1 + r_V + r_w)~~\text{and}~~~  \|\Delta(V, w, x)\|_{\cH} \leq c_\Omega(1+ r_V + r_w) \, . 
\end{align*} 
\end{proposition} 

We use the stochastic mirror-prox algorithm\footnote{Although this algorithm has been introduced for finite-dimensional spaces, its convergence analysis being dimension-independent, it directly generalizes to Hilbert spaces.} introduced by~\cite{juditsky2011solving}. Let $r >0$, we define $Z_r := \{ z=(w, V) \in Z, \text{ s.t. } \|w\|_{\cH^m } \leq r \text{ and } \|V\|_\cH \leq r\}$.  To ensure that the gradients remain bounded, we will use a proximal operator that projects the iterates on this set, and then take $r \to +\infty$ to ensure universal approximation of the primal and dual variables~$w^*$ and~$V^*$. Let us define the following proximal operator:
\begin{align*}
    \textnormal{Prox}^{Z_r}_{z_0}(z) = \argmin_{\zeta \in Z_r} \left\{ \frac{1}{2} \| \zeta \|^2_{\cH^m \times \cH} + \langle z - z_0, \zeta \rangle \right\} .
\end{align*}
This is a separable operation in $(w, V)$ that can be computed as follows:
\begin{align}
    \text{Prox}_{(w_0, V_0)}^{Z_r}(w, V) = \begin{pmatrix}
    \text{Proj}_{\cB_r(\cH^m)} (w_0 - w) \\
    \text{Proj}_{\cB_r(\cH)} (V_0 - V)
\end{pmatrix} ,
\end{align}
where $\cB_r(\cH^m)$ and $\cB_r(\cH)$ are respectively the balls of radius $r$ of $\cH^m$ and $\cH$ centered at the origin.

Let $\zeta_0 \in \cH^m \times \cH$, $\gamma >0$. At each time step $n \geq 1$, the stochastic mirror-prox algorithm samples~$x_n$ and $x_n'$  uniformly and independently on~$\cX$ and makes the two updates:
\begin{align} \label{eqn:primaldualiter}
    z_n & = \text{Prox}_{\zeta_{n-1}}^{Z_r}( \gamma \Omega(\zeta_{n-1}, x_n)) \\
    \zeta_n & = \text{Prox}_{\zeta_{n-1}}^{Z_r} (\gamma \Omega(z_n, x_n')) \, . \label{eqn:primaldualiter2} 
\end{align}
Let us finally define \begin{align*}
    \bar \Psi_r (w) &= \max_{ \| V \|_\cH \leq r} \Ee_{x \sim \cU(\cX)} [\Psi(V, w, x)] \\ \ubar{\Psi}_r(V) &= \min_{ \| w \|_{\cH^m} \leq r} \Ee_{x \sim \cU(\cX)} [\Psi(V, w, x)] \, ,
\end{align*}
along with the following error metric on $Z_r \subset \cH^m \times \cH$, for a candidate couple $(\tilde V, \tilde w)$, sometimes called the \textit{duality gap}: \begin{align} \label{eqn:errr}
    \cE_r(\tilde V, \tilde w) = \bar \Psi_r(\tilde w) - \ubar{\Psi}_r(\tilde V) \, .
\end{align}
In the following theorem, we prove that for an appropriate choice of $\gamma$ and $r$, the error from~\eqref{eqn:errr} converges to 0 when $n \to + \infty$.%
\begin{theorem} \label{thm3} Assume that Assumptions \ref{hypo:regularity}-\ref{hypo:alpha}-\ref{hypo:mustarbounded} hold. %
     Let $n \geq 1$, define $\bar w_n = \frac{1}{n} \sum_{k=1}^{n} w_k$ and $\bar V_n = \frac{1}{n} \sum_{k=1}^{n} V_k$, where $w_k$ and $V_k$ are the first and second component of $z_k$, obtained from the iterations~\eqref{eqn:primaldualiter}--\eqref{eqn:primaldualiter2}, with parameters $r \geq 1$, and $\gamma =\frac{r}{14 c_\Omega\sqrt{n}}$. Then we have:
    \begin{align*}
        \Ee[\cE_r(\bar V_n, \bar w_n) ] \leq 110 c_\Omega \frac{r^3}{\sqrt{n}} \, .
    \end{align*}
    In particular, if we set $r_n \propto n^{\alpha}$ for some $\alpha \in (0, 1/6)$, then: %
    \begin{align*}
         \lim_{n \to +\infty} \Ee[\cE_{r_n}(\bar V_n, \bar w_n) ] = 0 \, .
    \end{align*} %
\end{theorem}

\begin{remark}
    If we assume that $V^*$ and $w^*$ belong to the RKHS (hence have finite norm),  the projection radius can be set constant, and the convergence rate is of order $O(n^{-1/2})$. This is similar, in the same situation, to the rates obtained for the dual\footnote{Although Theorem~\ref{thm:dual} provides a slower rate $O(n^{-1/3})$ in the general case $V^* \in L^2(\cX)$, if $V^*$ has bounded and known $\cH$-norm, averaged SGD benefits from the optimal rate $O(n^{-1/2})$.} and primal in the previous sections. Moreover, the constants of all three rates have a polynomial dependence in $1/\rho$, exploding as $\rho \to 0$. This is similar to classical RL algorithms  scaling with the horizon $1/(1-\gamma)$, where $\gamma$ is the discount factor~\cite{Sutton1998}.%
\end{remark}

Consider, for any $n \geq 1$, $(w_n, V_n)^\top:=z_n$. With a simple recursion using the expressions of $\Pi(V, w, x)$ and $\Delta(V, w, x)$ derived in the proof of Proposition~\ref{prop:projection}, we can prove that $w_n$ and $V_n$ can be represented in finite-dimension as follows.
\begin{proposition} \label{prop:implemprimaldual}
 For any $n \geq 1$, there exist $\tilde \alpha^{(n)} \in \Rr^{n \times m}$, $\alpha^{(n)} \in \Rr^n$, $\beta^{(n)} \in \Rr^{n \times d}$ and $\tilde \alpha_0, \alpha_0 \in \Rr$ such that 
\begin{align*}
   w_n &=  \tilde \alpha_0 [\zeta_0]_1 + \sum_{k=1}^{n} (\tilde \alpha^{(n)}_{k, \cdot} \otimes \Phi(x_k) )  \\
   V_n &= \alpha_0 [\zeta_0]_2 + \sum_{k=1}^{n}\alpha^{(n)}_{k} \Phi(x_k) + \sum_{k=1}^{n} \sum_{i=1}^d \beta^{(n)}_{k, i} \frac{\partial \Phi}{\partial x_i}(x_k) \, . %
\end{align*}
\end{proposition}

\subsection{Numerical example} As before, we use the same one-dimensional example with dynamics $\dot x = u$. We implement the iterations~\eqref{eqn:primaldualiter}--\eqref{eqn:primaldualiter2}, with step size~$\gamma$ initialized to $10^{-3}$ and decreased by $25 \%$ every $20,000$ steps. In this experiment, projecting the iterates was not necessary. For each $x\in \cX$ and at step $n$, $\mu_n(x)$ is obtained from $\bar V_n$ and $\bar w_n$ by \begin{align} \label{eqn:mu}
     \mu_n(x) := \argmin_{0\leq \mu \leq M} ~\left\{  \mu \left( f(x) -   \bar V_n(x) + \frac{1}{\rho} \nabla \bar V_n(x)^\top a(x) \right) + \frac{\|\bar w_n(x)\|_2^2}{2\mu} \right \} ,
\end{align}
with the closed-form expression derived in Lemma~\ref{lemma3}. Figure~\ref{fig:PD} shows the primal and dual iterates, along with the obtained  $\mu_n$. It approximates $\mu^*$ correctly, except around 0, where the value of $\mu_n$ has little influence on the objective as soon as $\bar V_n$ is near optimal, since $f- V^* \simeq 0$ in this area. From $\mu_n$, a natural candidate controller is $ \bar w_n/\mu_n$, displayed in Figure~\ref{fig:PD}, which closely approximates $u^*$. 

In Figure~\ref{fig:PDperf}, we show the evolution of the functions $J$ and $F_M$, respectively defined in Equation~\eqref{eqn:J} and Section~\ref{sec:sgddual}, across the iterations.  Because $a \equiv 0$ in this example, we can evaluate the primal objective~$G$ defined in Section~\ref{sec:sgdprimal} as well. Finally, we plot the performance of the controller $ \bar w_n/\mu_n$, which quickly becomes optimal, and remains upper-bounded by $J(\bar w_n, \mu_n)$. Note that in this example, the performance of the candidate controller $- \bar V'_n / \rho$ solely obtained from the dual variable $\bar V_n$ also obtains a near optimal performance, although with a slightly higher cost. However, no explicitly computable functional plays the role of $J$ to control the performance of this second controller.

\begin{figure}%
 \hspace{2.6cm}   \includegraphics[width=0.6\linewidth]{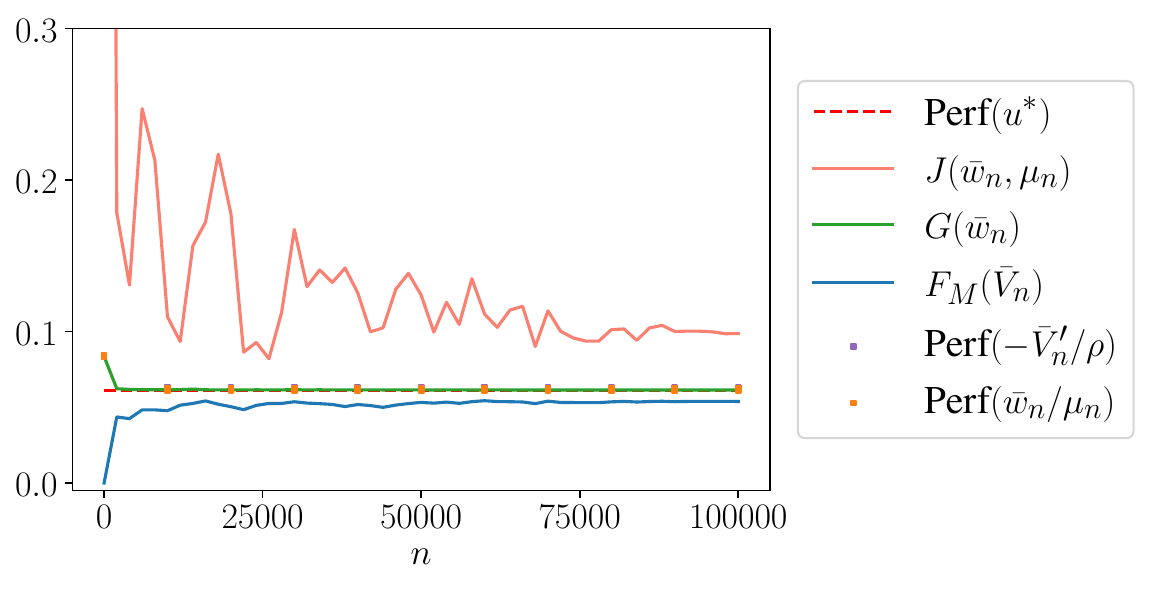} \vspace{-0.35cm}
    \caption{Evolution of functions $J$, $F_M$ and $G$, along with the performance of the controller computed by $\bar w_n/\mu_n$, over the iterations. Note that the performance is evaluated up to some finite precision, due to discrete-time approximations in the simulated trajectories.}
    \label{fig:PDperf}
\end{figure}

\begin{figure}%
   \hspace{1.6cm} 
\includegraphics[width=0.75\linewidth]{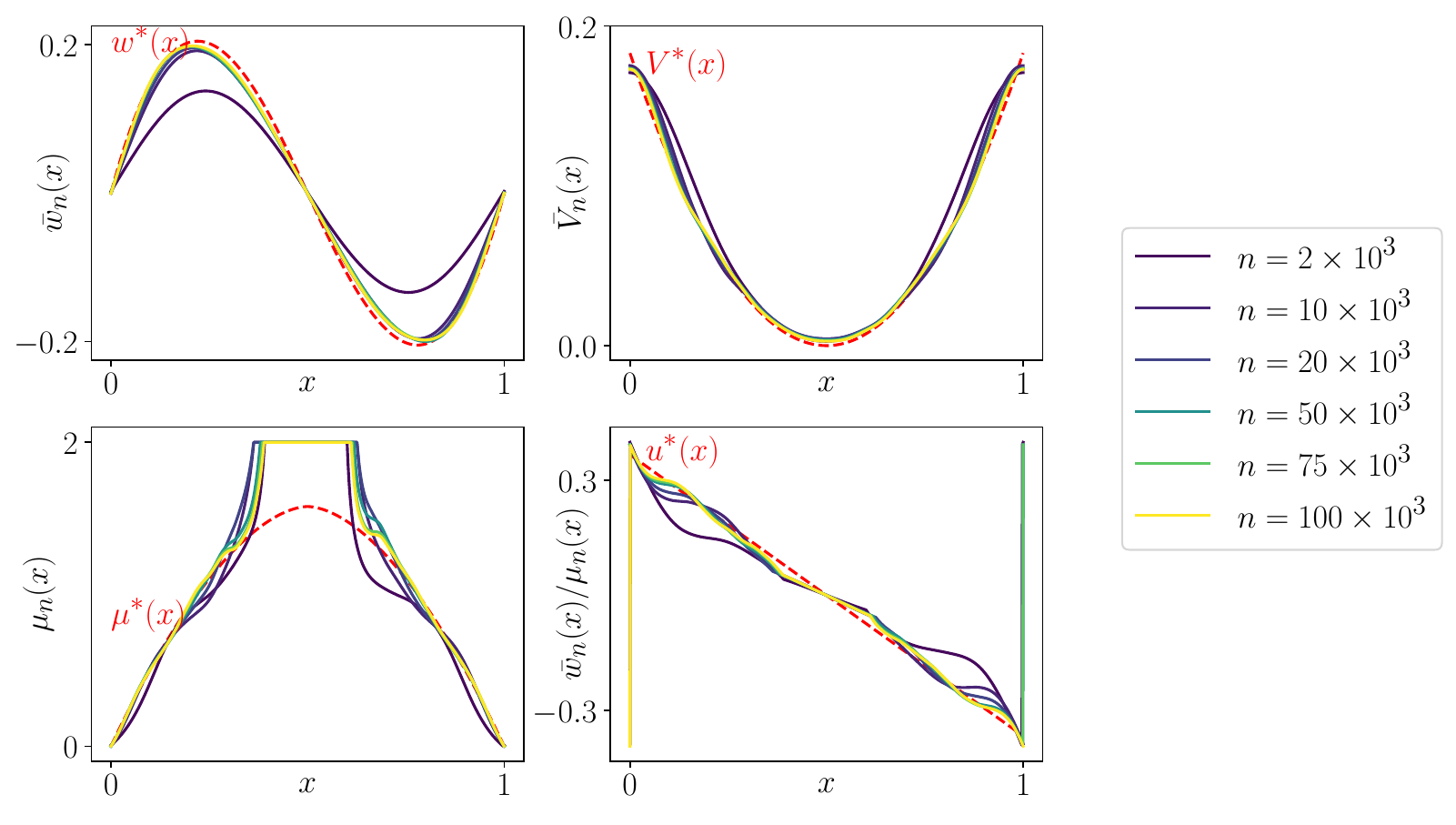} \vspace{-0.3cm}
    \caption{$\bar w_n$ and $\bar V_n$ obtained from iterations~\eqref{eqn:primaldualiter}--\eqref{eqn:primaldualiter2}, $\mu_n$ obtained by equation~\eqref{eqn:mu} and the corresponding candidate controller obtained from $\bar w_n/\mu_n$, where $\mu_n$ is computed using Lemma~\ref{lemma2}, are plotted for growing values of~$n$, along with $w^*$ and $u^*$ in red dotted lines.}
    \label{fig:PD}
\end{figure}

\subsection{Relation with the primal and dual approaches} \label{subsec:PDcomp} In such primal-dual optimization problems, the duality gap, as defined in Equation~\eqref{eqn:errr}, usually serves as a stopping criterion. However, this quantity cannot always be explicitly computed, but can be related to functions which we can be easily evaluated (namely $F_M$ and $J$, as defined in Sections~\ref{sec:sgddual} and~\ref{sec:poleval}), as follows.
\begin{proposition} \label{prop:gap}
   Assume that $q=2$. Then for any given candidate $\tilde V \in \cH$: \begin{align*}
       \inf_{w : \cX \to \Rr^m} \Ee_{x \sim \cU(\cX)} \Psi(\tilde V, w, x)  = F_M ( \tilde V) \, .
   \end{align*}
    For any $q \in (1, 2]$, and any given candidate $\tilde w \in \cH^m$,
   \begin{align*}
       \sup_{V:\cX \to [0, C_V]} \Ee_{x \sim \cU(\cX)} \Psi(V, \tilde  w, x) = \inf_{\mu : \cX \to [0, M]} J(\tilde w, \mu) \, .
   \end{align*}
    Moreover, if $a \equiv 0$:
  \begin{align*}
       \sup_{V:\cX \to \Rr} \Ee_{x \sim \cU(\cX)} \Psi(V, \tilde  w, x)  \geq G (\tilde w) \, .
  \end{align*}
\end{proposition}

Broadly speaking, this primal-dual approach can be seen as an extension of the primal one, where the necessity of $a \equiv 0$ in Problem~\eqref{eqn:Ppen} for variable separation is bypassed by the introduction of the dual variable $V$ and an adequate integration by parts. As such, no advantage over the primal is foreseen in the $a \equiv 0$ case, as confirmed in our numerical example by the much larger number of samples required for convergence in Figure~\ref{fig:PDperf} compared to Figure~\ref{fig:primalperf}. Still, this approach retains the advantage of the primal over the dual, of offering some control over the performance of the candidate suboptimal controllers.

The primal-dual approach can be compared to classical actor-critic algorithms. While the latter are two time-scale algorithms~\cite{konda2003onactor}, ours relies on a single time-scale mirror-prox scheme.  Specifically, actor-critic algorithms optimize both the value function~$V$ and the policy (or controller)~$u$, but~$u$ is updated at a much slower rate than~$V$, so that one can usually assume that the value function has converged when the policy is updated. Our method uses slightly different variables: the same value function $V$ as dual variable, but the couple~$(w, \mu)$ instead of the policy~$u$ as primal variable, which benefit from a convex-concave structure and are updated at the same rate.

\section{Behavior of the optimal occupation measure} \label{sec:mustar} We now briefly discuss under which conditions Assumption~\ref{hypo:mustarbounded} may or may not hold. In a nutshell, we show that, if $V^*$ is sufficiently smooth and $\rho$ is large enough, then $\mu^*$ cannot explode. Additionally, in the particular case $q<2$ and $a \equiv 0$, $\mu^*$ remains bounded for any $\rho$. Note that this analysis is specific to our choice of discounted setting, and divergences do not typically happen with finite-horizon control problems. In this Section, to first obtain regularity conditions on $V^*$, we now make the following assumption, strengthening Assumptions~\ref{hypo:regularity}-\ref{hypo:alpha}.

\begin{enumerate}[label={\bf A4}]
    \item
    \label{hypo:C2}
The functions $f,B,\alpha$ are $\cC^2$ and 
$\alpha$ is such that $a(x)=B(x)\alpha(x)$ for all $x\in\cX$.
Moreover, there exists $C_{\partial B} \geq 0$ such that 
$\partial_{\xi}B(x)\partial_{\xi}B(x)^{\top}\leq C_{\partial B}^2 B(x)B(x)^{\top}$
for all $x\in\cX$ and $\xi\in\Rr^d$ with $\|\xi\|=1$. %
\end{enumerate}

\begin{lemma}
    \label{lem:bound_gradient}
    Under Assumption~\ref{hypo:C2},
    $V^*$ is uniformly Lipschitz continuous
    and satisfies, for a.e.~$x\in\cX$: \begin{align*}
       \|\nabla V^*(x)\|\leq C_{\nabla V^*} :=
    \sqrt{d} \left(C_u^{q-1}C_{\partial B}\|\alpha\|_{\infty}
    +C_u^{q-1}\|\partial_{\xi}\alpha\|_{\infty}
    +C_u^{q} C_{\partial B}
    +\|\partial_{\xi}f\|_{\infty} \right) \, .
    \end{align*} 
    \end{lemma}
    \begin{lemma}
    \label{lem:SC}
    Under Assumption~\ref{hypo:C2}, either if $q=2$, or if $q<2$ and 
    \begin{align*}
        \rho>\rho_1:=4\sup\{-\xi^{\top}D_xa(x)\xi,x\in\cX,\|\xi\|=1\} \, ,
    \end{align*}
    then $V^*$ is semi-concave~\cite{cannarsa2004semiconcave},  \textit{i.e.}, there exists $C_{\textnormal{SC}} \geq 0$ such that $D^2_{x} V^*(x) \preceq C_{\textnormal{SC}} I_d$, 
where $D^2_x V^*$ denotes the Hessian matrix of $V^*$. 
\end{lemma}

We can now study the behavior of the optimal (discounted) occupation measure~$\mu^*$. Note that $\mu^*$ can only grow unbounded on zero measure sets, otherwise it would not be a probability measure. Divergence remains possible  at accumulation points of the optimal trajectories, typically at the minima of $V^*$ if $a \equiv 0$. We prove below that such accumulation does not happen in two cases:  if the time horizon is short enough (\textit{i.e.}, $\rho$ large enough), or if controllers that keep the state stationary are penalized (\textit{i.e.}, $a \equiv 0$ and $q<2$).  %

\begin{proposition}
    \label{prop:bound_mu*}
    Under Assumption~\ref{hypo:C2}, 
    $\mu^*$ is uniformly bounded in any of the following two cases:
    \emph{(i)} $a\equiv0$ and $q<2$;
    or \emph{(ii)} $\rho>\rho_0$ for some $\rho_0$ depending on the
    $L^{\infty}$-norms of $B,a$ and $\alpha$ and their derivatives.
\end{proposition}

\begin{remark} The control density $w^* = \mu^* u^* $ can remain bounded, even though~$\mu^*$ diverges. In dimension one with $g(x,u)=u$, Liouville's equation is: ${w^*}'= \rho (\mu_0 - \mu^*)$. Let $x_0$ a global minimizer of $f$. The null controller is optimal at this point, hence $w^*(x_0) = \mu^*(x_0) u^*(x_0) = 0$. Integrating Liouville's equation, we obtain:
\begin{align*}
 \forall x\in \cX, \quad   w^*(x) - w^*(x_0) = \rho \int_{x_0}^{x} (\mu_0(\xi) - \mu^*(\xi)) \, \mathrm  d \xi  \, .
\end{align*}
$\mu_0$ and $\mu^*$ being probability measures, this implies that $\forall x, |w^*(x)| \leq \rho$. Note that this reasoning does not extend beyond $d=1$, due to possible curl components in~$w^*$.
\end{remark}

\section{Conclusion} We have presented three different approaches to apply stochastic optimization schemes to control-affine optimal control problems, with convergence guarantees, some of them providing control over the performance of the controller. This is a first step towards extending the reliability of sample-based optimization methods for control problems, among which fall many RL algorithms, with guarantees approaching those existing for supervised learning problems. Several future directions could be considered.

First, this work only considers deterministic dynamics.
A natural next step would be to extend the approach to stochastic dynamics,
where the state is no longer described by an ODE but by an SDE.
In this setting, the HJB and Fokker--Planck--Kolmogorov equations become second-order, 
due to the additional diffusion term appearing in both equations;
in the case of uniform isotropic noise, this term reduces to a Laplacian. We believe that several aspects of the present analysis 
would need to be revisited in this stochastic setting.
On the one hand, if the noise is uniformly non-degenerate, 
Assumption~\ref{hypo:mustarbounded} should be automatically satisfied, 
and the solutions of the associated PDEs should enjoy additional regularity properties.
This would be very useful and could simplify parts of the analysis. 
On the other hand, some specific arguments would no longer extend directly.
For instance, the primal formulation would require additional care:
in the deterministic case with $a=0$,
the variable $\mu$ can be eliminated from \eqref{eqn:Ppen} 
because no derivative of $\mu$ appears after the reformulation.
In the presence of noise, the Fokker--Planck--Kolmogorov equation
contains a second-order term involving $\mu$,
and this elimination is no longer straightforward.

Second, %
in the current formulations, the states are only  sampled  according to the uniform distribution. Alternative sampling distributions, with importance sampling, could be considered to reduce the variance.  
Finally, while the presented convergence results rely on the use of linear non-parametric representations of $V^*$ and $w^*$ guaranteeing convexity, a promising practical direction would be to explore different approximation schemes. In particular, our algorithms are readily applicable with non-linear parameterizations, including deep neural networks, using automatic differentiation. 

\section*{Declaration of LLM usage} We have used LLMs to help us complete the proof of Proposition~\ref{prop:bound_mu*} in the $q<2$ case only, and for revision purposes, which allowed us to correct minor inconsistencies in some of the proofs, without impacting the results.

\section*{Acknowledgements}

This work benefited from  State fundings managed by the French National Research Agency (ANR) under the France 2030 program, with references ``Hi! PARIS'' (ANR-23-IACL-0005), ANR-11-IDEX-0003 within the OI H-Code,  and “PR[AI]RIE-PSAI” (ANR-23-IACL-0008).

\bibliographystyle{abbrv}
\bibliography{biblio}

\newpage

\appendix

\section{Proofs of results} \label{app:proofs}
\begin{proof}[Proof of Proposition \ref{prop:bound}] %
The Hamiltonian of the HJB equation is not coercive in $p$: %
if $p\in\ker\bigl(B(x)^\top\bigr)$, 
then $p^\top\bigl(a(x)+B(x)u\bigr)=p^\top a(x)$
is independent of $u$, so the control 
has no effect along the directions of $\ker(B(x)^\top)$.
This degeneracy prevents us from deriving 
the almost-everywhere existence of $\nabla V^*$
from the standard theory of non-degenerate HJB equations.
To circumvent it, we introduce 
the following regularized optimal control problem,
in which an auxiliary control $\tilde u\in\Rr^d$
acts in every direction:
\begin{align*}
   & V_{\eta}(x) 
   =
   \inf_{(u,\tilde{u}):\Rr_+ \rightarrow \Rr^m\times\Rr^d}
   \rho \int_0^{+\infty} e^{-\rho t} 
   \left( f(x_t) + R(u_t)+\frac{\|\tilde{u}_t\|^2}2 \right) \mathrm dt
   \\
    \nonumber  &  \text{such that } ~ \forall t \geq 0, 
    ~ \dot x_t = a(x_t) + B(x_t) u_t + \eta\tilde{u}_t \, ~\text{ and }  x_0=x \, .
\end{align*}
The bound of Proposition~\ref{prop:V*_bounded} carries over verbatim,
so $0\leq V_{\eta}\leq \|f\|_{\infty}$.
Solving the infimum over $\tilde u$ explicitly,
$V_{\eta}$ is the unique viscosity solution of
    \begin{equation*}
        - V_{\eta}(x) + f(x) + \rho^{-1} \nabla V_{\eta}(x)^\top a(x)
        - R^*\big(  - \rho^{-1} B(x)^\top \nabla V_{\eta}(x)  \big)
        -\frac{\eta^2}2\|\nabla V_{\eta}(x)\|^2= 0.   
    \end{equation*}
The augmented control matrix $\big[\,B(x)\mid \eta\, I_d\,\big]$
now has full rank, 
so the regularized problem is non-degenerate;
by standard results~\cite{bardi1997optimal}, 
its value function $V_{\eta}$ is Lipschitz continuous,
and hence, by Rademacher's theorem, 
differentiable almost everywhere.
At every point of differentiability the viscosity equation
holds in the classical sense.

We now derive a bound on
$r_{\eta}(x):=-\rho^{-1}B(x)^\top\nabla V_{\eta}(x)$
that is uniform in $\eta$.
Using $a=B\alpha$ (Assumption~\ref{hypo:alpha}),
the drift term rewrites as
$\rho^{-1}\nabla V_\eta^\top a 
= -\alpha^\top r_\eta$, 
so for a.e.\ $x\in\cX$ the HJB  equation gives
\begin{align*}
    \frac1{q'}\|r_{\eta}(x)\|^{q'}
    +\frac{\eta^2}2\|\nabla V_{\eta}(x)\|^2
    &=
    R^*(r_{\eta}(x))
    +\frac{\eta^2}2\|\nabla V_{\eta}(x)\|^2
    \\
    &=
    -V_{\eta}(x) +f(x) 
    -\alpha(x)^{\top}r_{\eta}(x)
    \\
    &\leq
    \|f\|_{\infty}
    +\frac1{q'}\left\|2^{-\frac1{q'}}r_{\eta}(x)\right\|^{q'}
    +\frac1q\left\|2^{\frac1{q'}}\alpha(x)\right\|^{q}
    \\
    &\leq
    \frac{\|r_{\eta}(x)\|^{q'}}{2q'}
    +\|f\|_{\infty}
    +\frac{2^{q-1}}q\|\alpha\|^q_{\infty},
 \end{align*}
where we used $V_{\eta}\geq 0$ and Young's inequality 
$\alpha^\top r \leq \tfrac1q\|\,\cdot\,\|^q + 
\tfrac1{q'}\|\,\cdot\,\|^{q'}$ with weight $2^{1/q'}$.
Absorbing the $r_\eta$ term on the left 
and multiplying by $2q'$ yields
\[
   \|r_{\eta}\|_{\infty}^{q'} \leq 2q'\,\|f\|_{\infty} + 2^{q}(q'-1)\,\|\alpha\|_{\infty}^{q} =: C_r^{q'} .
\]
The crucial point is that this bound is independent of $\eta$.
Note that the above computations also imply  
$\frac{\eta^2}2\|\nabla V_{\eta}(x)\|^2\leq \frac{C^{q'}_r}{2q'}\leq C^{q'}_r$.
 
It remains to pass to the limit $\eta\to 0$. 
On the one hand, the additional term
$p\mapsto\tfrac{\eta^2}{2}\|p\|^2$ converges to $0$
locally uniformly in $p$, so by stability of viscosity solutions 
$V_{\eta}\to V^\ast$ uniformly on $\cX$. 
On the other hand, $(r_{\eta})_\eta$ is bounded in $L^{\infty}(\cX)$.
Since $L^{\infty}(\cX)$ is the dual of $L^1(\cX)$,
its closed balls are weak-$*$ compact by the Banach--Alaoglu theorem;
hence, along a subsequence, $r_{\eta}$ converges weak-$*$
to some $r\in L^{\infty}(\cX)$ with $\|r\|_{\infty}\leq C_r$.

To identify $r$,
fix a test function $\varphi\in\cC^{\infty}(\cX;\Rr^m)$.
Since $B$ is Lipschitz and $\varphi$ is smooth,
$B\varphi\in W^{1,\infty}(\cX)$,
and integration by parts on the torus 
(with no boundary term, by periodicity) gives
\begin{align*}
\int_{\cX}\varphi(x)^{\top}r_{\eta}(x)\, \mathrm dx
    &=
    -\frac1{\rho}\int_{\cX}(B(x)\varphi(x))^{\top}\nabla V_{\eta}(x)\, \mathrm dx \\
    &
    =
    \frac1{\rho}\int_{\cX}{\rm div}_x(B(x)\varphi(x))V_{\eta}(x)\, \mathrm dx \, .
\end{align*}
The left-hand side converges to 
$\int_{\cX}\varphi^\top r\,\mathrm dx$
by weak-$*$ convergence,
while the right-hand side converges to
$\frac{1}{\rho}\int_{\cX}{\rm div}_x(B\varphi)\,V^*\,\mathrm dx$
by uniform convergence of $V_{\eta}$.
Therefore $-\rho^{-1}B^\top\nabla V^* = r$
in the sense of distributions,
and hence almost everywhere,
since $r\in L^{\infty}(\cX)$.

In summary, $r=-\rho^{-1}B^\top\nabla V^\ast$ 
is well-defined almost everywhere 
and satisfies $\|r\|_{\infty}\leq C_r$.
Consequently the optimal control 
$u^\ast(x)=\nabla R^\ast\bigl(r(x)\bigr)$
is defined a.e., and
\begin{align*}
    \|u^*(x)\|=\|\nabla R^*(r(x))\|=\|r(x)\|^{q'-1}\leq C_u^{(q-1)(q'-1)}=C_u
\qquad\text{for a.e.\ } x\in\cX,
\end{align*}
which concludes the proof.

\end{proof}

\begin{proof}[Proof of Lemma~\ref{lemma1}] %
    Let $\varphi : \cX \to \Rr$ a  continuously differentiable test function (hence uniformly bounded) and $x_0 \in \cX$. We have, using Fubini's theorem:
    \begin{align}
        \langle \nu_{x_0}, \varphi \rangle &= \int_\cX \varphi(x) \nu_{x_0}(x) \, \mathrm dx \nonumber \\
        & = \int_0^{+\infty} \rho e^{-\rho t} \int_\cX \varphi(x) \indic_{\mathrm dx}(x(t \mid x(0)=x_0)) \, \mathrm dx \, \mathrm  dt \nonumber \\
        & = \int_0^{+\infty} \rho e^{-\rho t} \varphi(x(t \mid x(0) = x_0)) \, \mathrm dt \, . \label{eqn:pptnu}
    \end{align}
Hence integration with respect to $\nu_{x_0}$ is equivalent to discounted time-integration along the trajectory starting from~$x_0$. Integrating by parts, we obtain:
\begin{align*}
    \langle \nu_{x_0}, \varphi \rangle &= \varphi(x_0) + \int_0^{+\infty} e^{-\rho t} \frac{\mathrm d}{\mathrm dt}\left(\varphi(x(t \mid x(0)=x_0)) \right)\, \mathrm dt \\
    &= \varphi(x_0) + \frac{1}{\rho}\int \rho e^{-\rho t} \nabla \varphi(x(t \mid x(0)=x_0))^\top g(x(t \mid x(0)=x_0)) \, \mathrm dt \\
    &= \langle  \delta_{x_0}, \varphi \rangle + \frac{1}{\rho} \langle \nu_{x_0},  \nabla \varphi^\top g \rangle  ,
\end{align*}
using~\eqref{eqn:pptnu} on the test function $\nabla \varphi^\top g$. We can then derive the expression of the adjoint of the linear operator $ \varphi \mapsto \nabla \varphi^\top g$ using an integration by parts on $\cX$:
\begin{align*}
    \langle \nu_{x_0}, \nabla \varphi^\top g \rangle  = \int_\cX \nabla \varphi(x)^\top g(x) \nu_{x_0}(x) \mathrm dx 
    = - \langle \text{div}(g \nu_{x_0}),  \varphi \rangle \, .
\end{align*}

Since the relation $ \langle \nu_{x_0}, \varphi \rangle  = \langle \delta_{x_0}, \varphi \rangle -  \langle \text{div}(g \nu_{x_0}), \varphi \rangle $ is true for any test function~$\varphi$, in the sense of distributions, we have:
\begin{align*}
    \nu_{x_0} =  \delta_{x_0} - \frac{1}{\rho}\text{div}(g \nu_{x_0}),
\end{align*}
and, integrating this with respect to $\mu_0$, we get:
\begin{align*}
    \nu =  \mu_0 - \frac{1}{\rho}\text{div}(g \nu) \, .
\end{align*} 
\end{proof}

\begin{proof}[Proof of Lemma \ref{lem:primal_on_proba}]
    Define $\mathds{1}_{\cX}:x\in\cX\mapsto 1$
    and $\mathds{1}_{\cX\times\Rr^m}:(x,u)\in\cX\times\Rr^m\mapsto 1$
    and observe that $\cL\mathds{1}_{\cX}=\mathds{1}_{\cX\times\Rr^m}$.
    Then, recall that $\mu_0 \in \cP(\cX)$,
    so that for $\nu\in\cM_+(\cX\times\Rr^m)$
    feasible for \eqref{P}, we have
    \begin{equation*}
        \langle\mathds{1}_{\cX\times\Rr^m},\nu\rangle
        =
        \langle\cL\mathds{1}_{\cX},\nu\rangle
        =
        \langle\mathds{1}_{\cX},\cL^*\nu\rangle
        =
        \langle\mathds{1}_{\cX},\mu_0\rangle
        =
        1.
    \end{equation*}
    This and $\nu\in\cM_+(\cX\times\Rr^m)$ imply
    that $\nu\in\cP_c(\cX\times\Rr^m)$.

    The remainder of the proof shows that using such a feasible $\nu$,
    we can create another measure $\widetilde{\nu}$
    that is feasible and supported on the graph of a bounded function
    such that it admits at most the same cost
    as $\nu$ with respect to the minimization problem \eqref{P}.
    
    First, use once again the disintegration theorem
    and get $\nu(\mathrm dx, \mathrm du)=\nu_x(\mathrm du)\mu(\mathrm dx)$.
    We can define the averaged control
    $\widetilde u:x\in\cX\mapsto
    \int_{\Rr^m}u\,\nu_x(\mathrm du)$
    and the concentrated measure
    $\widetilde\nu(\mathrm dx,\mathrm du)
    :=\delta_{\tilde u(x)}(\mathrm du)\,\mu(\mathrm dx)$.
    It is straightforward to check that
    $\widetilde{\nu}$ is a probability measure.
    Its support is included in
    the convex hull of the support of $\nu$,
    which makes it compact, so that 
    $\widetilde{\nu}\in\cP_c(\cX\times\Rr^m)$.
    Moreover, $\nu$ and $\widetilde{\nu}$ share the
    same first marginal and
    we have $\int_{\Rr^m}b(x,u)\,\nu_x(\mathrm du)=b(x,\widetilde{u}(x))$
    for all $x\in\cX$, %
    so that $\widetilde{\nu}$ satisfies
    $\cL^*\widetilde{\nu}=\mu_0$ just as $\nu$ does.
    Therefore, $\widetilde{\nu}$ is feasible
    and Jensen's inequality implies 
    \begin{equation*}
        \iint_{\cX\times\Rr^m} R(u)\,\widetilde{\nu}(\mathrm dx, \mathrm du)
        =
        \int_{\cX} R\left(\int_{\Rr^m}u\,\nu_x(\mathrm du)\right)\mu(\mathrm dx)
        \leq
        \iint_{\cX\times\Rr^m} R(u)\,\nu(\mathrm dx, \mathrm du),
    \end{equation*}
    while $\iint f\widetilde{\nu}= \iint f\nu$
    remains unchanged.
    Therefore, $\widetilde{\nu}$ has
    no larger cost than $\nu$.
    Moreover, if $\nu$ is not supported on the graph
    of a function, the strict convexity
    of $R$ implies that the Jensen inequality above is in
    fact a strict inequality making the 
    cost of $\widetilde{\nu}$ being 
    lower than this of $\nu$.
    This concludes the proof.
\end{proof}

\begin{proof}[Proof of Theorem~\ref{thm_duality}]
\begin{figure}
    \centering
        \begin{tikzpicture}[every node/.style={scale=0.8}]
		\node[draw=OrangeRed!55, minimum size=2cm, very thick,fill=OrangeRed!20, initial] (s0) {P};
         \node[circle,draw,text=white,fill=white,draw=white,minimum size=0.5cm] (c) at (-1.05,0){}; 
		\node[draw=OrangeRed!55, minimum size=2cm, very thick,fill=OrangeRed!20, right of=s0] (s1) {$\text{Pc}$};
		\node[draw=Cerulean!55, minimum size=2cm, very thick,fill=Cerulean!20, below of=s0] (s2) {D};
		\node[draw=Cerulean!55, minimum size=2cm, very thick,fill=Cerulean!20, right of=s2] (s3) {$\text{Dc}$};
		\node[draw=orange!55, minimum size=2cm, very thick,fill=orange!20, below of=s2] (s7) {OCP};
		\node[draw=orange!55, minimum size=2cm, very thick,fill=orange!20, right of=s7] (s8) {$\text{OCP}_\cU$};
		
		\draw (s0) edge[left] node[above]{$=$} (s1);
		\draw (s0) edge[left] node[below]{\textit{(2.iii)}} (s1);

            \draw (s7) edge[left] node[above]{$=$} (s8);
		\draw (s7) edge[left] node[below]{\textit{(Cor.~\ref{cor_constr})}} (s8);

            \draw[dotted,thick] (s0) edge[bend right=40] node[below left, sloped]{$\geqslant$} (s2);
		\draw[dotted,thick] (s0) edge[bend right=40] node[below left]{\textit{(1.ii)}} (s2);

            \draw (s0) edge[bend left=40] node[above right]{~$=$} (s2);
		\draw (s0) edge[bend left=40] node[below right]{\textit{(2.iii)}} (s2);

            \draw (s1) edge[left] node[above right]{$=$} (s3);
		\draw (s1) edge[left] node[below right]{\textit{(2.i)}} (s3);
        \draw (s1) edge[left] node[left]{\textcolor{red}{\textbf{\cite{anderson1983review}}}} (s3);

            \draw[dotted,thick] (s2) edge[bend right=40] node[below left, sloped]{$\geqslant$} (s7);
		\draw[dotted,thick] (s2) edge[bend right=40] node[below left]{\textit{(1.iii)}} (s7);

            \draw (s2) edge[bend left=40] node[above right]{~$=$} (s7);
		\draw (s2) edge[bend left=40] node[below right]{\textit{(2.iii)}} (s7);
      
		\draw (s2) edge[left] node[below]{\textit{(2.iii)}} (s3);
		\draw (s2) edge[left] node[above]{$=$} (s3);

            \draw (s3) edge[left] node[below right]{\textit{(2.ii)}} (s8);
		\draw (s3) edge[left] node[above right]{~$=$} (s8);
        \draw (s3) edge[left] node[left]{\textcolor{red}{\textbf{\cite{bardi1997optimal}}}} (s8);

	\end{tikzpicture}
    \caption{Sketch of the proof of Theorem~\ref{thm_duality}. Primal LPs appear in pink, dual LPs in blue, and OCPs in orange. The references to the  central arguments supporting the theorem are written in red.}
    \label{fig:sketchofproof}
\end{figure}
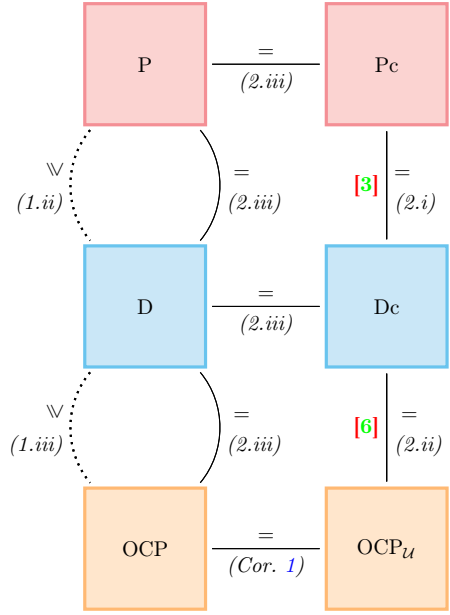

To prove the first claim, we prove, in turn, that \eqref{P} and \eqref{D}
are feasible \textit{(1.i)}, that \eqref{D}
is the Lagrange dual of problem \eqref{P} \textit{(1.ii)},
and that $\int_\cX V^* \mathrm d\mu_0$ is a lower-bound on \text{val}(D) \textit{(1.iii)}. 
    
    \textit{(1.i) \eqref{P} and \eqref{D} are feasible.} Let $V_0 \equiv 0$. 
    This function is $\cC^1$ and satisfies
\begin{align*}
    \forall (x,u) \in \cX \times \Rr^m , \,  -  V_0(x) +  f(x) + R(u) + \frac{1}{\rho} \nabla V_0(x) ^\top b(x, u) = f(x) + R(u)\geq 0.
\end{align*}
Hence $V_0$ is feasible for \eqref{D}, and \eqref{D} is feasible.
We now consider the controller $u_0 \equiv 0$. The trajectories generated by $u_0$ are such that \begin{align*}
    \dot x(t) = b(x(t), 0) = a(x(t)) .
\end{align*}
Using Lemma~\ref{lemma1}, the corresponding occupation measure $\widetilde{\mu}(x)$ is such that
\begin{align*}
     \widetilde{\mu} = \mu_0 - \frac{1}{\rho} \textnormal{div}(a \widetilde{\mu}) \, .
\end{align*}
Now let $\widetilde{\nu} \in \cP(\cX \times \cU)$
defined as $\widetilde{\nu}(\mathrm dx, \mathrm du):=
\widetilde{\mu}(\mathrm dx) \delta_{0}(\mathrm du)$,
with first marginal $\widetilde{\mu}$.
Its averaged drifts are $\int_{\Rr^m}b(x,u)\widetilde{\nu}_x(du) =
b(x,0)=a(x)$,
so the above Liouville equation implies that $\widetilde{\nu}$
is feasible for \eqref{P}, so \eqref{P} is feasible.

    \textit{(1.ii) \eqref{D}
    is the Lagrange dual of \eqref{P}.} 
    Consider problem \eqref{P}, which is a linear program
    with a constraint posed in $C^1(\cX)^*$.
    Therefore, a Lagrange multiplier $V$ may be taken in $C^1(\cX)$
    to formulate \eqref{P} as follows
    \begin{align*}
        {\rm val}\eqref{P}
        &= 
        \inf_{\nu\in\cM_+(\cX\times\Rr^m)}
        \left\{
        \iint_{\cX\times\Rr^m} (f(x)+R(u))\,\nu(\mathrm dx, \mathrm du)
        +\sup_{V\in\cC^1(\cX)}\langle V,\mu_0-\cL^*\nu\rangle\right\}
        \\
        &=
        \inf_{\nu\in\cM_+(\cX\times\Rr^m)}
        \sup_{V\in\cC^1(\cX)}
        \iint_{\cX\times\Rr^m} (f+R)\,\nu
        +\int_{\cX}V\,\mu_0
        -\langle\cL V,\nu\rangle.
    \end{align*}
    Then using the inequality $\inf \sup\geq\sup\inf$, we can recover
    the Lagrange dual of \eqref{P} and the desired inequality
    involving \eqref{D}:
    \begin{align*}
        {\rm val}\eqref{P}
        &\geq 
        \sup_{V\in\cC^1(\cX)}
        \inf_{\nu\in\cM_+(\cX\times\Rr^m)}
        \iint_{\cX\times\Rr^m} (f+R)\,\nu
        +\int_{\cX}V\,\mu_0
        -\langle\cL V,\nu\rangle
        \\
        &=
        \sup_{V\in\cC^1(\cX)}
        \int_{\cX}V\,\mu_0
        +\inf_{\nu\in\cM_+(\cX\times\Rr^m)}
        \langle-\cL V+f+R,\nu\rangle
        \\
        &=
        {\rm val}\eqref{D}.
    \end{align*}
    We just proved val\eqref{P}$\geq$val\eqref{D}.

    \textit{(1.iii) \eqref{eq:OCP} has not a larger value
    than \eqref{D}.}
    We define the Hamiltonian $H$ as follows, 
    for $(x,p)\in\cX\times\Rr^d$ %
    \begin{align*} 
        H(x,p)
        &=
        \sup_{u\in\Rr^m}\{- \rho^{-1}p^{\top}b(x,u)-R(u)-f(x)\}\\
        &=
        R^*(-\rho^{-1}B(x)^{\top}p)-\rho^{-1}a(x)^{\top}p-f(x).
    \end{align*}
    Observe that $H$ is convex in $p$ %
    and it is Lipschitz continuous in $x$ on any compact
    $\cX\times \overline{B}_{\Rr^d}(0,r)$ for $r>0$,
    let us denote $C(r)$ its associated Lipschitz constant.
    Consider $V_{\eta}$ defined as in the proof of Proposition~\ref{prop:bound}.
    It satisfies
    \begin{equation}
        \label{eq:PDE_V_eta}
        V_{\eta}+H(x,\nabla_x V_{\eta})
        =
        -\frac{\eta^2}2\|\nabla_x V_{\eta}\|^2
        \leq
        0
    \end{equation}
    in the sense of viscosity and at any point of differentiability.
    From the proof of Proposition \ref{prop:bound}, we have
    that $V_{\eta}$ is Lipschitz continuous with constant
    $\eta^{-1} C^{\frac{q'}2}_r$.
    By Rademacher's theorem, it is differentiable almost everywhere
    and
    $\|\nabla_xV_{\eta}\|_{L^{\infty}}\leq \eta^{-1} C^{\frac{q'}2}_r$.
    Now, define $V_{\eta,\varepsilon}=\chi_{\varepsilon}\star V_{\eta}$ with 
    $\chi_{\varepsilon}$ a $\cC^{\infty}$-mollifier supported in 
    $\overline{B}_{\cX}(0,\varepsilon)$
    that converges to $\delta_0$
    in the sense of distribution
    as $\varepsilon$ tends to zero.
    Since $H$ is convex with respect to its second entry, we have, by Jensen's inequality: %
    \begin{align*}
        H(x,\nabla_x V_{\eta,\varepsilon})
        &=
        H\left(x,\int_{\cX}\nabla_x V_{\eta}(x-y)\chi_{\varepsilon}(y)\,\mathrm dy\right)
        \\        
        &\leq
        \int_{\cX}
        H(x,\nabla_x V_{\eta}(x-y))
        \chi_{\varepsilon}(y) \, \mathrm dy
        \\
        &\leq
        \int_{\cX}
        H(x-y,\nabla_x V_{\eta}(x-y))
        \chi_{\varepsilon}(y) \, \mathrm dy
        +C(\|\nabla_xV_{\eta}\|_{L^{\infty}})\varepsilon,
    \end{align*}
    where we used the local Lipschitz continuity of
    $H$ in $x$ and the definition of $\chi_{\varepsilon}$
    to get the latter inequality.
    Using the latter inequality and \eqref{eq:PDE_V_eta},
    with $\varepsilon=\varepsilon(\eta):= C(\eta^{-1}C_r^{\frac{q'}2})^{-1}\eta^{-1}$,
    we obtain,
    for almost every $x\in\cX$,
    \begin{equation*}
        V_{\eta,\varepsilon(\eta)}(x)+
        H(x,\nabla_x V_{\eta,\varepsilon(\eta)})
        \leq
        \eta.
    \end{equation*}
    Recall that $V_{\eta,\varepsilon(\eta)}$ is $\cC^{\infty}$ so the latter
    inequality holds for all $x\in\cX$.
    This implies that
    $V_{\eta,\varepsilon(\eta)}-\eta$
    is a $\cC^{\infty}$
    subsolution of the HJB equation,
    in particular it is feasible for \eqref{D}.
    Therefore, we obtain
    \begin{equation*}
        {\rm val}\eqref{D}
        \geq
        \int_{\cX}(V_{\eta,\varepsilon(\eta)}-\eta)\mu_0
         \xrightarrow[\eta\to 0]{}
        \int_{\cX}V^*\mu_0,
    \end{equation*}
    since $V_{\eta,\varepsilon(\eta)}$ converges to $V^*$
    in $\cC^0(\cX)$
    where we used that $V_{\eta}$ converges to $V^*$ in $\cC^0(\cX)$.
    This concludes the first item of Theorem \ref{thm_duality}.\\

To prove the second and third claims of the theorem, we first prove them for the constrained problem  defined in Corollary~\ref{cor_constr}, with its corresponding primal and dual problems \eqref{PU} and \eqref{DU}, defined below (steps \textit{(2)} and \textit{(3)}).  This intermediate step is required because existing results for both claims require the control set to be compact, which is not the case in our setting.
Afterwards, simple manipulations and Corollary~\ref{cor_constr} will allow us to transfer these results for the true unconstrained problem and its primal and dual weak formulations (step \textit{(4)}).

First, let us define $\cL_{\cU}$ similarly as $\cL$
but with $\cU$ as the set of admissible control:
\begin{equation*}
    \cL_{\cU}:
    \left\{
    \begin{aligned}
        \; \cC^1(\cX)&  \rightarrow \cC(\cX \times \cU)
        \\
        \varphi ~~ &\mapsto  \Bigl\{ (x, u) 
        \mapsto
        \varphi(x) - \frac{1}{\rho} \nabla \varphi(x)^\top b(x, u) \Bigr\}.
\end{aligned}
\right.
\end{equation*}
We are now in position to define the primal and
dual problems with $\cU$ as the set of admissible control:
\begin{align} 
    \label{PU}\tag{Pc}
    &\inf_{\nu \in \cM_+(\cX \times \cU)} 
    \iint_{\cX\times\Rr^m}
    (f(x)+R(u))\,\nu(\mathrm dx, \mathrm du)
    \hspace*{0.3cm}
    \text{such that}
    \hspace*{0.3cm}
    \cL_{\cU}^* \nu = \mu_0 \, .
    \\
    \label{DU}\tag{Dc}
    &\sup_{V \in \cC^1(\cX)}
    \int_{\cX}V(x)\,\mu_0(x) \, \mathrm dx
    \hspace*{0.3cm}
    \text{such that}
    \hspace*{0.3cm}
    \cL_{\cU} V \leq f + R_{|\cU} \, .
\end{align} 

Repeating the arguments in the proof of
the first item of Theorem \ref{thm_duality},
we obtain the \eqref{PU} and \eqref{DU} are feasible
and satisfy
\begin{equation*}
    \textnormal{val}\eqref{PU}  
    \geq
    \textnormal{val}\eqref{DU} 
    \geq 
    \int_\cX V^*(x) \mu_0(x) \,  \mathrm dx \,,
\end{equation*}
where we used
Corollary \ref{cor_constr} 
to get that the optimal value
of \eqref{eq:OCP} remains unchanged
if we constrain the set of admissible control
to be $\cU$.
Below, we will prove that the two inequalities 
above are in fact equalities.
The proof relies on the fact that $\cU$ is compact.

\textit{(2.i) Prove the strong duality between
\eqref{PU} and \eqref{DU}.}
Using Theorem 5 from~\cite{anderson1983review},
a sufficient condition for strong duality to hold,
i.e., val\eqref{PU}=val\eqref{DU}, is that the set $\cH_{\cU}$ defined by
\begin{align*}
    \cH_\cU := \left\{ \left( \cL^*_\cU(\mu),  \langle  f+R, \mu \rangle \right) ~\mid ~ \mu \in \cM_+(\cX \times \cU) \right\} 
\end{align*}
is closed. 
Define $\mathcal{T}_{\cU}:\mu\in\cM_+(\cX \times \cU)\mapsto \left( \cL^*_\cU(\mu),  \langle  f+R, \mu \rangle \right)\in \cC^1(\cX)^*\times \Rr$.
Observe that $\mathcal{T}_{\cU}$ is weak$^*$ continuous since $\cL_{\cU}$ is bounded.
For any $r>0$, using the Banach-Alaoglu theorem~\cite[Theorem 3.15]{rudin}, the closed ball  $B_r \subset \cM(\cX \times \cU)$ of radius $r$ %
is weak$^*$ compact.
Since the image of a compact by a continuous map is itself compact,
we deduce that $\mathcal{T}_{\cU}(B_r)$ is weak$^*$ compact,
so it is weakly$^*$ closed, so it is strongly closed.

We are now in position to prove that $\cH_{\cU}$ is closed.
Consider a sequence $( \mu_n )_{n \geq 1}$ of elements of $\cM_+(\cX \times \cU)$ such that: 
\begin{align*}
    \left\{
    \begin{array}{ll}
       \cL^*_\cU(\mu_n) & \xrightarrow[n \to +\infty]{} a \\
       \langle f+R, \mu_n \rangle & \xrightarrow[n \to +\infty]{}  b \, ,
    \end{array}
\right.
\end{align*}
for some $(a, b) \in \cC^1(\cX)^* \times \Rr$.  In particular, integrating the first limit with respect to the test function $\varphi_0 : x \mapsto 1$, we get:
\begin{align*}
    \langle \varphi_0 , \cL_\cU^*(\mu_n) \rangle &= \langle   \cL_\cU( \varphi_0) , \mu_n \rangle \\
    &= \langle   1, \mu_n \rangle \\
    &= \mu_n(\cX \times \cU) \xrightarrow[n \to +\infty]{}   \langle  \varphi_0, a  \rangle \, .
\end{align*}
Hence there exists $n_0 \geq 1$ and a closed ball  $B_r$ of $\cM(\cX \times \cU)$ (for the total variation norm) such that for all $n \geq n_0$, $\mu_n \in B_r$.
This and the fact that $\mathcal{T}_{\cU}(B_r)$ is closed imply that $(a,b)\in\mathcal{T}_{\cU}(B_r)\subset\cH_{\cU}$.
In turn, we obtain that $\cH_{\cU}$ is closed.   From that, we deduce from~\cite{anderson1983review}, that strong duality holds for the constrained primal and dual problems:
\begin{align*}\text{val}(\text{Pc}) = \text{val}(\text{Dc}) \, .\end{align*}

\textit{(2.ii) Prove that \eqref{DU} and \eqref{eq:OCP} share 
the same value.}
We already have one inequality;
it only remains to prove val\eqref{DU}$\leq \int_{\cX}V^*\mu_0$
which will be obtained using a standard comparison principle
for viscosity subsolutions and supersolutions.

Define $H_{\cU}$, the Hamiltonian with $\cU$ as the admissible set
of controls, by %
    \begin{equation*} 
        H_{\cU}(x,p)
        =
        \sup_{u\in\cU}\{-p^{\top}b(x,u)-R(u)-f(x)\}.
    \end{equation*}
    It satisfies the condition, for $x,y\in\cX$ and $p\in\Rr^d$,
    \begin{equation}
        \label{eq:cond_comparison}
        |H_{\cU}(x,p)-H_{\cU}(y,p)|
        \leq
        C\|x-y\|(1+\|p\|),
    \end{equation}
    with $C$ a constant obtained using
    $f,a,B$ be $\cC^1$ from Assumption
    \ref{hypo:regularity};
    more precisely, it can be taken as
    $C=\|\nabla_xf\|_{\infty}+\rho^{-1}\|D_xa\|_{\infty}
    C_u\|D_xB\|_{\infty}$.

    Taking $V$ feasible for \eqref{PU},
    it is a subsolution of
    $V+H_{\cU}(x,\nabla_xV)=0$.
    Using Corollary \ref{cor_constr},
    $V^*$ is a viscosity solution of the same HJB equation,
    in particular it is a viscosity super solution.
    Because $H_{\cU}$ satisfies Condition \eqref{eq:cond_comparison},
    a standard comparison principle \cite{bardi1997optimal}
    implies $V\leq V^*$.
    Taking the integral over $\mu_0$ and the supremum over
    $V$ feasible for \eqref{PU}, we conclude
    \begin{equation*}
        {\rm val}\eqref{PU}
        \leq 
        \int_{\cX}V^*\,\mu_0.
    \end{equation*}

    \textit{(2.iii) Conclude the remainder of Theorem \ref{thm_duality}.}
    Observe that any $V$ feasible for \eqref{PU} is
    also feasible for \eqref{P}, so that we
    have val\eqref{P}$\leq$val\eqref{PU}.
    This and the results from the previous steps imply
    \begin{equation*}
        \int_{\cX}V^*\,\mu_0
        \leq
        {\rm val}\eqref{D}
        \leq
        {\rm val}\eqref{P}
        \leq
        {\rm val}\eqref{PU}
        =
        {\rm val}\eqref{DU}
        =
        \int_{\cX}V^*\,\mu_0.
    \end{equation*}
    Consequently, the latter inequalities are in fact equalities,
    which concludes the proof.

\end{proof}

\begin{proof}[Proof of Proposition~\ref{prop:sgd}]
    Let $\theta \in \cH$ and $n \geq 1$. The following decomposition holds:
    \begin{align*}
        \| \theta_n - \theta \|^2 &=  \| \theta_{n-1} - \theta \|^2 \\&  - 2 \gamma \langle \rho_1'(g_1(\theta_{n-1}, z_n)) g_1'(\theta_{n-1}, z_n) + \rho_2'(g_2(\theta_{n-1}, z_n)) g_2'(\theta_{n-1}, z_n), \theta_{n-1}-\theta \rangle \\
        &  + \gamma^2 \| \rho_1'(g_1(\theta_{n-1}, z_n)) g_1'(\theta_{n-1}, z_n) + \rho_2'(g_2(\theta_{n-1}, z_n)) g_2'(\theta_{n-1}, z_n) \|^2 \\
        & \leq  \| \theta_{n-1} - \theta \|^2 \\
        &\quad + \rho_1'(g_1(\theta_{n-1}, z_n)) (-2 \gamma \langle g_1'(\theta_{n-1}, z_n), \theta_{n-1}-\theta \rangle + 2\gamma^2 \alpha_1 \|g_1'(\theta_{n-1}, z_n)\|^2 ) \\
        &\quad + \rho_2'(g_2(\theta_{n-1}, z_n)) (-2 \gamma \langle g_2'(\theta_{n-1}, z_n), \theta_{n-1}-\theta \rangle + 2\gamma^2 \alpha_2 \|g_2'(\theta_{n-1}, z_n)\|^2 ), 
    \end{align*}
    using the fact that the $\rho'_i \in [0, \alpha_i]$.
    Moreover,  for $\beta=\gamma\max\{4\alpha_iL_i,1\leq i\leq 2\}\leq 1$,
    using successively the co-coercivity of $g'_i$
    for $i \in \{1, 2\}$~\cite{bubeck2015convex},
    and the Cauchy-Schwarz inequality,
    we obtain:
    \label{coco}
    \begin{align*}
    2\gamma\alpha_i\| g'_i(\theta_{n-1}, z_n) \|^2 
    & \leq 
    4\gamma\alpha_i\left( \| g'_i(\theta, z_n) \|^2 
    + \| g'_i(\theta_{n-1}, z_n) - g'_i(\theta, z_n) \|^2\right) \\
     & \leq
     4\gamma\alpha_i\left( \| g'_i(\theta, z_n) \|^2 +
     L_i \langle g_i'(\theta_{n-1}, z_n) -  g_i'(\theta, z_n), \theta_{n-1} - \theta \rangle\right) \\
     & \leq
     4\gamma\alpha_i \| g'_i(\theta, z_n) \|^2 +
     \beta \langle g_i'(\theta_{n-1}, z_n) -  g_i'(\theta, z_n), \theta_{n-1} - \theta \rangle \\
    & \leq
    4\gamma\alpha_i\| g'_i(\theta, z_n) \|^2 
    +\beta \langle g_i'(\theta_{n-1}, z_n) , \theta_{n-1} - \theta \rangle \\
        & \qquad\qquad\qquad~~~ +\beta \|g_i'(\theta, z_n)\| \|\theta_{n-1}-\theta\| \, . %
    \end{align*}
    The latter two chains of inequalities imply that,
\begin{align*}
    \| \theta_n - \theta \|^2 
    &\leq
    \| \theta_{n-1} - \theta \|^2  + \sum_{i=1}^2 \rho_i'(g_i(\theta_{n-1}, z_n)) 
    \left\{ (2-\beta)\gamma \langle g_i'(\theta_{n-1}, z_n), \theta-\theta_{n-1} \rangle  \right. \\
    & \left. \qquad\qquad\qquad + 4 \gamma^2 \alpha_i \| g_i'(\theta, z_n)\|^2 
    + \beta\gamma \|g_i'(\theta, z_n)\| \|\theta_{n-1}-\theta\| \right\} \\
    &\leq
    \| \theta_{n-1} - \theta \|^2  
    + \sum_{i=1}^2 \left\{\rho_i'(g_i(\theta_{n-1}, z_n)) (2-\beta)\gamma \langle g_i'(\theta_{n-1}, z_n), \theta-\theta_{n-1} \rangle  \right. \\
    & \left. \qquad\qquad\qquad + 4 \gamma^2 \alpha_i^2 \| g_i'(\theta, z_n)\|^2 
    + \beta\gamma\alpha_i \|g_i'(\theta, z_n)\| \|\theta_{n-1}-\theta\| \right\} \, .
\end{align*}
Taking the expectation,
first conditionally on $\theta_{n-1}$ and then a full expectation, 
and using Jensen's inequality
$\Ee[\|g_i'(\theta, z_n)\| \|\theta_{n-1}-\theta\|]^2\leq
\Ee[\|g_i'(\theta, z_n)\|^2]
\Ee[\|\theta_{n-1}-\theta\|^2]$, we obtain:
\begin{equation}
\label{eq:aux_sgd_dual}
\begin{aligned}
    \Ee[ \| \theta_n - \theta \|^2 ] 
    & \leq
    \Ee[ \| \theta_{n-1} - \theta \|^2 ] 
    +(2-\beta)\gamma  \Ee [ \langle F'(\theta_{n-1}), \theta-\theta_{n-1} \rangle ] 
    \\
    &\hspace*{0.4cm}+\sum_{i=1}^2\left\{ 4\gamma^2\alpha_i^2 \sigma_i^2(\theta)
    + \beta\gamma\alpha_i\sigma_i(\theta)
    \sqrt{\Ee[\| \theta_{n-1}-\theta\|^2 ]}\right\} \, ,
\end{aligned}
\end{equation}
where $\sigma_i^2(\theta):=\Ee[\| g_i'(\theta, z)\|^2 ]$, for $i=1,2$.
The convexity of $F_i$ and Cauchy-Schwarz inequality imply
\begin{equation*}
    \langle F'(\theta_{n-1}),\theta-\theta_{n-1}\rangle
    \leq
    \langle F'(\theta),\theta-\theta_{n-1}\rangle
    \leq
    \| F'(\theta)\|\|\theta-\theta_{n-1}\|.
\end{equation*}
Moreover, $\| F'(\theta)\|$ can be upper-bounded as follows:
\begin{align*}
    \| F'(\theta)\| = 
    \| \Ee[ \rho'_1(g_1(\theta, z))g_1'(\theta, z)
    + \rho'_2(g_2(\theta, z))g_2'(\theta, z)]\|
    \leq
    \alpha_1\sigma_1(\theta)
    +\alpha_2\sigma_2(\theta)\, .
\end{align*}
From the last three chains of inequalities
and $0\leq\beta\leq1$,
using $u_n:=\sqrt{\Ee[\|\theta_n-\theta\|^2]}$,
we obtain the following.
\begin{equation*}
    u_n^2 
    \leq
    u_{n-1}^2 
    + 2\gamma\sum_{i=1}^2\alpha_i\sigma_i(\theta)
    + 4 \gamma^2  \sum_{i=1}^2 \alpha_i^2 \sigma^2_i(\theta).
\end{equation*}
For $a=2\gamma \sum_{i=1}^2\alpha_i\sigma_i(\theta)$,
it is straightforward to check that
$4\gamma^2\sum_{i=1}^2 \alpha_i^2 \sigma^2_i(\theta)\leq a^2$
and
\begin{equation*}
    u_n^2
    \leq
    u_{n-1}^2
    +a(u_{n-1}+a).
\end{equation*}
Now consider $v_n=an+u_0$ and let us show by induction on $n$
that $u_n\leq v_n$.
We have $u_0\leq v_0$.
Assume that the latter inequality holds at index $n-1$, we obtain
\begin{align*}
    u_n^2
    &\leq
    v_{n-1}^2
    +a(v_{n-1}+b)
    \\
    &=
    a^2(n-1)^2+2a(n-1)u_0+u_0^2+a^2(n-1)+au_0+a^2
    \\
    &=
    a^2n^2+2anu_0+u_0^2-a^2(n-1)-au_0
    \leq
    v_n^2.
\end{align*}
This concludes the induction and implies that
$u_n\leq an+u_0$ for any $n\geq0$.
We will now use this estimates to obtain
the desired inequality.
From \eqref{eq:aux_sgd_dual} with the convex inequality
$\langle F'(\theta_{n-1}),\theta-\theta_{n-1}\rangle\leq F(\theta)-F(\theta_{n-1})$,
we obtain
\begin{equation*}
    (2-\beta)\gamma(\Ee[F(\theta_{n-1})]-F(\theta))
    \leq
    u_{n-1}^2-u_n^2+\frac{a\beta}2u_{n-1}+a^2.
\end{equation*}
Let us consider the Cesaro sum of the latter inequality, we get
\begin{equation*}
      \Ee [ F(\bar{\theta}_{n})] - F(\theta)
      \leq
      \Ee \left[ \frac{1}{n} \sum_{k=0}^{n-1} F(\theta_{k}) \right] - F(\theta)
      \leq
      \frac1{(2-\beta)\gamma}\left(
      \frac{u_0^2}{ n}
      +\frac{a\beta}{2n}\sum_{k=0}^{n-1}u_k+a^2
      \right),
\end{equation*}
where the inequality on the left is obtained from the convexity of $F$.
Recall that $u_n\leq an+u_0$ and $\beta\leq 1$, we get
\begin{equation*}
      \Ee [ F(\bar{\theta}_{n})] - F(\theta)
      \leq
      \frac{u_0^2}{\gamma n}
      +\frac{a^2\beta n}{4\gamma}
      +\frac{a^2}{\gamma}.
\end{equation*}
Moreover, 
$\sigma^2_i(\theta)$ can be controlled as follows, because each $g_i'$ is $L_i$-Lipschitz:
\begin{align*}
    \sigma^2_i(\theta)  \leq 2 \sigma^2_i(\theta_0) + 2 \Ee [\| g'_i(\theta, z) - g'_i(\theta_0, z)\|^2] \leq 2 \sigma^2_i(\theta_0) + 2 L_i^2 \| \theta-\theta_0\|^2 \, .%
\end{align*}
Using the latter inequality along with $\alpha=\max(\alpha_1,\alpha_2)$,
$L=\max(L_1,L_2)$ and $\sigma(\theta)=\max(\sigma_1(\theta),\sigma_2(\theta)$,
we obtain $\beta\leq4\gamma\alpha L$ and
\begin{align*}
    a^2
    =
    4\gamma^2\left(\sum_{i=1}^2\alpha_i\sigma_i(\theta)\right)^2
    \leq
    8\gamma^2\sum_{i=1}^2\alpha_i^2\sigma_i^2(\theta)
    &\leq
    16\gamma^2\sum_{i=1}^2\alpha_i^2(\sigma_i^2(\theta_0)+L_i^2u_0^2)
    \\
    &\leq
    32\gamma^2\alpha^2(\sigma^2(\theta_0)+L^2u_0^2).
\end{align*}
The latter inequalities and imply
\begin{equation*}
      \Ee [ F(\bar{\theta}_{n})] - F(\theta)
      \leq
      \frac{\|\theta-\theta_0\|^2}{\gamma n}
      +32\gamma\alpha^2
      (\gamma \alpha Ln+1)
      (\sigma^2(\theta_0)+L^2\|\theta-\theta_0\|^2).
\end{equation*}
The result follows by subtracting $\inf F$ and taking the infimum over $\theta \in \cH$.%
\end{proof}

\begin{proof}[Proof of Theorem~\ref{thm:dual}]%
    We apply Proposition~\ref{prop:sgd} with $F=-F_M$, the convex non-decreasing functions $\rho_1 = \text{Id}$, $\rho_2 = [\cdot]_+$, with derivatives bounded by $\alpha_1=\alpha_2=1$, and  \begin{align*}
        g_1(V, x) &= - \Lambda (V, x) \qquad~ 
    g_1'(V, x)  = -\mu_0(x) \Phi(x)\\
    g_2(V, x)  & = M Q(V, x) \qquad  g_2'(V, x) = M \Phi(x) - \frac{M}{\rho}\sum_{i=1}^d a_i(x) \frac{\partial \Phi}{\partial x_i}(x)  + 2M \Xi(x) V \, . \end{align*}
    For any $x \in \cX$, $g_1(\cdot, x)$ and $g_2(\cdot, x)$ are convex. $g_1(\cdot, x)$ is $L_1$-smooth with $L_1 = 0$, $g_2(\cdot, x)$ is $L_2$-smooth with $L_2=\frac{M}{\rho^2} m d^2 C_B^2 {c'}_K^2$, where $C_B=\sup_{x \in \cX} \sigma_{\max}(B(x))$ 
    and $ c_K' = \max_{1\leq i\leq d}\sup_{1\leq i\leq d}\|\partial_{x_i}\Phi(x)\|_{\cH}
    =\max_{1\leq i\leq d}\sup_{1\leq i\leq d}|\partial^2_{x_i,y_i}K(x,x)|^{\frac12}$
    is boun\-ded because $K$ is $\cC^2$ on the compact set $\cX$.
Furthermore, we have $\alpha = 1$, $L =  L_2 $ and \begin{align*}
    \sigma^2(0) \leq \max \left\{ C_K^2 \| \mu_0\|_\infty^2, ~ 2 M^2 C_K^2 + \frac{2M^2}{\rho^2} d^2 C_a^2 {c'}_K^2 \right\} := \sigma^2,
\end{align*} 
where $C_K = \|K\|_\infty$ and $C_a = \sup_{x \in \cX}\|a(x)\|_2$.  Setting $\gamma = \frac{1}{4L_2} n^{-2/3}$ satisfies the constraint on the step size $\gamma \leq 1/(4 L)$ for any $n \geq 1$, and we obtain:
       \begin{align} \label{eqn:descent}
     \Ee [F_M(\bar V_n)]   \geq \sup_{V \in  \cH} \left\{  F_M(V)   - \left( \frac{1}{\gamma n} + 32 \gamma^2 n  L^3 + 32 \gamma L^2   \right) \| V\|^2 \right\} 
 \\
  \qquad\qquad - \left( 32 \gamma^2 n  L + 32 \gamma  \right) \sigma^2 \, ,
\end{align}
where we have, for any $n \geq 1$:
\begin{align*}
     1/(\gamma n) &\leq 4L_2n^{-1/3} 
     \hspace*{0.3cm}
     \text{ and }
     \hspace*{0.3cm}
     32 \gamma^2 n  L + 32 \gamma  \leq (10/L) n^{-1/3} \, ,
\end{align*}
which gives the first result for $C_1 = 14Ln^{-1/3}$ and $C_2=(10/L) \sigma^2  $.

In order to prove the convergence result, we need to use the density of $\cH$ in $H^1(\cX)$.  Let $\e >0$. Using Theorem~\ref{thm_duality}, there exits some $V_\e \in \cC^1(\cX)$ such that $\cL V_\e \leq f + R$ and \begin{align} \label{eqn:sub}
    \int_\cX V_\e(x) \mu_0(x) \, \mathrm dx \geq \int_\cX V^*(x) \mu_0(x)  \, \mathrm dx - \e/4 \, .
\end{align}
Observe that $F_M$ is continuous on $H^1(\cX)$.
Then, using the density of $\cH$ in $H^1(\cX)$,
there exists some $\tilde V \in \cH$ such that:
\begin{align} \label{eqn:conti}
    | F_M(\tilde V) - F_M(V_\e) | \leq \e/4 \, .
\end{align}
Furthermore, since $\cL V_\e \leq f + R$, $F_M(V_\e) = \int_\cX V_\e(x) \mu_0(x)$.

Since $\tilde V$ is suboptimal on the right hand side of \eqref{eqn:descent} and subsequently using \eqref{eqn:conti} and \eqref{eqn:sub}, we also have:
\begin{align*}
   \Ee [ F_M(\bar V_n) ] &\geq F_M(\tilde V) - C_1 n^{-1/3} \| \tilde V\|^2_\cH - C_2 n^{-1/3} \\
   & \geq F_M(V_\e) - \e/4 - C_1 n^{-1/3} \| \tilde V\|^2_\cH - C_2 n^{-1/3}  \\
   & \geq \int_\cX V^*(x) \mu_0(x) \, \mathrm dx - \e/4 - \e/4 - C_1 n^{-1/3} \| \tilde V\|^2_\cH - C_2 n^{-1/3} \, . 
\end{align*}
Let $n_0 = \frac{64}{\e^3} \max \{ C_1^3 \| \tilde V\|_\cH^6 , C_2^3 \}$. As soon as $n \geq n_0$, we then have:
\begin{align} \label{eqn:upper}
    \Ee [ F_M(\bar V_n) ] &\geq \int_\cX V^*(x) \mu_0(x) \, \mathrm dx  - \e \, .
\end{align}
Moreover, using Proposition~\ref{prop:DDM}, \begin{align}\label{eqn:lower}
  \Ee [ F_M(\bar V_n) ] \leq   \sup_{V \in \cH} \Ee [F_M(V)] \leq \sup_{V \in \cC^1(\cX)} \Ee [F_M(V)] = \int_\cX V^*(x) \mu_0(x) \, \mathrm dx
\end{align}  
and combining \eqref{eqn:upper} and \eqref{eqn:lower} gives the result.
\end{proof}

\begin{proof}[Proof of Proposition~\ref{prop:implem}]  First,  notice that~\eqref{eqn:vnx} and~\eqref{eqn:vndx} are direct consequences of~\eqref{eqn:vn} and the reproducing property applied to $\langle V_n, \Phi(x) \rangle_\cH$ and $\langle V_n, \frac{\partial \Phi}{\partial x_j}(x) \rangle_\cH$.

We prove \eqref{eqn:vn} by recursion on $n\geq 1$. Since $V_0=0$, we have $Q(V_0, x^{(1)}) = - f(x^{(1)})  \leq 0 \, $, 
and hence
$ V_1 = \gamma \mu_0(x_1) \Phi(x^{(1)})  $, %
so that \eqref{eqn:vn} holds for $n=1$ with $\alpha_1= \gamma \mu_0(x^{(1)})$ and $\beta_1 = 0$. Let us now assume that \eqref{eqn:vn} holds for $V_{n-1}$. 

Given this representation, we can compute the scalar $V_{n-1}(x^{(n)})$ and the vector $\nabla V_{n-1}(x^{(n)})$ using \eqref{eqn:vnx} and~\eqref{eqn:vndx}, depending only on $x^{(n)}$ and $\alpha_{1:n-1}$ and $\beta_{\cdot, 1:n-1}$. Then, we may compute the following auxiliary variable $s_n:=Q(V_{n-1}, x^{(n)}) $ that depends only on $x^{(n)}$ and $\alpha_{1:n-1}$ and $\beta_{\cdot, 1:n-1}$:
\begin{align*}
    s_n = V_{n-1}(x^{(n)}) - f(x^{(n)}) &- \frac{1}{\rho}\nabla V_{n-1}(x^{(n)})^\top a(x^{(n)}) \\
    & +\frac{1}{2\rho^2}\nabla V_{n-1}(x^{(n)})^\top  B(x^{(n)}) B(x^{(n)})^\top \nabla V_{n-1}(x^{(n)}) \, .
\end{align*}
Finally, since the gradient step is computed as: %
\begin{align*}
  V_n &= V_{n-1} + \gamma \mu_0(x^{(n)}) \Phi(x^{(n)}) \\
  & \qquad \quad ~ - \gamma M \left(\Phi(x^{(n)}) - \frac{1}{\rho}\sum_{i=1}^d a_i(x^{(n)}) \frac{\partial \Phi}{\partial x_i}(x^{(n)}) + 2  \Xi(x^{(n)}) V_{n-1} \right) \times \mathbf{1}_{s_n \geq 0}  \, ,
\end{align*}
then \eqref{eqn:vn} holds at rank $n$, with the following recursion on $\alpha$ and $\beta$:
\begin{align*}
    \alpha_n &= \gamma \mu_0(x^{(n)}) - \gamma M \mathbf{1}_{s_n \geq 0} \,  \\
    \beta_{n, \cdot} &=    \left( \frac{\gamma M}{\rho} a(x^{(n)}) - \frac{\gamma M}{\rho^2}  B(x^{(n)}) B(x^{(n)})^\top \nabla V_{n-1}(x^{(n)}) \right) \mathbf{1}_{s_n \geq 0} \, .
\end{align*}

\end{proof}

\begin{proof}[Proof of Proposition~\ref{prop:boundedgrad}] First, let us prove the convexity of $\Gamma(\cdot, x)$. For any $c, C$, $L$ is a convex function of $(\alpha, \beta)$ as the infimum of a convex function of $(\alpha, \beta, \mu)$ over a convex set. Then $\Gamma(\cdot, x)$ is convex as a composition of $L$ with affine functions of~$w$.

   Let $w \in \cH^m$, $x \in \cX$. We have:
   \begin{align*}
    \frac{\partial}{\partial w} \Gamma(w, x) &= \partial_1 L \left(\|w(x)\|, \mu_0(x) - \frac{1}{\rho}  {\rm div} \left(B(x) w(x) \right) \right)\frac{\partial \|w(x)\|}{\partial w} \\
    & \quad - \frac{1}{\rho} \partial_2 L \left(\|w(x)\|, \mu_0(x) - \frac{1}{\rho}  {\rm div} \left(B(x) w(x) \right) \right) \frac{\partial \left( {\rm div} \left(B(x) w(x) \right) \right)}{\partial w} \, .
\end{align*}
Lemma~\ref{lemma2} provides the expressions of the scalar terms involving $\partial_1 L$ and $\partial_2 L$. The other terms belong to~$\cH^m$ and are computed as follows:
\begin{align*}
    \frac{\partial \|w(x)\|}{ \partial w} &= \frac{w(x)}{ \|w(x)\|}  \otimes \Phi(x)\\
    \frac{\partial \left( {\rm div} \left(B(x) w(x) \right) \right)}{\partial w} &= \sum_{k=1}^d \frac{\partial}{\partial x_k} \left[ B_{k, \cdot} (x) \otimes \Phi(x) \right] \\
    &=  \sum_{k=1}^d  \left[ \frac{\partial B_{k, \cdot}(x)}{\partial x_k}   \otimes \Phi(x) + B_{k, \cdot}(x)\otimes \frac{\partial \Phi(x)}{\partial x_k}   \right] \, ,
\end{align*}
where $\otimes$ denotes the outer product defined, for $v \in \Rr^m$, $h \in \cH$ by: \begin{align*}
    v \otimes h = [v_1 h, \dots, v_m h]^\top \in \cH^m \, .
\end{align*}
Let us now bound the norm of the stochastic gradient. First, we can show that the partial derivatives of $L$ are uniformly bounded. For any $C \geq c \geq 0$, $\alpha \geq 0$ and $\beta \in \Rr$, we have, using Lemma~\ref{lemma2}:
\begin{align*}
   \left| \frac{\partial L_{c, C}}{\partial \alpha} \right| &= \left\{
\begin{array}{cl}
   \left| \left(\frac{q}{q-1}\right)^{\frac{q-1}{q}}  c^{\frac{q-1}{q}}  \right| & \textnormal{if ~} \beta \leq c_0 \alpha \\
    \left| \alpha/\beta  \right|^{q-1}  & \textnormal{otherwise} 
\end{array}
     \right. \\
     & \leq \max \left\{ \left(\frac{q}{q-1}\right)^{\frac{q-1}{q}}  c^{\frac{q-1}{q}}, \left(\frac{1}{c_0}\right)^{q-1}\right\} \,
\end{align*}
where $1/c_0 =  \left(\frac{q-1}{q}\right)^{-1/q}  c^{1/q}$. Similarly:
\begin{align*}
    \left| \frac{\partial L_{c, C}}{\partial \beta} \right| &=\left\{
\begin{array}{cl}
    0 & \textnormal{if ~}\beta \leq c_0 \alpha\\
    \left| c - \frac{q-1}{q} (\alpha/\beta)^q \right| & \textnormal{otherwise } \, .
\end{array}
     \right. \\
     & \leq |c| + \frac{q-1}{q} \left(\frac{1}{c_0}\right)^q  \, .
\end{align*}
In our case, $c = f(x) \leq \| f\|_\infty$, hence we have \begin{align*}
    \left| \frac{\partial L_{c, C}}{\partial \alpha} \right| & \leq \left( \frac{q}{q-1}\right)^{\frac{q-1}{q}} \|f\|_\infty^{\frac{q-1}{q}} =: c_\alpha \\
   \left| \frac{\partial L_{c, C}}{\partial \beta} \right| & \leq 2 \| f\|_\infty =: c_\beta \, .
\end{align*} 
Putting everything together, and noticing that by Rademacher's theorem that $B$ has a bounded gradient almost everywhere, we have, for almost every $x$:
\begin{align*}
\left\|  \frac{\partial \Gamma(w, x)}{\partial w} \right\|_{\cH^m} & \leq c_\alpha \| \Phi(x) \|_\cH + \frac{c_\beta}{\rho}  \| \Phi(x) \|_\cH \left\| \sum_{k=1}^d \frac{\partial B_{k, \cdot(x)}}{\partial x_k}\right\|_2 \\
& \qquad\qquad\qquad + \frac{c_\beta}{\rho}  \sum_{k=1}^d \left\| \frac{\partial \Phi(x)}{\partial x_k} \right\|_\cH  \| B_{k, \cdot}(x)\|_2 \\
& \leq c_\alpha C_K + \frac{d c_\beta C_K C_B'}{\rho} + \frac{d c_\beta C_K' C_B}{\rho} =: D \, .
\end{align*}
\end{proof}

\begin{proof}[Proof of Theorem~\ref{thm:primal}]
    We first need to establish~\eqref{eqn:convG}. The proof is standard and based on classical proof for SGD with bounded gradients (see, \textit{e.g.}, ~\cite[Prop.~11.3]{bach2024learning}. %

Let $w \in \cH^m$, we have:
    \begin{align*}
        \| w_n - w \|^2_{\cH^m} &= \| w_{n-1} - w \|^2_{\cH^m} - 2 \gamma \sum_{i=1}^m \left\langle [w_{n-1}]_i-w_i, \frac{\partial \Gamma(w_{n-1}, x_n)}{\partial w_i} \right\rangle_\cH \\
        &\qquad\qquad\qquad\qquad + \gamma^2 \left\|  \frac{\partial \Gamma(w_{n-1}, x_n) }{\partial w} \right\|_{\cH^m}^2 \\
        & \leq \| w_{n-1} - w \|^2_{\cH^m} - 2  \gamma \left( \Gamma(w_{n-1}, x_n) - \Gamma(w, x_n) \right) + \gamma^2 D^2 \, ,
    \end{align*}
    where we have used the convexity of $\Gamma$.
    Taking expectations, we obtain:
    \begin{align*}
        \Ee \left[ \| w_n - w \|^2_{\cH^m} \right] & \leq \Ee \left[\| w_{n-1} - w \|^2_{\cH^m}\right] - 2  \gamma \left( \Ee [G(w_{n-1})] - G(w) \right) + \gamma^2 D^2 \, .
    \end{align*}
    Reorganizing terms, summing over $k \in \{0, ..., n-1\}$ with a telescopic sum, and dividing by $n$, we obtain:
    \begin{align*}
       \frac{1}{n} \sum_{k=0}^{n-1} \Ee [G(w_{k})] - G(w) & \leq \frac{\|w\|^2_{\cH^m}}{2 \gamma n} + \frac{\gamma D^2}{2 } \, .
    \end{align*}
    By convexity of $G$, we get the first result:
    \begin{align*}
        \Ee [G(\bar w_{n})] \leq G(w) + \frac{\|w\|^2_{\cH^m}}{2 \gamma n} +  \frac{\gamma D^2}{2 } \, .
    \end{align*}

     We now prove the asymptotic convergence.
    For $\e>0$, there exists $\hat{w} \in (H^1(\cX))^m$ such that
    \begin{equation*}
        G(\hat{w})
        \leq 
        \inf_{w\in (H^1(\cX))^m}
        G(w) +\e/4.
    \end{equation*}
    It is easy to check that $G$ is continuous on $(H^1(\cX))^m$
    with respect to the $H^1$ topology.
    Recall that $\cH$ is dense in $H^1(\cX)$,
    there exists $\tilde w \in \cH^m$ such that:
     \begin{align} \label{eqn:dens}
         |G(\tilde w) - G(\hat{w})| \leq \e/4 \, .
     \end{align}
     Using~\eqref{eqn:convG} with $\gamma = \gamma_0/\sqrt{n}$, for some constant $\gamma_0>0$, and then~\eqref{eqn:dens}, we have:
     \begin{align*}
         \Ee [G(\bar w_n)]  & \leq G(\tilde w) + \frac{\| \tilde w\|_{\cH^m}^2}{2  \gamma_0 \sqrt{n}}  + \frac{\gamma_0 D^2}{2 \sqrt{n}} \\
         & \leq \inf_{w\in (H^1(\cX))^m}G(w) + \e/2 +  \frac{\| \tilde w\|_{\cH^m}^2}{2  \gamma_0 \sqrt{n}}  + \frac{\gamma_0 D^2}{2 \sqrt{n}} \, .
\end{align*}
Let $n_0 =  \frac{4}{ \e^2} \max \{ \frac{\|\tilde w\|^4_{\cH^m}}{\gamma_0^2}, \gamma_0^2 D^4 \}$. As soon as $n \geq n_0$, we obtain
\begin{align*}
    \inf_{w\in (H^1(\cX))^m}G(w)
    \leq \Ee [ G(\bar w_n) ] 
    &\leq \inf_{w\in (H^1(\cX))^m}G(w)+ \e \, .
\end{align*}

\end{proof}

\begin{proof}[Proof of Lemma~\ref{lemmapolicyeval}]
    For any fixed controller $u_c$, the autonomous system $\dot x = a(x) + B(x)  u_c(x)$ has the same dynamics as the controlled system $\dot {\tilde x} = \tilde a(x) + \tilde B(x) v(x)$, for any controller~$v$, with $\tilde a(x) := a(x) + B(x) u_c(x)$ and $\tilde B(x)=0$. We define the cost functions $\tilde f(x) := f(x) + R( u_c(x))$ and $\tilde R(v) := R(v)$.  Since the dynamics of this system does not depend on $v$, and $R$ is positive, the optimal choice of $v$ is $0$, and the optimal value function is equal to $\int  \rho e^{-\rho t } \tilde f(\tilde x_t) \mathrm  dt$, where $\tilde x_t = x_t$ for all $t$, so that we have  $\tilde V^* \equiv V^{u_c}$. To prove the first part of the Lemma, we need to apply Theorem~\ref{thm_duality} to this new system. Since $\tilde B = 0$, Assumption~\ref{hypo:alpha} is not fulfilled, but we can replace it by any non-negative upper-bound on the optimal controller (since $v^* \equiv 0$).  Concerning Assumption~\ref{hypo:regularity},  it is fulfilled as soon as $u_c$ is uniformly Lipschitz-continuous.   Applying Theorem~\ref{thm_duality}, we then obtain that the primal and dual formulations of the control problem with tildas, respectively equivalent to $(\textnormal{P}(u_c))$  and $(\textnormal{D}(u_c))$, have a common value equal to $\int \tilde V^* \mathrm d\mu_0 = \textnormal{Perf}(u_c)$, hence the result. To prove the second part, we apply the same exact penalty reasoning as in Section~\ref{sec:exactpen}, after remarking that~$V^{u_c}$ is uniformly upper-bounded by~$C_V$, because
    \begin{align*}
        \forall x \in \cX, \quad V^{u_c}(x) \leq \rho \int_{0}^{+\infty} e^{-\rho t} \left( \|f\|_\infty + \sup_{x' \in \cX }\frac{\|u_c(x')\|_2^q}{q} \right) \mathrm dt \leq C_V \, .
    \end{align*}
\end{proof}

\begin{proof}[Proof of Proposition~\ref{prop:subopt}]
Let $u_c:= w/ \mu$. Then the primal objective function at $(w, \mu)$ is such that:
\begin{align*}
J(w,\mu)  &= 
\int_{\X}  \Big[ f(x) +  R (u_c(x)) \Big] \mu(x) \, \mathrm dx \\
&\qquad \qquad  
+ C_V \int_\cX  \left[\mu_0(x) - \mu(x) - \frac{1}{\rho} {\rm div} \left(
a(x) \mu(x) + B(x) u_c(x) \mu(x) \right) \right]_+ \mathrm dx \\
& \geq ~ \inf_{\nu:\cX \to \Rr_+} ~~\int_{\X}  \Big[ f(x) +  R (u_c(x)) \Big] \nu(x) \, \mathrm dx \\
&\qquad \qquad  + C_V \int_\cX  \left[\mu_0(x) - \nu(x) - \frac{1}{\rho} {\rm div} \left(
a(x) \nu(x) + B(x) u_c(x) \nu(x) \right) \right]_+ \mathrm dx \\ 
& = \text{val}(\text{P-P}(u_c)) \\
& = \text{Perf}(u_c)  \, ,
\end{align*}
where the last equality results from the application of Lemma~\ref{lemmapolicyeval} on the candidate controller $u_c$. Furthermore, the pair $(w^*, \mu^*)$ being optimal for problem~\eqref{eqn:Ppen}, we have $J(w^*, \mu^*) = \text{Perf}(u^*)$  and the result follows. 
\end{proof}

\begin{proof}[Proof of Proposition~\ref{prop:projection}]
    Using the reproducing property, we have:
\begin{align*}
    \Pi(V, w, x) &=\frac{1}{\rho}  \left(B(x)^\top \nabla V(x) \right) \otimes \Phi(x) \\
    &\qquad\qquad + \partial_1 S_M \left(w(x), f(x) -   V(x)+ \frac{1}{\rho} \nabla V(x)^\top a(x) \right) \otimes \Phi(x) \, ,
\end{align*}
where $\partial_1 S_M$ is the partial gradient of $S_M$ with respect to the first variable $\alpha$. Let us define similarly $\partial_2 S_M$ as the (scalar) derivative of $S_M$ with respect to the second variable~$\beta$.
\begin{align*}
    &\Delta(V, w, x) = - \mu_0(x) \Phi(x) - \frac{1}{\rho} \sum_{j=1}^d [B(x) w(x)]_j \frac{\partial \Phi}{\partial x_j}(x) \\
    &\quad - \partial_2 S_M \left(w(x), f(x) -   V(x)+ \frac{1}{\rho} \nabla V(x)^\top a(x) \right) \left( - \Phi(x) + \frac{1}{\rho} \sum_{j=1}^d a_j(x) \frac{\partial \Phi}{\partial x_j}(x) \right) .
\end{align*}
Since $V$ has bounded $\cH$-norm, we can bound $V(x)$ and $\nabla V(x)$ as follows:
\begin{align*}
    |V(x)| & = | \langle V, \Phi(x) \rangle_\cH | \leq \|V\|_\cH \| \Phi(x) \|_\cH \leq r_V C_K \\
    \|\nabla V(x)\|_2 & = \sqrt{ \sum_{i=1}^d \left(\langle V, \frac{\partial \Phi}{\partial x_i}(x) \rangle_\cH \right)^2 } \leq \sqrt{d} r_V C'_K  \, ,
\end{align*}
where $C_K' = \max_i \sup_{x \in \cX} \| \partial_i \Phi(x)\|_\cH$. 
Similarly, we can bound the norm of $w(x)$ by:
\begin{align*}
    \| w(x) \|_2 & \leq r_w C_K \, .
\end{align*}
Moreover, the partial derivatives of $S_M$ remain bounded if $\alpha$ and $\beta$ are both bounded:
\begin{align*}
    \left| \frac{\partial S_M}{\partial \beta} (\alpha, \beta) \right|&= \mu^*(\alpha, \beta) \in [0, M]  \\
    \| \nabla_\alpha S_M (\alpha, \beta) \|_2  &\leq \max\left\{\left( \frac{q}{q-1}\right)^{\frac{q-1}{q}} \beta^{\frac{q-1}{q}} \mathbf{1}_{\beta \geq 0} \, , \,\left(\frac{\|\alpha\|_2}{M}\right)^{q-1} \right\} \, ,
\end{align*}
where, in our case, $\alpha$ and $\beta$ are such that:
\begin{align*}
   \| \alpha \|_2 =  \| w(x) \|_2 &\leq C_K r_w \\
   | \beta  | = \left| f(x) -   V(x) + \frac{1}{\rho} \nabla V(x)^\top a(x) \right| &\leq \|f\|_\infty +  C_K r_V + \frac{\sqrt{d}  C_K'  C_a r_V}{\rho} \, .
\end{align*}
We can finally bound the gradients as follows:
\begin{align*}
\| \Delta(V, w, x) \|_{\cH} & \leq \| \mu_0 \|_\infty C_K + \frac{d C_B C_K C_K' }{\rho} r_w + M C_K + \frac{M d C_a C_K'}{\rho} \\
\| \Pi(V, w, x) \|_{\cH^m} & \leq \frac{\sqrt{d}}{\rho}  C_K C_K' C_B r_V + C_K \left\| \partial_1 S_M (\alpha, \beta) \right\|_2 \\
& \leq \frac{\sqrt{d}}{\rho}  C_K C_K' C_B r_V + C_K \underbrace{\left(\frac{ C_K r_w}{M}\right)^{q-1}}_{\leq 1+ (q-1) \frac{C_K r_w}{M}} \\
& \qquad~ + C_K \left( \frac{q}{q-1}\right)^{\frac{q-1}{q}} \underbrace{\left[\|f\|_\infty + C_K r_V + \frac{\sqrt{d}  C_K'  C_a r_V}{\rho} \right]^{\frac{q-1}{q}}}_{\leq 1+ \frac{q-1}{q} \|f\|_\infty +\frac{q-1}{q} (C_K + \sqrt{d}C_K' C_a / \rho) r_V} \,
\end{align*}
and the result follows from gathering terms.
\end{proof}

\begin{proof}[Proof of Theorem~\ref{thm3}] The upper-bound on the error metric $\cE_r$ is provided by~\cite[Theorem~1]{juditsky2011solving}. In particular, Assumption (4) from~\cite{juditsky2011solving} holds with $L=0$, $M=2 \sqrt{2} c_\Omega(1+2r)$, Assumption (6) from~\cite{juditsky2011solving} holds with $\mu =0$ and $\sigma^2 = 8 c_\Omega^2 (1+2r)^2$ (remarking that $M^2=\sigma^2$), and Assumption (16) from~\cite{juditsky2011solving} holds with $\Omega=\max_{z \in Z_r} \|z\|^2=2r^2$. Theorem~1 from~\cite{juditsky2011solving} hence reads, for any $\gamma >0$ and $r \geq 1$:
\begin{align*}
   \Ee[\cE_r(\bar V_n, \bar w_n) ] &\leq \frac{\Omega^2}{n \gamma} + \frac{21 \gamma M^2}{2}  \\
   & \leq \frac{4 r^4}{n \gamma} + 756 c_\Omega^2  r^2 \gamma \, .
   \end{align*}
    In particular, we can balance both terms (up to constant factors) if we set $\gamma =\frac{r}{14 c_\Omega\sqrt{n}}$, and we obtain:
    \begin{align} \label{eqn:rate}
        \Ee[\cE_r( \bar V_n, \bar w_n) ] \leq 110 c_\Omega \frac{r^3}{\sqrt{n}} \, .
    \end{align}
    The radius $r \geq 1$ constrains the approximation to lie in $Z_r$. To achieve global approximation, we need to take $r \to + \infty$ as $n \to +\infty$, but not too fast, so that~\eqref{eqn:rate} still converges to 0. Therefore, we set $r \propto n^\alpha$, for $\alpha \in (0, 1/6)$, which is slowly growing to infinity. %

\end{proof}

\begin{proof}[Proof of Proposition~\ref{prop:gap}]
    Set $q=2$ and let $\tilde V \in \cH$. We have:
    \begin{align*}
        &\inf_{w : \cX \to \Rr^m} \Ee_{x \sim \cU(\cX)} \Psi(\tilde V, w, x)  = \int_\cX \tilde V \, \mathrm d \mu_0 \\
        & \qquad\qquad + \int_{\cX}\inf_{\mu \in [0, M]} \inf_{w \in \Rr^m}  \left\{ \frac{\|w\|^2}{2 \mu} +\frac{1}{\rho} \nabla \tilde V(x)^\top B(x) w   \right. \\ & \qquad\qquad\qquad\qquad\qquad\qquad\left. + \frac{1}{\rho} \mu \nabla \tilde V(x)^\top a(x) + \mu f(x) - \mu \tilde V(x)\right \}  \mathrm dx \\
        & = \int_\cX \tilde V \, \mathrm d \mu_0 + \int_{\cX}\inf_{\mu \in [0, M]} \left\{ - \frac{\mu}{2 \rho^2} \|B(x)^\top \nabla \tilde V(x) \|_2^2+ \frac{1}{\rho} \mu \nabla \tilde V(x)^\top a(x) \right. \\
       & \qquad\qquad\qquad\qquad\qquad\qquad\left. + \mu f(x) - \mu \tilde V(x) \right\} \mathrm dx \\
       & =  \int_\cX \tilde V \, \mathrm d \mu_0 - M \int_\cX \left[ - f(x) +  \tilde V(x) - \frac{1}{\rho}  \nabla \tilde V(x)^\top a(x) + \frac{1}{2 \rho^2} \|B(x)^\top \nabla \tilde V(x) \|_2^2 \right]_+ \\
       &= F_M(\tilde V) \, .
    \end{align*}
    Now, set $q \in (1, 2]$ and let some $\tilde w \in \cH^m$. We have, using strong duality:
    \begin{align*}
        &\sup_{V:\cX \to \Rr} \Ee_{x \sim \cU(\cX)} \Psi(V, \tilde  w, x) = \inf_{\mu : \cX \to [0, M]} \int_\cX \left(  f(x) + R\left( \frac{\tilde w(x)}{\mu(x)}\right)  \right) \mu(x) \, \mathrm dx \\
        & \qquad\qquad + \int_\cX \sup_{V \in \Rr} ~ \left\{  \mu_0(x)- \mu(x) - \frac{1}{\rho} \text{div}(a(x) \mu(x)+B(x) \tilde w(x))  \right\} V \, \mathrm dx \\
        & \geq \inf_{\mu : \cX \to [0, M]} \int_\cX \left(  f(x) + R\left( \frac{\tilde w(x)}{\mu(x)}\right)  \right) \mu(x) \, \mathrm dx \\
        & \qquad\qquad + \int_\cX \sup_{V \in [0, C_V]} ~ \left\{  \mu_0(x)- \mu(x) - \frac{1}{\rho} \text{div}(a(x) \mu(x)+B(x) \tilde w(x))  \right\} V \, \mathrm dx \\
        &= \inf_{\mu : \cX \to [0, M]} \int_\cX \left(  f(x) + R\left( \frac{\tilde w(x)}{\mu(x)}\right)  \right) \mu(x) \, \mathrm dx \\
        & \qquad\qquad + C_V \int_\cX \left[   \mu_0(x)- \mu(x) - \frac{1}{\rho} \text{div}(a(x) \mu(x)+B(x) \tilde w(x))  \right]_+ \mathrm dx \\
        & = \inf_{\mu : \cX \to [0, M]} J(\tilde w, \mu) \, .
    \end{align*}
    If we assume that $a(x) = 0$ for all $x \in \cX$, then the problem becomes separable in~$x$ for~$\mu$ and $\inf_{\mu : \cX \to [0, M]} J(\tilde w, \mu) = G(\tilde w)$.
\end{proof}

\begin{proof}[Proof of Lemma \ref{lem:bound_gradient}]
For simplicity, we will assume that
    $a,B,f$ are $C^{\infty}$.
    In fact, those extra regularity assumptions will only be
    used to deal with solution of PDEs that are satisfied 
    in the strong sense, \textit{i.e.}, at every point.
    To recover the result in the general case, \textit{i.e.}, with less regularity,
    it is sufficient to add mollifiers to those functions
    and let the smoothing parameter tend to $0$:
    the estimates hold since they do not depend on
    the additional regularity assumptions.

    For $\delta,\varepsilon>0$ 
    let us introduce the regularized HJB equation
    \begin{equation}
    \label{eq:HJB_reg}
        V-\varepsilon\Delta V
        +\frac1{\rho}a(x)^{\top}\nabla_xV
        +\frac1{q'}\sigma\left(-\frac1{\rho}B(x)^{\top}\nabla_xV\right)^{q'}
        =
        f(x),
    \end{equation}
    where $\sigma(r)=(\delta^2+\|r\|^2)^{\frac12}$.
    Using standard argument from elliptic regularity,
    $V$ is $\cC^4$, so that the latter PDE can be understood pointwisely
    and can be differentiated in a direction $\xi\in\Rr^d$ with $\|\xi\|=1$
    and get %
\begin{multline*}
        \partial_{\xi}V-\varepsilon\Delta \partial_{\xi}V
        +\frac1{\rho}(\partial_{\xi}B\alpha+B\partial_{\xi}\alpha)^{\top}\nabla_xV
        +\frac1{\rho}a^{\top}\nabla_x\partial_{\xi}V
        \\
        +\rho^{-2} s^{q'-2}\nabla_xV^{\top}B(\partial_{\xi}B^{\top}\nabla_xV+B\nabla_x\partial_{\xi} V)
        =
        \partial_{\xi}f \, 
\end{multline*}
with $s := \sigma\left(-\frac1{\rho}B(x)^{\top}\nabla_xV\right)$.   
Consider $(x,\xi)$ such that $\partial_{\xi}V(x)$ is maximal,
so that $\Delta\partial_{\xi}V\leq 0$ and $\nabla_x\partial_{\xi}V=0$,
we obtain
\begin{align*}
    \sup_{(x',\xi')}\partial_{\xi'}V(x')
    =\partial_{\xi}V(x)
    &\leq
    \rho^{-1}(\|\alpha\|_{\infty}\|\partial_{\xi}B^\top \nabla_xV\|
    +\|\partial_{\xi}\alpha\|_{\infty}\|B^\top \nabla_xV\|)
    \\
    &\hspace*{2cm}
    +\rho^{-2}s^{q'-2}\|B^{\top}\nabla_xV\|\|\partial_{\xi}B^\top \nabla_xV\|
    +\|\partial_{\xi}f\|_{\infty}
    \\
    &\leq
    \rho^{-1}(C_{\partial B}\|\alpha\|_{\infty}
    +\|\partial_{\xi}\alpha\|_{\infty})
    \|B^\top \nabla_xV\|_{\infty}
    \\
    &\hspace*{2cm}
    +\rho^{-2}C_{\partial B} s^{q'-2}\|B^{\top}\nabla_xV\|_{\infty}^2
    +\|\partial_{\xi}f\|_{\infty},
\end{align*}
where we used 
$\partial_{\xi}B\partial_{\xi}B^{\top}\leq C_{\partial B}^2 BB^{\top}$
from Assumption~\ref{hypo:C2}.
Then, we let~$\delta$ and~$\e$ tend to zero,
so that 
$\|\rho^{-1}B^{\top}\nabla_xV\|_{\infty}$
tends to 
$\|\rho^{-1}B^{\top}\nabla_xV^*\|_{\infty}$
which is upper bounded by $C_u^{q-1}$ as proven in
the proof of Proposition \ref{prop:bound}.
In particular, this implies
\begin{equation*}
    \sup_{(x',\xi')}\partial_{\xi'}V^*(x')
    \leq
    C_{\partial V^*}
    =
    C_u^{q-1}C_{\partial B}\|\alpha\|_{\infty}
    +C_u^{q-1}\|\partial_{\xi}\alpha\|_{\infty}
    +C_u^{q}C_{\partial B}
    +\|\partial_{\xi}f\|_{\infty}.
\end{equation*}
We prove that 
$\sup_{(x',\xi')}-\partial_{\xi'}V^*(x')\leq C_{\partial V^*}$
using similar arguments.
To conclude the proof,
it is sufficient to take $C_{\nabla  V^*}=\sqrt{d}\,C_{\partial V^*}$.
\end{proof}

\begin{proof}[Proof of Lemma \ref{lem:SC}]
    We make the same simplification assumptions as in the proof
    of Lemma \ref{lem:bound_gradient}.
    Observe that, if $V$ is semi-concave,
    its semi-concavity constant is upper bounded by
    $\sup_{x\in\cX,\|\xi\|=1}v(x,\xi)$
    where $v$ is defined by
    $v(x,\xi)=\partial^2_{\xi,\xi}V(x)$.
    Therefore, it is sufficient to prove that $v$
    is uniformly bounded under the present assumption 
    to conclude the proof.
    To do so, we start by taking the second derivative 
    of the regularized HJB equation \eqref{eq:HJB_reg}
    in the direction $\xi\in\Rr^d$ with $\|\xi\|=1$
    and obtain
    \begin{align*}
        v
        -\varepsilon\Delta_x v
        =\;&
        -\frac1{\rho}\left(\partial^2_{\xi,\xi}\alpha^{\top}B^{\top}\nabla_xV
        +2\partial_{\xi}\alpha^{\top}\partial_{\xi}
        \left(B^{\top}\nabla_xV\right)
        +\alpha^{\top}\partial^2_{\xi,\xi}(B^{\top}\nabla_xV)\right)
        \\
        &-\frac{s^{q'-2}}{\rho^2}
        \left(\left\|\partial_{\xi}(B^{\top}\nabla_xV)\right\|^2
        +\nabla_xV^{\top}B\partial^2_{\xi,\xi}(B^{\top}\nabla_xV)\right)
        \\
        &-\frac{(q'-2)s^{q'-4}}{\rho^4}
        \left(\nabla_xV^{\top}B\partial_{\xi}(B^{\top}\nabla_xV)\right)^2
        +\partial^2_{\xi,\xi}f
        \\
        =\;&
        -\nabla_xv^{\top}
        \left(\frac1{\rho}B\alpha+\frac{s^{q'-2}}{\rho^2}BB^{\top}\nabla_xV\right)
        -\frac{s^{q'-2}}{\rho^2}\|B^{\top}\nabla_x\partial_{\xi}V\|^2
        \\
        &-\frac{2}{\rho}
        \left(\frac{s^{q'-2}}{\rho}\nabla_xV^{\top}
        \left(\partial_{\xi}BB^{\top}+B\partial_{\xi}B^{\top}\right)
        +\partial_{\xi}a^{\top}\right)\nabla_x\partial_{\xi}V
        \\
        &-\frac1{\rho}\partial^2_{\xi,\xi}a^{\top}\nabla_xV
        -\frac{s^{q'-2}}{\rho^2}
        \left(\nabla_xV^{\top}B\partial^2_{\xi,\xi}B^{\top}\nabla_xV
        +\|\partial_{\xi}B^{\top}\nabla_xV\|^2
        \right)
        \\
        &-\frac{(q'-2)s^{q'-4}}{\rho^4}
        \left(\nabla_xV^{\top}B\partial_{\xi}(B^{\top}\nabla_xV)\right)^2
        +\partial^2_{\xi,\xi}f \, .
    \end{align*}
    Then,  we have 
    \begin{equation*}
        -\rho^{-1}\partial^2_{\xi,\xi}a^{\top}\nabla_xV
        \leq
        C_1
        :=
        \rho^{-1}\|\partial^2_{\xi,\xi}a\|_{\infty}\|\nabla_xV\|_{\infty},
    \end{equation*}
    and,
    using $C_r$ defined by
    $C_r = \|\rho^{-1}B^{\top}\nabla_xV\|_{\infty}$
    and Assumption~\ref{hypo:C2},
    \begin{equation*}
        \frac{s^{q'-2}}{\rho^2}
        \nabla_xV^{\top}B\partial_{\xi,\xi}^2B^{\top}\nabla_xV
        \leq
        C_2
        :=
        \rho^{-1}(\delta^2+C_r^2)^{\frac{q'-2}2}C_r
        \|\partial_{\xi,\xi}B\|_{\infty}\|\nabla_xV\|_{\infty},
    \end{equation*}   
    using $s\leq\sqrt{\delta^2+C_r^2}$.
    To deal with the remaining terms,
    we only consider a couple $(x,\xi)$ that maximizes the quantity $v$.
    For such a choice, we get $\nabla_xv(x)=0$, $\Delta_x v(x)\leq0$,
    and $\nabla_x\partial_{\xi}V(x)=D^2_{x,x}V(x)\xi
    =v\xi$ since $\xi$ is an eigenvector of $D^2_{x,x}V(x)$
    associated to the eigenvalue $v$.
    Consequently, we obtain
    \begin{equation}
    \label{eq:aux_SC}
    \begin{aligned}
        v(x,\xi)
        \leq\;& 
        -\frac{s^{q'-2}}{\rho^2}\|B^{\top}\xi\|^2v^2
        +C_1+C_2+\|\partial^2_{\xi,\xi}f\|_{\infty}
        \\
        &-\frac{2v}{\rho}
        \xi^{\top}\left(\frac{s^{q'-2}}{\rho}
        \left(\partial_{\xi}BB^{\top}+B\partial_{\xi}B^{\top}\right)
        \nabla_xV
        +\partial_{\xi}a\right),
    \end{aligned}
    \end{equation}
    It only remains to deal with the terms depending on $v$.
    Let us consider two cases.

    \noindent
    \textbf{First case:} 
    Assume $q=2$, we have $s^{q'-2}=1$ and then,
     \begin{align*}
        -\frac{2v}{\rho}\xi^{\top}\left(\rho^{-1}
        \left(\partial_{\xi}BB^{\top}+B\partial_{\xi}B^{\top}\right)
        \nabla_xV
        +\partial_{\xi}a\right)
        &\leq
        \frac{2v}{\rho}\|B^{\top}\xi\|
        \left(2 C_{\partial B} C_r+\|\partial_{\xi}\alpha\|_{\infty}
         \right. \\
         &  \left.  \qquad\qquad\qquad\qquad\quad + C_{\partial B}\|\alpha\|_{\infty}\right)
        \\
        &\leq
        \frac{v^2}{\rho^2}\|B^{\top}\xi\|
        +C_3,
    \end{align*}
    with $C_3=
        \left(2 C_{\partial B}C_r+\|\partial_{\xi}\alpha\|_{\infty}
        + C_{\partial B}\|\alpha\|_{\infty}\right)^2$. 
    The latter inequality and \eqref{eq:aux_SC} imply
    \begin{equation*}
        \sup_{x',\xi'}v(x',\xi')
        =
        v(x,\xi)
        \leq
        C_1+C_2+C_3+\|\partial^2_{\xi,\xi}f\|_{\infty}.
    \end{equation*}
    This proves that $V$ is semi-concave,
    so is $V^*$ by letting $\delta$ and $\varepsilon$ tend to zero.
    To obtain the desired semi-concavity constant,
    recall that $\lim_{\delta,\varepsilon\to0}C_r=C_u^{q-1}$
    as proven in the proof of Proposition \ref{prop:bound}.

    \noindent

    \noindent
    \textbf{Second case:} consider $q<2$ and $\rho>\rho_1=4\sup_{x\in\cX,\|\xi\|=1}-\xi^{\top}D_{x}a(x)\xi$. We obtain
    \begin{align*}
        -\frac{2v}{\rho}\xi^{\top}\left(\frac{s^{q'-2}}{\rho}
        \left(\partial_{\xi}BB^{\top}+B\partial_{\xi}B^{\top}\right)
        \nabla_xV
        +\partial_{\xi}a\right)
        &\leq
        \frac{4s^{q'-2}v}{\rho}\|B^{\top}\xi\|C_{\partial B}C_r
        +\frac{\rho_1}{2\rho}v
        \\
        &\leq
        \frac{s^{q'-2}}{\rho^2}\|B^{\top}\xi\|v^2
        +C_4
        +\frac{v}2,
    \end{align*}
    with $C_4=4(\delta^2+C_r^2)^{\frac{q'-2}2}{C_{\partial B}}^2C_r^2$.
    Using the latter inequality in \eqref{eq:aux_SC},
    then subtracting $\frac{v}2$ on both side of resulting inequality,
    we obtain
    \begin{align*}
        \frac12v(x,\xi)
        =
        \frac12
        \sup_{x',\xi'}v(x',\xi')
        \leq
        C_1+C_2+C_4+\|\partial^2_{\xi,\xi}f\|_{\infty}.
    \end{align*}
    To conclude, it is sufficient to multiply by $2$
    and let tend $\delta$ and $\varepsilon$ to zero, like in the first case.

    In both cases, we obtain the following
    upper bound on the semi-concavity constant
    \begin{multline}
        \label{eq:SC_constant}
        C_{\rm SC}
        =
        \frac{2}{\rho}\left(\|D^2_{x,x}a\|_{\infty}
        +C_u\|D^2_{x,x}B\|_{\infty}C_{\nabla V^*}\right)
        \\
        +2(2C_{\partial B}C_u^{q-1}
        +\|D_x\alpha\|_{\infty}
        +C_{\partial B}\|\alpha\|_{\infty})^2
        +2\|D^2_{x,x}f\|_{\infty}.
    \end{multline}
    \end{proof}

\begin{proof}[Proof of Proposition \ref{prop:bound_mu*}]
    Let $(t,x)\mapsto\Phi_t(x)$ be the flow of the dynamic
    under the optimal control, \textit{i.e.},
    it satisfies
    \begin{align*}
        &\partial_t\Phi_t(x)
        =
        a(\Phi_t(x))+B(\Phi_t(x))u^*(\Phi_t(x))
        \\
        &\text{ with }
        \hspace*{0.2cm}
        u^*(y)
        =
        -\rho^{-(q'-1)}\|B(y)^{\top}\nabla_xV^*(y)\|^{q'-2}B(y)^{\top}\nabla_xV^*(y).
    \end{align*}
    Standard regularity results from \cite{cannarsa2004semiconcave} state that
    there exists a set $E\subset\cX$ such that $\lambda_{\cX}(\cX\backslash E)=0$
    (where $\lambda_{\cX}$ is the Lebesgue measure on the torus)
    and, for all $x\in E$, the value function $V^*$ is $\cC^2$ on 
    an open neighborhood of the optimal trajectory starting from $x$.
    This implies that, almost everywhere, we can define the function
    $\ell:(t,x)\mapsto-\log (\det D_x\Phi_t(x))$.
    Consider $\widetilde{\mu}_t$ as the density of the law of the state at time $t$
    with the initial state distributed according to $\mu_0$.
    It satisfies, for almost all $x\in\cX$,
    \begin{align*}
        \mu_0(x)
        =
        \widetilde{\mu}_t(\Phi_t(x))
        |\det D_x\Phi_t(x)|
        \;\text{ so that }\;
        \widetilde{\mu}_t(\Phi_t(x))
        =
        \mu_0(x)e^{\ell(t,x)} \, .
    \end{align*}
    Using the uniqueness of the solution of Liouville's discounted equation
    and Lemma~\ref{lemma1}, $\mu^*$ is the $\rho$-discounted occupation and
    satisfies
    \begin{equation}
    \label{eq:carac_mu}
        \mu^*(x)
        =
        \int_0^{+\infty}\widetilde{\mu}_t(x)\rho e^{-\rho t}\, \mathrm dt
        =
        \rho\int_0^{+\infty}\mu_0(\Phi_t^{-1}(x))e^{\ell(t,\Phi_t^{-1}(x))} e^{-\rho t}\, \mathrm dt.
    \end{equation}
    Recall that $\mu_0$ is uniformly bounded, it only remains
    to get a convenient upper bound on $\ell$ to conclude on the result.
    This is the purpose of the remainder of this proof.

    For $x\in E$, we can differentiate the PDE satisfied by $\Phi_t(x)$
    and get
    \begin{align*}
        \partial_tD_x\Phi_t(x)
        =
        (D_xa+D_xBu^*+BD_xu^*)(\Phi_t(x))D_x\Phi_t(x) \, ,
    \end{align*}
    where $D_xBu^*$ is a $\Rr^{d\times d}$-matrix such that
    $(D_xBu^*)_{i,j}=\sum_{k=1}^m\partial_{x_j}B_{i,k}u_k$ for $1\leq i,j\leq d$.
    This and Jacobi's formula yield%
    \begin{align*}
        \partial_t\det(D_x\Phi_t(x))
        &=
        \det(D_x\Phi_t(x))
        \tr((D_x\Phi_t(x))^{-1}\partial_t D_x \Phi_t(x)))
        \\
        &=
        \det(D_x\Phi_t(x))
        \tr(D_xa+D_xBu^*+BD_xu^*) (\Phi_t(x)) \, .
    \end{align*}
    This, in turn, implies that
    \begin{equation*}
        \partial_t\ell(t,x)
        =
        -\tr(D_xa+D_xBu^*+BD_xu^*)(\Phi_t(x)) \, .
    \end{equation*}
    Then, observe that
    \begin{align*}
        \tr((D_xBu^*)(\Phi_t(x))
        &\leq
        \|u(\Phi_t(x)\|\sum_{i=1}^d\|\partial_{x_i}B^{\top}(\Phi_t(x))e_i\|
        \\
        &\leq
        \|z(t)\|^{q'-1}
        \sum_{i=1}^dC_{\partial B}\|B^{\top}(\Phi_t(x))e_i\|
        \\
        &\leq
        C_{\partial B}\sqrt{d}\|B\|_{F,\infty}
        \|z(t)\|^{q'-1},
    \end{align*}
    with $z(t)$ defined by $z(t)=-\rho^{-1}(B^{\top}\nabla_xV^*)(\Phi_t(x))$, 
    and
    \begin{align*}
        D_xu^*
        &=
        -\rho^{-(q'-1)}
        \Bigl(\|B^{\top}\nabla_xV^*\|^{q'-2}
        (D_xB^{\top}\nabla_xV^*+B^{\top}D^2_{x,x}V^*)
        \\
        & ~
        +(q'-2)\|B^{\top}\nabla_xV^*\|^{q'-4}
        B^{\top}\nabla_xV^*(y)\nabla_x{V^*}^{\top}B
        (B^{\top}D^2_{x,x}V^*+D_xB^{\top}\nabla_xV^*)
        \Bigr).
    \end{align*}
    Lemma \ref{lem:SC} implies that $V^*$ is $\overline{C}$-semi-concave, if $q=2$, or if $q<2$ and $\rho>\rho_1:=4\sup\{-\xi^{\top}D_xa(x)\xi,x\in\cX,\|\xi\|=1\}=0$ if $a \equiv 0$.  %
    Then $D^2_{x,x}V^*\leq \overline{C}I_d$ almost everywhere, and we obtain
    \begin{align*}
        -\tr(BD_xu^*)
        &=
        \frac{1}{\rho}\|z\|^{q'-4}\tr\left((\|z\|^2I_m+(q'-2) zz^{\top})
        (B^{\top}D^2_{x,x}V^*B+BD_xB^{\top}\nabla V^*)\right)
        \\
        &\leq
        \frac{1}{\rho}(q'-1)\|z\|^{q'-2}
        (\overline{C}+C_{\partial B}\|\nabla_xV^*\|_{\infty})
        \|B\|^2_{F,\infty} \, . 
    \end{align*}
    Consequently,
    for $C_{\ell,1}=C_{\partial B}\sqrt{d}\|B\|_{F,\infty}$
    and $C_{\ell,2}=(\overline{C}+C_{\partial B}\|\nabla_xV^*\|_{\infty})\|B\|^2_{F,\infty}$,
    we obtain
    \begin{equation}
        \label{eq:ineq_ell}
        \partial_t\ell(t,x)
        \leq
        \sup_{x\in\cX}-{\rm div}_x a(x)
        +C_{\ell,1}\|z(t,x)\|^{q'-1}
        +\rho^{-1}C_{\ell,2}\|z(t,x)\|^{q'-2}.
    \end{equation}

    Let us now make out two cases, whether or not we assume $a\equiv0$.%

    \noindent
    \textbf{First case:} Assume $a\not\equiv 0$.
    From the proof of Proposition \ref{prop:bound},
    we have that $\|z\|_{\infty}\leq C_u^{q-1}$, so that
    we obtain
    \begin{align*}
        \ell(t,x)
        &=
        \ell(0,x)
        +\int_0^t\partial_t\ell(s,x)\, \mathrm ds
        \\
        &\leq
        \left(\|{\rm div}_xa\|_{\infty}
        +C_{\ell,1}C_u^{(q-1)(q'-1)}
        +\rho^{-1}C_{\ell,2}C_u^{(q-1)(q'-2)}\right) t
        \\
        &=
        \left(\|{\rm div}_xa\|_{\infty}
        +C_{\ell,1}C_u
        +\rho^{-1}C_{\ell,2}C_u^{2-q}\right)t
        =
        C_{\mu}(\rho)t \, ,
    \end{align*}
    where we used $\ell(0,x)=0$ and Inequality \eqref{eq:ineq_ell}.
    Observe that $C_{\mu}(\rho)$ is decreasing with respect to $\rho$,
    so that there exists $\rho_0>0$ such that $C_{\mu}(\rho)<\rho$ for all
    $\rho>\rho_0$.
    This and Equality \eqref{eq:carac_mu} imply that, for $\rho>\rho_0$,
    we obtain %
    \begin{equation*}
        \mu^*(x)
        \leq
        \rho\|\mu_0\|_{+\infty}\int_0^{\infty}e^{-(\rho-C_{\mu})t}\, \mathrm dt
        =
        \frac{\rho\|\mu_0\|_{\infty}}{\rho-C_{\mu}}.
    \end{equation*}

    \noindent
    \textbf{Second case:} Assume $a\equiv 0$ and $q<2$.  
    Observe that $t\mapsto V^*(\Phi_t(x))$ is a Lyapunov function 
    for almost all $x\in\cX$, \textit{i.e.},
    \begin{align*}
        \partial_tV^*(\Phi_t(x))
        &=
        \nabla_x{V^*}^{\top}(Bu^*)(\Phi_t(x))
        \\
        &=
        -\rho^{-(q'-1)}\|(B^{\top}\nabla_xV)(\Phi_t(x))\|^{q'}
        =
        -\rho\|z(t)\|^{q'}
        \leq0,
    \end{align*}
    where we recall that 
    $z(t)=-\rho^{-1}(B^{\top}\nabla_xV)(\Phi_t(x))$. %
    Therefore, for $t\geq0$, we get
    \begin{equation*}
        \int_0^{t}\|z(s)\|^{q'}\,\mathrm ds
        \leq
        \rho^{-1}(V^*(x)-V^*(\Phi_t(x)))
        \leq 
        \rho^{-1}\|f\|_{\infty},
    \end{equation*}
    where we used $0\leq V^*\leq \|f\|_{\infty}$.
    Then, using Holdër's inequality, we have
    \begin{align*}
        \int_0^{t}\|z(s)\|^{q'-1}\, \mathrm ds
        &\leq
        \left(\int_0^{t}\|z(s)\|^{q'}\, \mathrm ds\right)^{\frac{q'-1}{q'}}t^{\frac1{q'}}
        \leq
        \rho^{-\frac1q}\|f\|_{\infty}^{\frac1q}t^{\frac1{q'}},
        \\
        \int_0^{t}\|z(s)\|^{q'-2}\, \mathrm ds
        &\leq
        \left(\int_0^{t}\|z(s)\|^{q'}\, \mathrm ds\right)^{\frac{q'-2}{q'}}t^{\frac2{q'}}
        \leq
        \rho^{-(\frac2q-1)}\|f\|_{\infty}^{\frac2q-1}t^{\frac2{q'}}.
    \end{align*}
    Using the latter two inequality and integrating Inequality \eqref{eq:ineq_ell}
    as above, we get
    \begin{align*}
        \mu^*(x)
        \leq
        \rho\|\mu_0\|_{\infty}
        \int_0^{+\infty}e^{-\rho t 
        +C_{\ell,1}
        \rho^{-\frac1q}\|f\|_{\infty}^{\frac1q}t^{\frac1{q'}}
        +C_{\ell,2}
        \rho^{-\frac2q}\|f\|_{\infty}^{\frac2q-1}t^{\frac2{q'}}}\, \mathrm dt
        <\infty \, .
    \end{align*}
    This concludes the proof.
\end{proof}

\end{document}